\documentclass[hidelinks,onefignum,onetabnum]{siamart251216}

\usepackage{lipsum}
\usepackage{amsfonts}
\usepackage{amssymb}
\usepackage{amsmath}
\usepackage{graphicx}
\usepackage{epstopdf}
\usepackage{algorithm}
\usepackage{algpseudocode}
\usepackage{subfig}

\ifpdf
\DeclareGraphicsExtensions{.eps,.pdf,.png,.jpg}
\else
\DeclareGraphicsExtensions{.eps}
\fi

\newsiamremark{remark}{Remark}
\newsiamremark{hypothesis}{Hypothesis}
\crefname{hypothesis}{Hypothesis}{Hypotheses}
\newsiamthm{claim}{Claim}
\newsiamremark{fact}{Fact}
\crefname{fact}{Fact}{Facts}

\renewcommand{\O}{\Omega}

\def\S{\mathrm{S}}

\newcommand\R{\mathbb{R}}

\newcommand{\x}{{\mathbf{x}}}

\DeclareSymbolFont{rsfs}{U}{rsfs}{m}{n}
\DeclareSymbolFontAlphabet{\mathscrsfs}{rsfs}
\newcommand{\mscr}[1]{\mathscrsfs{#1}}

\def\H{{X}}

\def\0{\boldsymbol{0}}
\def\div{{\rm div} \:}

\def\y{{y}}

\def\vh{{v}_h}

\def\e{{e}}
\def\z{{z}}

\def\ey{e^{{y}}}
\def\ez{e^{{z}}}

\def\v{{v}}
\def\u{{u}}

\def\f{{f}}

\def\n{{\boldsymbol{n}}}

\def\a{{\boldsymbol{{a}}}}
\def\w{{w}}

\def\U{{U}}

\def\CF{\mathcal{F}}

\def\EVT{\mathcal{E}_{V,T}}
\def\EVTm{\mscr{E}_{V,T}}
\def\EUT{\mathcal{E}_{U,T}}

\def\uliney{\hat{{y}}}
\def\ulineyh{\hat{{y}}_{h}}
\def\ulinez{\hat{{z}}}
\def\ulinev{\hat{{v}}}
\def\ulinevh{\hat{{v}}_{h}}
\def\ulinew{\hat{{w}}}

\def\ulineu{\hat{u}}
\def\ulineuh{\hat{{u}}_h}
\def\ulinee{\hat{e}}
\def\ulineeh{\hat{{e}}_{h}}

\def\CMh{\mathcal{M}_{h}}
\def\ACTh{\tilde{\mathcal{T}}_{h}}
\def\CTh{\mathcal{T}_{{h}}}
\def\CFh{\mathcal{F}_{{h}}}

\def\IVTk{\hat{\mathcal{I}}_{V,T}^k}
\def\IUTk{\hat{\mathcal{I}}_{U,T}^k}
\def\IVTkm{\hat{\mscr{I}}_{V,T}^k}

\def\RT{\mathcal{R}_T}
\def\GT{\mathcal{G}_T}
\def\RTc{{\mathcal{R}_T^c}}

\def\IVhk{\hat{\mathcal{I}}_{V,h}^k}
\def\IUhk{\hat{\mathcal{I}}_{U,h}^k}
\def\IVhkm{\hat{\mscr{I}}_{V,h}^k}

\def\Lhl{\mathcal{L}_h^l}

\def\LTki{\mathcal{L}_T^{k+1}}
\def\LTl{\mathcal{L}_T^l}
\def\LTkpt{\mathcal{L}_T^{k+2}}
\def\Lhkpt{\mathcal{L}_h^{k+2}}
\def\pT{\partial T}

\def\dn{\partial_{n}}
\def\dnn{\partial_{{nn}}}
\def\dnt{\partial_{{nt}}}
\def\dt{\partial_{{t}}}
\def\dtt{\partial_{{tt}}}

\def\PiFk{\Pi_F^k}
\def\PiTkpt{\Pi_T^{k+2}}

\def\res{\mathrm{res}}
\def\sta{\mathrm{sta}}
\def\nor{\mathrm{nor}}
\def\tan{\mathrm{tan}}

\renewcommand{\P}{{\mathbb P}}  

\def\V{V}
\def\A{\mathcal{A}}

\def\LT{L^2(\Omega)}
\def\Linf{L^{\infty}(\Omega)}
\def\L{\boldsymbol{L}}

\def\adel{{a^{\Delta}}}
\def\anab{{a^{\nabla}}}
\def\VTk{{\hat{{V}}}_T^k}
\def\UTk{{\hat{{U}}_T^k}}
\def\Vhk{{\hat{{V}}}_{{h}}^k}
\def\Uhk{{\hat{{U}}}_{{h}}^k}
\def\Vhok{{\hat{{V}}}_{{h},0}^k}
\def\Zhok{\hat{{X}}_{h,0}^k}
\def\Uhok{{\hat{{U}}}_{{h},0}^k}
\def\Usfhok{{U}_{h,0}^k}

\headers{HHO methods for sixth-order problems}{A. Khan, A. Kumar and H. Singh}

\title{
	Mixed Hybrid High-Order Methods for Sixth-Order Problems: Fully Non-conforming and $C^0$-Conforming Discretizations\thanks{Submitted to the editors DATE.
		\funding{A. Khan was supported by ANRF ARG-Matrics
			ANRF/ARGM/2025/001949/MTR. A. Kumar and H. Singh were financially supported by CSIR India, with Award no. 09/0143(21346)/2025-EMR-I and 09/0143(16048)/2022-EMR-I, respectively. }
	}
}

\author{ 
	Arbaz Khan\thanks{Department of Mathematics, IIT Roorkee, India.
		E-mails: arbaz@ma.iitr.ac.in (A. Khan), ajay\_k@ma.iitr.ac.in (A. Kumar), harpal\_s@ma.iitr.ac.in (H. Singh).}
	\and
	Ajay Kumar\footnotemark[2]
	\and 
	Harpal Singh\footnotemark[2]
}
\ifpdf
\hypersetup{pdftitle={New \LaTeX\ Style} ,pdfauthor={A. Khan, A. Kumar and H. Singh}}
\fi	

\usepackage{amsopn}

\ifpdf
\hypersetup{
	pdftitle={An Example Article},
	pdfauthor={}
}
\fi

\begin{document}
	
	\maketitle
	
	\begin{abstract}
We consider a class of sixth-order elliptic partial differential equations in two and three dimensions subject to simply supported and Cahn--Hilliard-type boundary conditions. Based on the Ciarlet--Raviart reformulation, we recast the sixth-order problem as an equivalent mixed system involving second- and fourth-order equations and establish its well-posedness under suitable assumptions. For the resulting mixed formulation, we propose and analyse two discretization frameworks: a fully non-conforming hybrid high-order (HHO) method on general polytopal meshes and a $C^0$-conforming HHO--finite element method on simplicial meshes. We prove stability and derive optimal-order a priori error estimates for the primary variables under appropriate regularity assumptions. For the fully non-conforming HHO method, we further derive a reliable residual-based a posteriori error estimator. Numerical experiments confirm the theoretical convergence rates and illustrate the robustness of the proposed methods for a range of mobility parameters.
	\end{abstract}
	\section{Introduction}
	Partial differential equations (PDEs) are widely used to describe various phenomena in science and engineering. Higher-order PDEs, in particular, arise in the modelling of complex physical processes that involve higher-order spatial derivatives, such as fourth- and sixth-order terms. Such equations are encountered in applications including thin film dynamics \cite{MR2028722, MR3327707}, fluid flows \cite{MR2667016}, geometric and surface modelling \cite{liu2007general}, phase field crystal modelling \cite{MR3003062}, and fracture modelling \cite{MR2807583}. In recent years, a wide range of numerical methods have been developed to analyze such systems, with several notable contributions closely aligned with the present study \cite{ MR2028722, MR4496380, dassi2022virtual, MR4842145, MR4604440, MR4982456, MR5047454, MR3564350, MR2519603, MR2361532}.
	
	For a given source function $\f\in\LT$, the sixth-order problem is to find the unknown $\y: \Omega \rightarrow \R$ such that
	\begin{align}
		\label{modelproblem}
		\y - \div\!(\kappa \nabla \Delta^2 \y) &= \f \qquad \text{in} \ \O,
	\end{align}
	with the \textit{simply supported boundary conditions}:
		\begin{align}
			\label{boundcndn1} 
			\y = \Delta\y = \Delta^2\y &=0 \qquad \text{on} \ \partial\Omega.
		\end{align}
		or the \textit{Cahn-Hilliard boundary conditions}:
		\begin{align}
			\label{boundcndn2} 
			\partial_{n} \y = \partial_{n} \Delta \y = \partial_{n} \Delta^2 \y &=0 \qquad \text{on} \ \partial\Omega.
		\end{align}
		
	In \eqref{modelproblem}, $\kappa>0$ is a positive physical parameter which represents the mobility of the concentration field in phase field crystal equations and micro-emulsion systems and $\text{div}$ represents the usual divergence operator.

	Classical finite element methods (FEMs) face significant challenges in discretizing higher-order PDEs because they require discrete spaces with high global continuity. One alternative is to use non-conforming discretizations, as in \cite{MR1930132}, but this lack of conformity introduces technical difficulties in error analysis. A simpler and well established approach is to use mixed schemes, which reformulate the original problem as a coupled system and reduce the regularity requirements. The authors in \cite{MR2028722, MR4974657} introduced two additional unknowns and the sixth-order problem was decomposed into three coupled second-order problems where the work in \cite{MR4974657} is restricted to polygonal domains with the largest interior angle no more than $\pi/2$. However, this approach introduces several challenges, including higher computational costs compared to the primal formulation and difficulties in developing stable mixed schemes. The authors in \cite{MR4570555, MR4604440} developed an unconditional energy stable and solvable $C^0$ interior penalty method for the phase field crystal equation. Recently, a $C^{1}$-$C^{0}$-conforming virtual element method (VEM) which allows general polygonal meshes, was proposed in \cite{MR4982456} for the sixth-order problems with clamped and simply supported boundary conditions. The authors introduced an additional unknown using a Ciarlet-Raviart formulation and reformulated the original problem as a coupled system of a fourth-order and a second-order problem. Optimal order a priori error estimates were derived for the main and auxiliary variable under standard mesh assumptions.
	
	The hybrid high-order (HHO) methods were firstly introduced in \cite{ MR3259024} for diffusion problem and a locking-free version in \cite{MR3283758} for the linear elasticity problem. HHO methods use local reconstruction and a local stabilization operator on each mesh cell to approximate the bilinear form. The first operator reconstructs differential operators using cell and face unknowns motivated by an integration by part formula while second operator penalizes the trace of cell and face unknowns on each mesh interface. HHO methods allow static condensation, which significantly reduces computational cost by eliminating local degrees of freedom. Owing to their flexibility and high-order accuracy, these methods naturally support general polygonal (and polyhedral) meshes, making them well suited for complex geometries. As a result, HHO methods have been successfully extended to a wide range of problems with some notable contributions including \cite{MR3504993, MR4439881, MR3784906, MR3767823, MR4621059}, \cite{MR4779761, MR3834430, MR4355427, MR4485999, MR4699573} for fourth-order problems and \cite{MR4630545, MR4891729, dong2025hperroranalysismixedorderhybrid, dong2026hp, MR4621059} focus on a posteriori error analysis.
	
	 In \cite{MR4355427}, the authors developed an HHO method for singularly perturbed fourth-order problems. The method employs a Nitsche-type penalty approach to impose boundary conditions weakly, and incorporates a carefully scaled stabilization parameter that balances the singular perturbation parameter with the square of the local mesh size. Later, in \cite{MR4485999} two different HHO methods were devised for biharmonic problem using hybrid spaces with the cell unknowns approximated by polynomials of order $k + 2$, the face unknowns approximating the trace of the solution on the mesh interfaces are polynomials of order $k+1$ for the first HHO method, and order (k + 2) for the second HHO method, and face unknowns approximating the trace of normal derivative are of order $k$. Recently, a $C^0$-conforming HHO method for the biharmonic problem accommodating both simply supported and clamped boundary conditions, was introduced in \cite{MR4699573}. This approach combines the advantages of conformity with the flexibility of HHO frameworks. In addition, a posteriori error estimates for the biharmonic problem have been developed in \cite{dong2026hp}, providing reliable tool for error control and adaptive mesh refinement. While the HHO method offers several advantages and has been successfully applied to a variety of problems, its application to sixth-order problems is still an open and largely unexplored area.
	
	In the present work, we focus on the development and analysis of two distinct HHO schemes for sixth-order problems. We use the Ciarlet-Raviart formulation inspired from \cite{MR1930132} and introduce a new auxiliary unknown to reduce the sixth-order problem to a coupled fourth-order and a second-order problem. The coupled problem is shown to be well-posed by an application of Lax-Milgram lemma. For the numerical approximation, we consider HHO(B) method devised in \cite{MR4485999} alongside a $C^0$-conforming HHO method introduced in \cite{MR4699573}. The main novelties and contributions of this work can be summarized as follows.
	\begin{itemize}
		\item \textit{Two novel discretization frameworks for sixth-order problems:}
We propose two new and fundamentally different discretization frameworks for the numerical approximation of sixth-order elliptic problems. The first is based on a fully non-conforming HHO formulation for both the fourth- and second-order components of the mixed problem. The second combines a $C^0$-conforming HHO method for the fourth-order component with a standard conforming finite element method for the second-order component. These constructions employ carefully designed mixed-order hybrid spaces, allowing the proposed methods to achieve high-order accuracy while reducing the number of globally coupled degrees of freedom. In particular, the $C^0$-conforming HHO--FEM formulation provides an efficient and implementation-friendly alternative to a fully non-conforming discretization of the sixth-order problem.
		\item \textit{A priori and a posteriori error analysis:} 
		We establish optimal $O(h^{k+1})$ a priori convergence rates for both proposed schemes in the corresponding energy norms. For the fully non-conforming HHO method, we additionally derive a reliable residual-based a posteriori error estimator by separating conforming and non-conforming error components and employing suitable modified Babu\v{s}ka--Suri and Karkulik--Melenk interpolation operators. The resulting framework provides an efficient and computable error control mechanism.
\end{itemize}
	\vspace{3mm}
	
	The remainder of the paper is organized as follows: Sec.~\ref{Continuous settings and well-posedness} introduces the notations, functional framework and establishes the well-posedness of model problem. Secs.~\ref{HHO Scheme} and~\ref{C0 HHO Scheme} present two HHO schemes along with the a priori error estimates. Sec.~\ref{A posteriori error estimates} is dedicated to the development of a posteriori error estimates. Finally, Sec.~\ref{Numerical experiments} provides some numerical results that confirm the theoretical findings.
	\section{Continuous settings and well-posedness}\label{Continuous settings and well-posedness}
	In this section, we introduce the notations and function spaces and then discuss the well-posedness results.
	\subsection{Preliminaries}\label{Preliminaries}
	Let $\Omega \subset \mathbb{R}^d$ $(d =2,3)$ be an open and bounded convex domain with polygonal boundary $ \partial\Omega$. For any open and bounded set $\mathcal{S} \subset \mathbb{R}^d$ or $\mathcal{S} \subset \mathbb{R}^{d-1}$, let $W^{r,p}(\mathcal{S})$ denotes usual Sobolev spaces with corresponding norms $\|\cdot\|_{r,p,\mathcal{S}}$ and semi-norms $|\cdot|_{r,p,\mathcal{S}}$
	where $0\le r< \infty$ and $1\le p \le \infty$. For $p=2$, we define $W^{r,2}(\mathcal{S}):= H^{r}(\mathcal{S})$ with norm $\|\cdot\|_{H^{r}(\mathcal{S})}$ and semi-norm $|\cdot|_{H^{r}(\mathcal{S})}$. If $r = 0$, $H^{0}(\mathcal{S}):= L^{2}(\mathcal{S})$ with standard $L^2$ inner-product and norm defined as:
	\begin{align*}
		(\v_1,\v_2)_{\mathcal{S}} := \int_{\mathcal{S}} \v_1 \v_2 \, dx , \qquad \text{and} \qquad \|\v\|_{\mathcal{S}}^{2} := (\v,\v)_{\mathcal{S}},
	\end{align*}
	respectively. To simplify the notations, we define 
	\[\V :=H_0^1(\O)\cap H^2(\O) \qquad \text{and} \qquad
	\U := H_0^1(\O),
	\]
	with corresponding norms 
	\[
	\|\cdot\|_{\V} := |\cdot|_{H^2(\O)} \qquad \text{and} \qquad
	\|\cdot\|_{\U} := |\cdot|_{H^1(\Omega)},
	\]
	 respectively. {By $\dn$, $\dt$, $\dnn$, $\dnt$ and $\dtt$, we denote the scalar-valued normal derivative, the $\R^d$-valued tangential derivative, the scalar-valued normal-normal second-order derivative, $\R^d$-valued normal-tangential second-order derivative  and the $\R^{d\times d}$ valued tangential-tangential second-order derivative, respectively.} We drop the domain in subscript of norm and inner-product if it is $\Omega$ itself. The relation $m\lesssim n$ states that $m\leq Cn$ with a positive constant $C$ independent of $m$, $n$, or the mesh size. 
	\subsection{Mixed formulation}
	Firstly, we derive a variational formulation for the mixed version of \eqref{modelproblem}-\eqref{boundcndn1} by introducing a new unknown $\z := \Delta^2 \y$ using the Ciarlet-Raviart method (see \cite[Sec.~7.1]{MR1930132}).

	  Then, \eqref{modelproblem} along with the boundary conditions \eqref{boundcndn1} yield the following system:
	  \begin{subequations}
	\begin{align}
		\label{4thorder} 
		\Delta^2 \y - \z&=0 && \text{in} \ \O,&&
		 \y - \div\!(\kappa \nabla \z)= \f && \text{in} \ \O,\\
		\label{SSBC}	
		\y = \dnn\y &= 0 && \text{on} \ \partial \O,&&
			\z =0 && \text{on} \ \partial \O.
	\end{align}
	\end{subequations}
%

	We multiply the fourth order problem with a test function $\v \in \V$ and the second order problem by $\u \in \U$ and then use integration by parts to obtain the variational formulation: find $(\y,\z)\in \V\times\U$ such that
	\begin{subequations}\label{wf12}
	\begin{align}
		\label{WF1}
		\adel(\y,\v) - (\z,\v) &=0 &&\hspace{-1cm} \forall \v \in \V, \\
		\label{WF2} (\y,\u) + \anab(\z,\u) &= (\f,\u) &&\hspace{-1cm} \forall \u \in \U,
	\end{align}
	\end{subequations}
	where the bilinear forms $\adel: \V \times \V \rightarrow \R$ and $\anab: \U \times \U \rightarrow \R$ are defined as:
	\begin{align}\label{bilnear_forms}
		\adel(\y,\v) &:= (\nabla^2\y,\nabla^2\v),\qquad \qquad
		\anab(\z,\u) := (\kappa \nabla \z,\nabla \u).
	\end{align}
	\subsection{Well-posedness}\label{Well-posedness}
	{{To discuss the well-posedness of the mixed formulation \eqref{wf12}, consider the product space $\H := \V\times\U$ which is equipped with the norm 
			\begin{align*}
				\|(\v,\u)\|_{\H}^{2} := \left(\|\v\|_{\V}^{2} +\|\u\|_{\U}^2 \right).
			\end{align*}
		
		 Consider the variational problem: find $(\y,\z) \in \H$ such that
		\begin{align}
			\label{sum} \A((\y,\z),(\v,\u))&=(\f,\u) \qquad \forall (\v,\u)\in \H,
		\end{align}
		where the bilinear form $\A((\cdot,\cdot),(\cdot,\cdot)): \H \times \H \rightarrow \R$ is defined as:
		\begin{equation}\label{cal_bil_a}
			\A((\y,\z),(\v,\u)) := \adel(\y,\v) - (\z,\v) + (\y,\u) + \anab(\z,\u) \qquad \forall (\y,\z), (\v,\u) \in \H.
		\end{equation}}}
		\begin{lemma}\label{Lemma 2.1}
			(i) Let $\adel(\cdot,\cdot)$ and $\anab(\cdot,\cdot)$ be the bilinear forms defined in \eqref{bilnear_forms}. Then, the following estimates hold true:
			\begin{subequations}
			\begin{align}
			\label{cont1}	|\adel(\y,\v)| &\lesssim \|\y\|_{\V} \|\v\|_{\V}, &&  \|\v\|_{\V}^{2} =\adel(\v,\v) ,\\
			\label{cont2}	|\anab(\z,\u)| &\lesssim \|\z\|_{\U} \|\u\|_{\U}, &&   \|\u\|_{\U}^{2} \lesssim \anab(\u,\u) ,
			\end{align}
			\end{subequations}
			
			(ii) The bilinear form $\A((\cdot,\cdot),(\cdot,\cdot))$ satisfies the following properties:
			\begin{subequations}
			\begin{align}
			\label{a_cont}	|\A((\y, \z), (\v, \u))| &\lesssim \|(\y, \z)\|_{\H} \|(\v, \u)\|_{\H} &&\forall (\y,\z), (\v, \u) \in \H,\\
			\label{a_coerc}	 \|(\v, \u)\|_{\H}^{2}&\lesssim \A((\v, \u), (\v, \u))  && \forall (\v, \u) \in \H.
			\end{align}
			\end{subequations}
		\end{lemma}
		\begin{proof}
			(i) The estimates \eqref{cont1} and \eqref{cont2} are obtained by an application of the Cauchy-Schwarz inequality and the generalized Poincar\'e inequality.
			
			(ii) The estimate \eqref{a_cont} directly follows by Cauchy-Schwarz inequality and the estimates \eqref{cont1} and \eqref{cont2}. Next, for all $(\v,\u) \in \H$, we have
			\begin{align*}
				\|(\v,\u)\|_{\H}^{2} \lesssim \adel(\v,\v)+ \anab(\u,\u) &= \adel(\v,\v) - (\u,\v) + (\v,\u) + \anab(\u,\u)\\
				& = \A((\v, \u), (\v, \u)).
			\end{align*}
		\end{proof}
		
		By an application of Lax-Milgram lemma along with the coercivity property \eqref{a_coerc}, there exists a unique solution $(\y,\z) \in \H$ of \eqref{sum} which satisfies the following stability estimate:
		\[
		 \|\y\|_{\V} + \|\z\|_{\U} \lesssim \|\f\|.
		\]

		Furthermore, by the regularity results discussed in \cite[Theorem~2]{MR4982456}, $(\y,\z) \in (\V \cap H^{3}(\O)) \times (\U \cap H^{2}(\O))$ and
		\[
		 {\|\y\|_{H^3} + \|\z\|_{H^2}} \lesssim \|\f\|.
		\]
	\begin{remark}
		In the case of Cahn-Hilliard type boundary conditions \eqref{boundcndn2}, the mixed variational formulation is: find $(y,z) \in \tilde{\H} := \tilde{\V} \times \tilde{\U}$ such that
		\begin{align}\label{CHsum}
			\A((\y,\z),(\v,\u))&=(\f,\u) \qquad \forall (\v,\u)\in \tilde{\H}.
		\end{align}
		where
		\[
		\tilde{\V}:= \{\v\in H^2(\O)|\, (\v,1)=0,\ \dn\v=0\ \text{on}\  \partial\O\}, \qquad 
		\tilde{\U}:= \{\u\in H^1(\O)|\, (\u,1)=0\},
		\]
		and $\tilde{\H}$ is equipped with norm $\|(\v,\u)\|_{\tilde{\H}}^{2} := \left(\|\nabla^2\v\|^2 + \|\nabla\u\|^2 \right)$. The well-posedness of \eqref{CHsum} is established by employing arguments analogous to those used in proving the solvability of \eqref{sum}.
		\end{remark}
	\begin{remark}
		In the model problem \eqref{modelproblem}, we may consider the mobility parameter $\kappa$ as a positive and bounded scalar function which belongs to $\Linf$ such that
		\begin{align*}
			0<\kappa_{1} \le \kappa(\x) \le \kappa_{2} \qquad \forall \x \in \O.
		\end{align*}
		
		To simplify the presentation, we assume that $\kappa$ is a positive constant throughout the domain $\Omega$. {Nevertheless, all the results presented in this work can be extended to the case of a spatially varying piecewise constant mobility parameter $\kappa$ satisfying the above bounds, with only slight modifications in the present analysis}.
	\end{remark}
	\section{HHO method discretization}\label{HHO Scheme}\setcounter{equation}{0}
	In this section, we present the hybrid high-order method 
	to discretize the mixed problem \eqref{sum}. Let's introduce some terminology and basic notations which are helpful in establishing the discrete problem. 
	\subsection{Preliminaries}\label{Preliminaries2}
	Let $\mathcal{M}_h=\{\CTh,\CFh\}$ denotes a polytopal mesh on $\O$, which belongs to a regular mesh sequence $\{\mathcal{M}_h\}_{h>0}$ in the sense of \cite[Sec.~1.1.2]{MR4230986}. Let $T\in\CTh$ {be a generic mesh element with boundary $\partial T$ (which can be split as $\pT = \pT^i\cup \pT^b$ where $\pT^i = \pT \cap \Omega$ and $\pT^b = \pT \cap \partial \Omega$), unit outward normal $\n_T$, inradius $r_T$ and diameter $h_T$ such that the global mesh size $ h=\text{max}_{T\in\CTh} h_T$. Let $F\in\CFh$  be a generic mesh face and $n_F$ denotes the unit outward normal to $F$.} Also, $\CFh=\CFh^i \cup \CFh^b$, where $\CFh^i$ is the set containing interfaces and $\CFh^b$ contains boundary faces. For any mesh element $T\in\CTh$, $\CF_{\partial T}:=\{F\in\CFh|\, F\subset\partial T\}$ denotes the set of faces contained in $\partial T$. The shape regularity implies that for all $T\in \CTh$ and $F\in\CF_{\partial T}$, the diameters $h_T$ and $h_F$ are uniformly equivalent and $\text{card}(\CF_{\partial T})$ is uniformly bounded. The mesh vertices are collected in the set $\mathcal{V}_h$. 
	
	For an integer $l\ge0$, let $\P^l(T) $ denote the set of polynomials up to degree $l$ on $T$. The broken Sobolev and polynomial spaces are defined as:
	\begin{align*}
		H^{s}(\CTh)&:=\{\v \in L^2(\Omega)|\, \v|_{T}\in H^{s}(T)\quad\forall T\in\CTh\},\\
		\P^l(\CTh)&:=\{\vh \in L^1(\Omega)|\, \vh|_{T} \in \P^l(T) \quad \forall T \in \CTh\}.
	\end{align*}
	
	For all $\v \in H^s(\CTh)$, $s>1/2$, its jump across any mesh interface $F \in \CF_h^i$ is defined as $[\![\v]\!] := \v|_{T_1}|_F - \v|_{T_2}|_F$, where $\n_F$ is oriented from $T_1$ to $T_2$. On every boundary face $F \in \CF_h^b$, we set $[\![\v]\!] := \v_T|_F$.
	
	Let $\Pi_T^l$ and $\Pi_{\pT}^l$ be the $L^2$-orthogonal projections onto $\P^{l}(T)$ and $\P^l(\CF_{\pT})=\times_{F\in\CF_{\pT}}\P^l(F)$, respectively. 
	Now, we discuss some useful inequalities and results that are helpful for further analysis in which $\rho$ is the mesh regularity parameter, $k\ge0$ is the polynomial degree, and $d$ denotes space dimension.
	\begin{lemma}\label{disinvtrineq}
		\begin{enumerate}
			\item {Discrete inverse and trace inequalities.} There exists a constant $C>0$ depending on  $\rho$, $k$ and {$d$} such that for all $T\in\CTh$
			\begin{subequations}
			\begin{align}
				\|\vh\|_{\partial T}&\le Ch_T^{-\frac{1}{2}}\|\vh\|_T && \forall \vh\in\P^k(T), \label{invtr}\\
				\|\nabla\vh\|_{T}&\le Ch_T^{-1}\|\vh\|_T &&  \forall \vh\in\P^k(T), \label{inv}\\
				\|\dt \vh\|_F&\le Ch_T^{-1}\|\vh\|_F &&  \forall \vh\in\P^k(T), \forall F\in\CF_{\pT}.\label{faceinv}
			\end{align} 
			\end{subequations}
			\item {Multiplicative trace inequality.}
			There exists a positive constant $C_{m}$ depending on $\rho$ and {$d$} such that for all $T\in\CTh$
			\begin{align}\label{mt}
				\|\v\|_{\partial T}\le C_m\left( h_T^{-\frac{1}{2}}\|\v\|_T+\|\v\|_T^{\frac{1}{2}}\|\nabla\v\|_T^{\frac{1}{2}}\right) \quad \forall \v\in H^1(T).
			\end{align}  
			\item {Poincar\'e inequality.} 
			There exists a constant $C_{P} >0$ depending on $\rho$ and $d$ such that for all $\v\in H^2(T)^\perp:=\{\v\in H^2(T)\ |\ (\v,\xi)_T=0 \ \forall \xi\in\P^1(T)\} $,
			\begin{align}\label{pcineq}
				h_T^{-2}\|\v\|_{ T}+ h_T^{-1}\|\nabla\v\|_{ T}\le C_P\|\nabla^2\v\|_T \qquad \forall T\in\CTh.
			\end{align} 
		\end{enumerate}
	\end{lemma}
	The proof for Lemma~\ref{disinvtrineq} follows from \cite[Sec.~1.4]{MR2882148} and \cite{MR1726480}, respectively.
	
	\subsection{Local discrete spaces and bilinear forms}
	For $T\in\CTh$, we define the local HHO spaces as:
	\[
	\VTk:=\P^{k+2}(T) \times \P^{k+2}(\CF_{\partial T}) \times \P^{k}(\CF_{\partial T}), \qquad \qquad
	\UTk:=\P^{k+1}(T) \times \P^{k}(\CF_{\partial T}). 
	\]
	{A generic element $\ulinev_T \in \VTk$ is of the form $\ulinev_T=(\v_T,\v_{\partial T},\eta_{\partial T})$, where $\v_T \in \P^{k+2}(T)$ (represents the solution inside the mesh cell), $\v_{\partial T} \in \P^{k+2}(\CF_{\partial T})$ (represents the trace on the cell boundary) and $\eta_{\partial T} \in \P^{k}(\CF_{\partial T})$ (represents the normal derivative on the cell boundary along the the direction of the outward normal $\n_T$).} Any arbitrary element $\ulineu_T\in \UTk$ is of the form $\ulineu_T=(\u_T,\u_{\partial T})$, where $\u_T \in \P^{k+1}(T)$ and $\u_{\partial T} \in \P^{k}(\CF_{\partial T})$ have similar definitions. The local discrete semi-norms on $\VTk$ and $\UTk$ are defined as:
   \begin{subequations}
	\begin{align}
		|\ulinev_T|^2_{\VTk} &= \|\nabla^2\v_T\|^2_T+ h_T^{-1}\|\dn\v_T-\eta_{\pT}\|_{\pT} + h_T^{-3}\|\v_T-\v_{\pT}\|_{\pT},\\
		|\ulineu_T|^2_{\UTk} & = \|\nabla\u_T\|^2_T + h_T^{-1}\|\u_T-\u_{\pT}\|^2_{\pT}.
	\end{align}
	\end{subequations}
	To approximate the differential operators, we define the local hessian reconstruction operator {$\RT:\VTk\to\P^{k+2}(T)$} as:
	\begin{align}
		\label{R1} (\nabla^2 \RT(\ulinev_T),\nabla^2 \w)_{T} 
		&=(\nabla ^{2}\v_T,\nabla^{2} \w)_{T} +(\v_T-\v_{\pT},\dn\Delta \w)_{\pT} \\
		\notag & \quad \  -(\dn\v_T-\eta_{\pT},\dnn \w)_{\pT} \\
		\nonumber &\quad \ -(\dt(\v_T-\v_{\pT}),\dnt \w)_{\pT} && \forall\w\in\P^{k+2}(T),\\
		(\RT(\ulinev_T),\xi)_T&=(\v_T,\xi)_T &&\forall\xi\in\P^1(T),\notag
	\end{align}
	and the gradient reconstruction operator $\GT: \UTk\to\P^{k+1}(T)$ as: 
	\begin{align}
		\label{G1} (\nabla \GT(\ulineu_T),\nabla \w)_{T} 
		&=(\nabla \u_T,\nabla \w)_{T} -  (\u_T-\u_{\pT},\dn \w)_{\pT}\quad\forall\w\in\P^{k+1}(T), \\
		(\GT(\ulineu_T),1)_T&=(\u_T,1)_T.\notag
	\end{align}
	
	Next, in order to define the element-wise contribution, for all $\uliney_T,\ulinev_T \in \VTk$ such that $\uliney_T = (\y_T, \y_{\pT}, \gamma_{\pT}), \ \ulinev_T = (\v_T,\v_{\partial T},\eta_{\partial T})$ and $\ulinez_T,\ulineu_T \in \UTk$ such that $\ulinez_T = (\z_T, \z_{\pT}),\  \ulineu_T = (\u_T,\u_{\partial T})$, we define the following local bilinear forms:
	\begin{subequations}
	\begin{align}
		a^{\Delta}_T : \VTk \times \VTk \to \R, \ \  a^{\Delta}_T(\uliney_T,\ulinev_T) &= (\nabla^2 \RT(\uliney_T), \nabla^2 \RT(\ulinev_T))_T + S^{\Delta}_{\pT}(\uliney_T,\ulinev_T),\\
		a^{\nabla}_T : \UTk \times \UTk \to \R, \ \  a^{\nabla}_T(\ulinez_T,\ulineu_T) &= (\nabla \GT(\ulinez_T), \nabla \GT(\ulineu_T))_T  + S^{\nabla}_{\pT}(\ulinez_T,\ulineu_T),
	\end{align}
	\end{subequations}
	where the stabilization terms $S_{\pT}^{\Delta}(\cdot,\cdot)$ and $S_{\pT}^{\nabla}(\cdot,\cdot)$ are defined by following \cite[Sec.~5.1]{MR4485999} as:
	\begin{align}
		\label{stab4th} S_{\pT}^{\Delta}(\uliney_T,\ulinev_T) &:= h_T^{-3}(\y_{\pT}-\y_T, \v_{\pT}-\v_T)_{\pT}\\
		\nonumber &\quad +h_T^{-1}\left(\Pi_{\pT}^{k}(\gamma_{\pT}-\dn\y_T), \Pi_{\pT}^{k}(\eta_{\pT}-\dn\v_T)\right)_{\pT},\\
	\label{stab2nd}	S_{\pT}^{\nabla}(\ulinez_T,\ulineu_T) &:= h_T^{-1} (\Pi_{\pT}^{k}(\z_{\pT}-\z_T), \Pi_{\pT}^{k}(\u_{\pT}-\u_T))_{\pT}.
	\end{align}
	
	It is easy to verify that for all $\ulinev_T \in \VTk$ and $\ulineu_T \in \UTk$, we have the estimates
	\begin{align}
		\label{lbs1}	|\ulinev_T|^2_{\VTk} &\lesssim a^{\Delta}_T(\ulinev_T,\ulinev_T) \lesssim         |\ulinev_T|^2_{\VTk}, && 
		\ulineu_T|^2_{\UTk} \lesssim a^{\nabla}_T(\ulineu_T,\ulineu_T)\lesssim |\ulineu_T|^2_{\UTk}.
	\end{align}
	\subsection{Discrete problem}\label{Discrete problem}
	In the global setting, the discrete spaces are defined by 
	\[
	\Vhk:=\P^{k+2}(\CTh) \times \P^{k+2}(\CFh) \times {\P^{k}(\CFh)}, \qquad 
	\Uhk:=\P^{k+1}(\CTh) \times \P^{k}(\CFh).
	\]
	{A generic element $(\ulinevh,\ulineu_h) \in\Vhk\times\Uhk$ is such that $\ulinevh=(\v_{\CTh},\v_{\CFh},\eta_{\CFh})$ and $\ulineu_h=(\u_{\CTh},\u_{\CFh})$ with $\v_{\CTh}=(\v_T)_{T\in\CTh}$, $\v_{\CFh}=(\v_F)_{F\in\CFh}$ and $\eta_{\CFh}=(\eta_F)_{F\in\CFh}$, where $\eta_F$ is meant to approximate the normal derivative in the direction of unit normal vector $\n_F$ orienting $F$. The local components of $\ulinevh$ and $\ulineuh$ are collected in the triplet $\ulinev_T:= (\v_T,\v_{\partial T},\eta_{\partial T}) \in \VTk$ with $\v_{\partial T}|_{F}:= \v_{F}$, $\eta_{\partial T}|_{F}:= (\n_F \cdot \n_T)\eta_{F}$} and  $\ulineu_T:= (\u_T,\u_{\partial T}) \in \UTk$ with $\u_{\partial T}|_{F}:= \u_{F}$, respectively.
	For all $T\in\CTh$, a natural choice for the global HHO spaces which enforce the homogeneous boundary conditions are
	\[
	\Vhok :=\{\ulinevh\in \Vhk\ |\ \v_F=0\ \forall\ F \in\CFh^b\}, \qquad
	\Uhok :=\{\ulineuh\in \Uhk\ |\ \u_F=0\ \forall\ F \in\CFh^b\},
	\]
	which are equipped with norms 
	\[
	\|\ulinevh\|^2_{\Vhok} = \sum_{T\in\CTh}|\ulinev_T|^2_{\VTk}, \qquad 
	\|\ulineuh\|^2_{\Uhok} = \sum_{T\in\CTh}|\ulineu_T|^2_{\UTk}.
	\]

	The global discrete problem is: find $(\uliney_h,\ulinez_h)\in\Vhok\times\Uhok$ such that
	\begin{subequations}\label{discrform}
	\begin{align}
		\label{discrete1} a_h^{\Delta}(\uliney_h,\ulinev_h)- (\z_{\CTh},\v_{\CTh}) &= 0 && \forall \ulinevh \in \Vhok,\\
		\label{discrete2}	(\y_{\CTh},\u_{\CTh}) + \kappa a_h^{\nabla}(\ulinez_h,\ulineu_h)&=(\f,\u_{\CTh}) && \forall \ulineu_h \in \Uhok,
	\end{align}
	\end{subequations}
	where
	\[
	a_h^{\Delta}(\uliney_h,\ulinev_h) := \sum_{T\in\CTh}a_T^{\Delta}(\uliney_T,\ulinev_T) \qquad\text{and}\qquad 	a_h^{\nabla}(\ulinez_h,\ulineu_h):=\sum_{T\in\CTh}a_T^{\nabla}(\ulinez_T,\ulineu_T).
	\]
	
        To discuss the well-posedness of the discrete problem, consider the product space $\Zhok=\Vhok\times\Uhok$ which is equipped with the norm
        \[
        \|(\ulinev_h,\ulineu_h)\|_{\Zhok}:= \Big(\|\ulinev_h\|^2_{\Vhok}+\|\ulineu_h\|^2_{\Uhok}\Big)^{1/2}.
        \]
        
         Consider the discrete problem: find $(\uliney_h,\ulinez_h)\in \Zhok$ such that
		\begin{equation}\label{sumdiscrete} \mathcal{A}_h((\uliney_h,\ulinez_h),(\ulinev_h,\ulineu_h))=(\f,\u_{\CTh})\qquad  \forall(\ulinev_h,\ulineu_h)\in \Zhok,
		\end{equation}
		where the bilinear form $\mathcal{A}_h: \Zhok \times \Zhok \rightarrow \R$ is defined as:
		\begin{equation*}
			\mathcal{A}_h((\uliney_h,\ulinez_h),(\ulinev_h,\ulineu_h)):= a_h^{\Delta}(\uliney_h,\ulinev_h)-(\z_{\CTh},\v_{\CTh}) + (\y_{\CTh},\u_{\CTh}) + \kappa a_h^{\nabla}(\ulinez_h,\ulineu_h).
		\end{equation*}
		
		 By summing \eqref{lbs1} over all $T \in \CTh$, we get the global coercivity condition
		\begin{equation}
			\label{gcp} \|(\ulinev_h,\ulineu_h)\|^2_{\Zhok}\lesssim \mathcal{A}_h((\ulinev_h,\ulineu_h),(\ulinev_h,\ulineu_h)).
		\end{equation}

    Using global coercivity property \eqref{gcp}, an application of Lax-Milgram lemma ensures the well-posedness of the discrete problem \eqref{sumdiscrete}.
	\begin{remark}
		 For the Cahn-Hilliard boundary conditions, the discrete mixed problem is: find $(\uliney_h,\ulinez_h)\in {\tilde{\hat{V}}_{h,0}^k \times \tilde{\hat{U}}_{h,0}^k}$ such that
		 \[
		 \mathcal{A}_h((\uliney_h,\ulinez_h),(\ulinev_h,\ulineu_h))=(\f,\u_{\CTh})\qquad  \forall(\ulinev_h,\ulineu_h)\in {\tilde{\hat{V}}_{h,0}^k \times \tilde{\hat{U}}_{h,0}^k},
		 \]
	 where the discrete spaces ${\tilde{\hat{V}}_{h,0}^k}$ and $\tilde{\hat{U}}_{h,0}^k$ are defined as:
	 \begin{align*}
	 	{\tilde{\hat{V}}_{h,0}^k} \hspace{-0.2mm}:=\hspace{-0.2mm} \left\{\ulinev_h\in \Vhk|\, (\v_{\CTh},1) \hspace{-0.2mm}=\hspace{-0.2mm} 0,\, \eta_{F} \hspace{-0.2mm}=\hspace{-0.2mm} 0\ \forall F \in \CF^b\right\}, \quad
	 	{\tilde{\hat{U}}_{h,0}^k} \hspace{-0.2mm}:=\hspace{-0.2mm} \left\{\ulineu_h\in \Uhk|\, (\u_{\CTh},1) = 0\right\}.
	 \end{align*}
	\end{remark}

\subsection{A priori error analysis}
	To establish the a priori error estimates for first type of boundary conditions, consider the local reduction operators $\IVTk:H^{2}(T)\to\VTk$ and $\IUTk: H^1(T) \to\UTk$ from \cite[Sec.~5.2]{MR4485999} defined as:
	\begin{subequations}
	\begin{align*}
		\IVTk(\v)&:=(\PiTkpt\v,\Pi_{\pT}^{k+2}\v,\Pi_{\pT}^k(\n_T\cdot\nabla\v)),	&&
		\IUTk(\u):=(\Pi_T^{k+1}\u,\Pi_{\pT}^k\u).
	\end{align*}
		\end{subequations}
		
	The composition  of reconstruction and local reduction operators defines the following local elliptic projections:
	\begin{align}\label{ellipproj}
		\EVT&\hspace{-0.43mm}:=\hspace{-0.43mm}\RT \hspace{-0.3mm} \circ \hspace{-0.3mm}\IVTk:H^{2}(T) \hspace{-0.3mm} \to \hspace{-0.3mm} \P^{k+2}(T),\ 
		\EUT\hspace{-0.43mm}:=\hspace{-0.43mm}\GT\circ\IUTk:H^{1}(T) \hspace{-0.3mm} \to \hspace{-0.3mm} \P^{k+1}(T).
	\end{align}
	
	For $\u\in\U$, $\EUT(\cdot)$ satisfies the following properties:
	\begin{align*}
		(\nabla\EUT(\u),\nabla\zeta)_T&=(\nabla\u,\nabla\zeta)_T &&\forall\zeta\in\P^{k+1}(T),&&
		(\EUT(\u),1)_T =(\u,1)_T.
	\end{align*}
	
	For all $T\in\CTh$ and $\v\in H^{2+s}(T)$ with $s>\frac{3}{2}$, to bound the consistency error, we define the norm $\|\cdot\|_{\sharp,T,1}$ as:
	$$\|\v\|^2_{\sharp,T,1} :=\|\nabla^2\v\|^2_{T}+h_T^3\|\dn\Delta\v\|^2_{\pT}+h_T\|\dnn\v\|^2_{\pT}+h_T\|\dnt\v\|^2_{\pT}.$$
	
	We have the following estimates \cite[Lemma~5.4]{MR4485999}, \cite[Theorm~1.48]{MR4230986}:
	\begin{subequations}\label{ellpest}
		\begin{align}
			\label{projest1}	\|\nabla^2(\v-\EVT(\v))\|_{\sharp,T,1}&\lesssim \|\nabla^2(\v-\PiTkpt(\v))\|_{\sharp,T,1}  \quad\forall \v\in\V,\\
			\label{projest2}	|\u-\EUT(\u)|_{H^m(T)}&\lesssim h_T^{\lambda-m}|\u|_{H^{\lambda}(T)} \qquad \forall \u\in\U,\  0\le m\le \lambda \le k+2.
		\end{align}
	\end{subequations}
	
	The global reduction operators $\IVhk: H^{2}(\O) \to \Vhk$ and $\IUhk: H^{1}(\O) \to \Uhk$ are defined such that for all $\v\in H^{2}(\O)$ and $\u\in H^{1}(\O)$  
	\begin{subequations}\label{intrest}
	\begin{align}
		\IVhk(\v)&:= \left((\PiTkpt(\v))_{T\in\CTh}, (\Pi_F^{k+2}(\v))_{F\in\CFh}, (\PiFk(\n_T \! \cdot \! \nabla \v))_{F\in\CFh} \right), \\
		\IUhk(\u) & := \left((\Pi_T^{k+1}(\u))_{T\in\CTh},  (\PiFk(\u))_{F\in\CFh} \right).
	\end{align}
	\end{subequations}
	
	The local components of $\IVhk(\v)$ and $\IUhk(\u)$ are $\IVTk(\v|_T)$ and $\IUTk(\u|_T)$ for all $T\in\CTh$, respectively. Note that, for the exact solutions $(\y,\z)\in\V\times\U$, their image under the global reduction operators $(\IVhk(\y),\IUhk(\z))\in\Vhok\times\Uhok$. 

	Next, we define the  consistency error $\delta_h((\y,\z);\cdot)\in(\Zhok)^{'}$ as:
	\begin{align}
	\delta_h((\y,\z);(\ulinev_h,\ulineu_h)) := (\f,\u_{\CTh}) - \mathcal{A}_h((\IVhk\y,\IUhk\z),(\ulinevh,\ulineuh)).
	\end{align}
	\begin{lemma}[Consistency error]\label{consistecy err}
		Suppose that $\y\in H^{2+s}(\O)$ for $s>{3}/{2}$ and $\z\in H^{1+r}(\O) \cap H^{k+2}(\CTh)$ for $r > {1}/{2}$ be the solutions of  \eqref{wf12} and $(\uliney_h,\ulinez_h)\in \Zhok$ be the solution of \eqref{discrform}. Then, the consistency error satisfies the estimate
		\begin{align*}
			\|\delta_h((\y,\z);\cdot)\|_{(\Zhok)'} &:= \underset{(\ulinev_h,\ulineu_h) \in \Zhok}{\sup} \frac{|\delta_h((\y,\z);(\ulinev_h,\ulineu_h))|}{\|(\ulinev_h,\ulineu_h)\|_{\Zhok}} \\
			&\lesssim R_{s} := \Bigg( \sum_{T\in\CTh}h_T^{2(k+1)} |\z|^2_{H^{k+2}(T)} \hspace{-0.3mm}+\hspace{-0.3mm} \|\y-\PiTkpt\y\|^2_{\sharp,T,1} \Bigg)^{\frac{1}{2}}.
		\end{align*}
	\end{lemma}
	\begin{proof}
		Let $(\ulinev_h,\ulineu_h)\in \Zhok$. Using \eqref{4thorder} and integrating by parts cellwise, we get 
		\begin{align*}
			(\f,\u_{\CTh}) &= \sum_{T\in\CTh} \Big( (\y,\u_T)_T + \kappa(\nabla\z,\nabla\u_T)_T - \kappa(\dn\z,\u_T)_{\pT} - (\z,\v_T)_T + (\nabla^2\y,\nabla^2\v_T)_T\\
			&\qquad \qquad + (\dn\Delta\y,\v_T)_{\pT} - (\dnn\y,\dn\v_T)_{\pT} - (\dnt\y,\dt\v_T)_{\pT} \Big).
		\end{align*}
		Since $\dn\z|_{\pT},\dn\Delta\y|_{\pT}, \dnn\y|_{\pT}$ and $\dnt\y|_{\pT}$ are meaningful in $L^2(\pT)$, single-valued at each mesh interfaces and $\dnn\y|_{\pT}$ vanishes on boundary faces. Moreover, since $\u_{\pT}, \v_{\pT}, \eta_{\pT}, \dt\v_T$ are single-valued at every mesh interface and $\u_{\pT}, \v_{\pT}, \dt\v_T$  vanish on boundary faces, we have
		\begin{align*}
			(\f,\u_{\CTh}) \! &=\hspace{-1.5mm} \sum_{T\in\CTh}\hspace{-1mm}\Big(\!(\y,\u_T)_T \! + \! \kappa(\nabla\z,\nabla\u_T)_T \! - \! \kappa(\dn\z,\u_T \!-\!\u_{\pT})_{\pT} \!- \! (\z,\v_T)_T \! + \! (\nabla^2\y,\!\nabla^2\v_T)_T\\ 
			&\qquad \qquad +(\dn\Delta\y,\v_T-\v_{\pT})_{\pT}-(\dnn\y,\dn\v_T)_{\pT} -(\dnt\y,\dt(\v_T-\v_{\pT}))_{\pT}\Big),
		\end{align*}
		
		Also, from \eqref{sumdiscrete}, \eqref{ellipproj} and the definition of reconstruction operators, we obtain
		\begin{align*}
			\mathcal{A}_h((\IVhk\y,\IUhk\z),(\ulinevh,\ulineuh)) 
			& =\! \sum_{T\in\CTh} \! \Big( (\PiTkpt\y,\u_T)_T + \kappa(\nabla\EUT\z,\nabla\u_T)_T \! - \! (\Pi_T^{k+1}\z,\v_T)_T \\
			&  \hspace{-1.8cm}+S^{\nabla}_{\pT}(\IUTk\z,\ulineu_T)- \kappa(\dn\EUT(\z),\u_T-\u_{\pT})_{\pT} + (\nabla^2 \EVT(\y),\nabla^2\v_T)_T\\
			&  \hspace{-1.67cm} + (\dn\Delta\EVT(\y),\v_T-\v_{\pT})_{\pT} -(\dnn\EVT(\y),\dn\v_T-\eta_{\pT})_{\pT} \\
			&  \hspace{-1.67cm} - (\dnt\EVT(\y),\dt(\v_T-\v_{\pT}))_{\pT} + S^{\Delta}_{\pT}(\IVTk\y,\ulinev_T) \Big).
		\end{align*}
		
		Define functions $\mu,\tau$ cellwise such that $\mu|_T = \y|_T-\EVT(\y|_T)$, $\tau|_T = \u|_T-\EUT(\u|_T)$ and infer that 
		\begin{align*}
			\delta_h((\y,\z);(\ulineu_h,\ulinev_h))
			& = \sum_{T\in\CTh} \Big( (\y-\PiTkpt\y,\u_T)_T + \kappa(\nabla\tau,\nabla\u_T)_T - (\z-\Pi_T^{k+1}\z,\v_T)_T \\
			& \qquad \qquad - S^{\nabla}_{\pT}(\IUTk\z,\ulineu_T) - \kappa(\dn\tau,\u_T-\u_{\pT})_{\pT} + (\nabla^2 \mu,\nabla^2\v_T)_T\\
			& \qquad \qquad +(\dn\Delta\mu,\v_T-\v_{\pT})_{\pT} -(\dnn\mu,\dn\v_T-\eta_{\pT})_{\pT}  \\
			& \qquad \qquad -(\dnt\mu,\dt(\v_T-\v_{\pT}))_{\pT} - S^{\Delta}_{\pT}(\IVTk\y,\ulinev_T) \Big).
		\end{align*}
		
		Note that, $(\nabla\tau,\nabla\u_T)_T=0$, but we keep this term since it can be bounded in a manner similar to the other terms. Also, by the definition of $L^2$-projection, $$(\y-\PiTkpt\y,\u_T)_T=0, \qquad (\z-\Pi_T^{k+1}\z,\v_T)_T = (\z-\Pi_T^{k+1}\z,\v_T - \Pi_T^{k+1}\v_T)_T.$$ 
		
		The remaining terms are bounded as:
		\begin{align*}
			\left|\sum_{T\in \CTh}(\z-\Pi_T^{k+1}\z,\v_T - \Pi_T^{k+1}\v_T)_T\right| 
			&\lesssim \sum_{T\in \CTh} h_T^{k+4}|\z|_{H^{k+2}(T)}|\ulinev_T|_{\VTk},\\
			\left|\sum_{T\in \CTh}\kappa(\nabla\tau,\nabla\u_T)_T - \kappa(\dn\tau,\u_T-\u_{\pT})_{\pT}\right| 
			&\lesssim \sum_{T\in \CTh} h_T^{k+1}|\z|_{H^{k+2}(T)}|\ulineu_T|_{\UTk},
		\end{align*}
		and
	\begin{align*}
			|S^{\nabla}_{\pT}(\IUTk\z,\ulineu_T)| \le S^{\nabla}_{\pT}(\IUTk\z,\IUTk\z)^{\frac{1}{2}} S^{\nabla}_{\pT}(\ulineu_T,\ulineu_T)^{\frac{1}{2}}
			&\lesssim h_T^{k+1}|\z|_{H^{k+2}(T)}|\ulineu_T|_{\UTk},\\
			|S^{\Delta}_{\pT}(\IVTk\y,\ulinev_T)| \le S^{\Delta}_{\pT}(\IVTk\y,\IVTk\y)^{\frac{1}{2}}S^{\Delta}_{\pT}(\ulinev_T,\ulinev_T)^{\frac{1}{2}} & \lesssim  \|\y-\PiTkpt\y\|_{\sharp,T,1}|\ulinev_T|_{\VTk},\\
			|(\nabla^2 \mu,\nabla^2\v_T)_T +(\dn\Delta\mu,\v_T-\v_{\pT})_{\pT}
			&\\ 
			- (\dnn\mu,\dn\v_T-\eta_{\pT})_{\pT} -(\dnt\mu,\dt(\v_T-\v_{\pT}))_{\pT}|
			&\lesssim \|\y-\PiTkpt\y\|_{\sharp,T,1}|\ulinev_T|_{\VTk}.
		\end{align*}
		
		Combining these estimates, we deduce that 
		\begin{align*}
			|\delta_h((\y,\z);(\ulinev_h,\ulineu_h))|&\lesssim R_{s} \|(\ulinev_h,\ulineu_h)\|_{\Zhok}.
		\end{align*}
	\end{proof}
	\begin{lemma}\label{discrete error}
{Under the assumptions made in Lemma~\ref{consistecy err}, we have the estimate}
		\begin{align*}
			\|(\uliney_h-\IVhk\y,\ulinez_h-\IUhk\z)\|_{\Zhok} &\lesssim R_{s} \|(\ulineu_h,\ulinev_h)\|_{\Zhok}.
		\end{align*}
	\end{lemma}
	\begin{proof}
 Since, $(\IVhk(\y),\IUhk(\z))\in \Zhok$, let $\ulinee_h=\uliney_h-\IVhk\y$ and $\ulinew_h=\ulinez_h-\IUhk\z$. From the global coercivity \eqref{gcp}, we deduce that 
 \begin{align*}
 		\|(\ulinee_h,\ulinew_h)\|^2_{\Zhok}&\lesssim \mathcal{A}_h \big((\ulinee_h,\ulinew_h),(\ulinee_h,\ulinew_h)\big)\\
 		&=|\mathcal{A}_h\big((\uliney_h,\ulinez_h),(\ulinee_h,\ulinew_h)\big)- \mathcal{A}_h\big((\IVhk(\y),\IUhk(\z)),(\ulinee_h,\ulinew_h)\big)|\\
 		&=|(\f,\u_{\CTh})- \mathcal{A}_h\big((\IVhk(\y),\IUhk(\z)),(\ulinee_h,\ulinew_h)\big)|\\
 		&=	|\delta_h((\y,\z);(\ulinee_h,\ulinew_h))| 
 		\lesssim \|\delta_h((\y,\z);\cdot)\|_{(\Zhok)'}	\|(\ulinee_h,\ulinew_h)\|_{\Zhok}.
 \end{align*}
 Now, an application of previous lemma concludes the proof.
	\end{proof}
	\begin{theorem} \label{HHOenergyerr}
		{Suppose that the assumptions made in Lemma~\ref{consistecy err} hold true. In addition, if we also have $\y \in H^{\max\{4, k+3\}}(\CTh)$, then}
		\begin{align*}
			&\sum_{T\in\CTh}(\|\nabla^2(\y-\RT(\uliney_T))\|_T^2 +\|\nabla(\z-\GT(\ulinez_T))\|_T^2) \\
			&\lesssim \begin{cases}
				(h_T|\y|_{H^{3}(T)} + h_T^2|y|_{H^4})^2 + h_T^2 |z|_{H^2(T)}^2 \quad \text{if}\ \ k=0,\\
				h_T^{2(k+1)}(|y|_{H^{k+3}(T)}^2 + |z|_{H^{k+2}(T)}^2) \qquad\qquad\ \  \text{if}\ \ k\ge1.
			\end{cases}
		\end{align*}
	\end{theorem}
	\begin{proof}
		By a use of the triangle inequality, we have
		\begin{align*}
			&\sum_{T\in\CTh}(\|\nabla^2(\y-\RT(\uliney_T))\|_T^2+\|\nabla(\z-\GT(\ulinez_T))\|_T^2)\\
			&\lesssim  \sum_{T\in\CTh}\Big(\|\nabla^2(\y-\EVT\y)\|_T^2+\|\nabla^2 \RT(\IVTk\y-\uliney_T)\|_T^2\\
			&\qquad \qquad +\|\nabla(\z-\EUT\z)\|_T^2+\|\nabla \GT(\IUTk\z-\ulinez_T)\|_T^2\Big)\\
			&\lesssim  \sum_{T\in\CTh}\Big(\|\nabla^2(\y-\EVT\y)\|_T^2+\|\IVTk\y-\uliney_T\|_{\VTk}^2\\
			&\qquad \qquad +\|\nabla(\z-\EUT\z)\|_T^2+\|\IUTk\z-\ulinez_T\|_{\UTk}^2\Big),
		\end{align*}
		where we have used the local boundedness of discrete problem. The bound for first and third term follows from the estimates \eqref{ellpest} and the bound for second and fourth term follows from Lemma~\ref{discrete error}. This concludes the proof.
	\end{proof}
	\section{$C^0$-conforming HHO discretization}\label{C0 HHO Scheme}\setcounter{equation}{0}
	In this section, we discretize the variable $\y$ using $C^0$-HHO method and $\z$ with conforming FEM scheme for a  shape regular simplicial mesh sequence $\CMh$ {(in the sense of \cite{MR1930132})}. Let $\P_\mathrm{g}^l(\CTh)$ denotes the usual $C^0$-conforming finite element space of continuous piecewise polynomials of degree at most $l$. For $l+1 > {d}/{2}$, let $\Lhl$ denotes the $H^1$-conforming Lagrange interpolation operator on $\P_\mathrm{g}^l(\CTh)$ and $\LTl$ is its local version mapping onto $\P^l(T)$ for all $T\in\CTh$. There exists a constant $C >0$ such that, for an integer $s$ with $1<s\le l+1$, and $m\in\{0,\dots,s\}$, for all $T\in\CTh$
	\begin{equation}\label{lagrangeerr}
		|\v-\LTl(\v)|_{H^m(T)}\le C h^{s-m}|\v|_{H^s(T)} \quad\forall \v\in H^s(T).
	\end{equation}
	
	We define the local HHO space for a mesh element $T\in\CTh$ as 
	$$\VTk:=\P^{k+2}(T) \times\P^{k}(\CF_{\partial T}),$$ such that {
	$\ulinev_T=(\v_T,\eta_{\partial T})$ is a generic element of $\VTk$, where $\v_T \in \P^{k+2}(T)$ represents the solution inside the mesh cell and $\eta_{\partial T} \in \P^{k}(\CF_{\partial T})$ represents the trace of normal derivative (oriented by $\n_T$) on the cell boundary.} 
		
	To formulate the HHO method on each cell, the local reconstruction operator {$\RTc:\VTk\to\P^{k+2}(T)$} is defined as: \vspace{-1mm}
	\begin{align}
		\label{R2} (\nabla^2 \RTc(\ulinev_T),\nabla^2 \w)_{T}
		&=(\nabla ^{2}\v_T,\nabla^{2} \w)_{T} \hspace{-0.4mm}-\hspace{-0.4mm}(\dn\v_T \hspace{-0.4mm}-\hspace{-0.4mm} \eta_{\pT},\dnn \w)_{\pT}\ \forall \w\in\P^{k+2}(T),\\
		(\RTc(\ulinev_T),\xi)_T&=(\v_T,\xi)_T\qquad \qquad \qquad \qquad \qquad \qquad \qquad \hspace{-0.9mm}\forall \xi\in\P^1(T).\notag
	\end{align}
	
	Now, we define the local bilinear form $a^{c}_T:\VTk\times\VTk\to\R$ as:
	\begin{align}
	  a^{c}_T(\uliney_T,\ulinev_T) &= (\nabla^2 \RTc(\uliney_T),\nabla^2 \RTc(\ulinev_T))_T + S^{c}_{\pT}(\uliney_T,\ulinev_T),
	\end{align}
	for all $\uliney_T, \ulinev_T \in \VTk$, where $\uliney_T = (\y_T,\gamma_{\pT})$ and stabilization term $S_{\pT}^c(\cdot,\cdot)$ is defined as:
	\[
	S_{\pT}^c(\uliney_T,\ulinev_T) :=  h_T^{-1}\left(\Pi_{\pT}^{k}(\gamma_{\pT}-\dn\y_T),\Pi_{\pT}^{k}(\eta_{\pT}-\dn\v_T)\right)_{\pT}.
	\]
	 
	We define the local energy semi-norm on the space $\VTk$ as:
	\begin{align*}
		|\ulinev_T|^2_{\VTk}&:=\|\nabla^2\v_T\|^2_{T}+ h_T^{-1}\|\eta_{\pT} -\dn\v_T\|^2_{\pT}.
	\end{align*}
	
	\subsection{Discrete problem}\label{Discrete problem 2}
	The global $C^0$-HHO space is given by 
	\[
	\Vhk:=\P_{\mathrm{g}}^{k+2}(\CTh) \times {\P^{k}(\CFh)}.
	\]

	{The generic element $\ulinevh \in\Vhk$ is such that $\ulinevh=(\v_{\CTh},\eta_{\CFh})$ with $\v_{\CTh}=(\v_T)_{T\in\CTh}$ and $\eta_{\CFh}=(\eta_F)_{F\in\CFh}$. For all $T \in \CTh$, the local components of $\ulinevh$ are collected in the pair $\ulinev_T:= (\v_T, \eta_{\partial T}) \in \VTk$ with $\eta_{\partial T}|_{F}:= (\n_F \cdot \n_T)\eta_{F}$.}
	We use space ${\P^{k}(\CF_h)}$ to approximate the normal derivative in the direction of unit normal vector $\n_F$. 
	A natural choice for the global HHO space with homogeneous boundary condition is 
	\[
	\Vhok:=\{\ulinevh\in \Vhk\ |\ \v_{\CTh}\in H_0^1(\O)\}.
	\]
	
	We equip the space $\Vhok$ with the norm $\|\ulinevh\|_{\Vhok}^2=\sum\limits_{T\in\CTh}|\ulinev_T|^2_{\VTk}$. Let $\Usfhok$ be the subspace of $\P_{\mathrm{g}}^{k+1}(\CTh)$ which consists of the functions vanishing on the boundary $\partial\O$. For any mesh element $T\in\CTh$ and $\u_{\CTh}\in\Usfhok$, the restriction of $\u_{\CTh}$ on $T$ is denoted by $\u_T$.
	The global discrete problem is: find $(\uliney_h,\z_{\CTh})\in\Vhok\times\Usfhok$ such that
	\begin{subequations}\label{c0discrete}
	\begin{align}
		\label{C0discrete1} a^{c}_h(\uliney_h,\ulinev_h)-(\z_{\CTh},\v_{\CTh}) &= 0 && \forall \ulinevh \in \Vhok,\\
		\label{C0discrete2} (\y_{\CTh},\u_{\CTh}) + \kappa(\nabla\z_{\CTh},\nabla\u_{\CTh})&=(\f,\u_{\CTh}) && \forall \u_{\CTh} \in \Usfhok,
	\end{align}
\end{subequations}
	where $a^{c}_h(\uliney_h,\ulinev_h):=\sum_{T\in\CTh} a^{c}_T(\uliney_T,\ulinev_T)$.
	For the well-posedness, consider the space $\Zhok=\Vhok\times\Usfhok$ equipped with the norm
	\[
	\|(\ulinev_h,\ulineu_h)\|_{\Zhok}:= \Big(\|\ulinev_h\|^2_{\Vhok}+\|\nabla\u_{\CTh}\|^2\Big)^{1/2}.
	\]
	Consider the problem: find $(\uliney_h,\z_{\CTh})\in \Zhok$ such that
		\begin{equation}\label{C0discrete}
			\mscr{A}_h((\uliney_h,\z_{\CTh}),(\ulinev_h,\u_{\CTh}))=(\f,\u_{\CTh}) \qquad \forall (\ulinev_h,\ulineu_h)\in \Zhok,
		\end{equation}
		where 
		\begin{equation*}
			\mscr{A}_h((\uliney_h,\z_{\CTh}),(\ulinev_h,\u_{\CTh}))=a^{c}_h(\uliney_h,\ulinev_h)-(\z_{\CTh},\v_{\CTh}) 
			+	(\y_{\CTh},\u_{\CTh}) + \kappa(\nabla\z_{\CTh},\nabla\u_{\CTh}).
		\end{equation*}
     Using the global stability and boundedness, an application of Lax-Milgram lemma ensures well-posedness of the discrete problem \eqref{C0discrete}.
	\begin{remark}
     For Cahn-Hilliard boundary conditions, the discrete mixed problem is: find $(\uliney_h,\z_{\CTh})\in {\tilde{\hat{V}}_{h,0}^k \times \tilde{{U}}_{h,0}^k}$ such that
     \[
     \mscr{A}_h((\uliney_h,\ulinez_h),(\ulinev_h,\ulineu_h))=(\f,\u_{\CTh})\qquad  \forall(\ulinev_h,\ulineu_h)\in  {\tilde{\hat{V}}_{h,0}^k \times \tilde{{U}}_{h,0}^k},
     \]
    where the discrete spaces $\tilde{\hat{V}}_{h,0}^k$ and $\tilde{{U}}_{h,0}^k$ are defined as:
    \begin{align*}
	     \tilde{\hat{V}}_{h,0}^k  &:= \{\ulinev_h\in \Vhk \ | \ (\v_{\CTh},1) = 0\}, \quad 
	     \tilde{{U}}_{h,0}^k := \{\u_{\CTh}\in \P_\mathrm{g}^{k+1}({\CTh})\cap H^1(\Omega)  \ | \ (\u_{\CTh},1) = 0\}.
    \end{align*}	
		\end{remark}
	
	For all $T\in\CTh$ and $\v\in H^{2+s}(T)$ with $s>{1}/{2}$, we introduce the norm $\|\cdot\|_{\sharp,T,2}$ to bound the consistency error
	\[
	\|\v\|^2_{\sharp,T,2} :=\|\nabla^2\v\|^2_{T}+h_T\|\dnn\v\|^2_{\pT}.
	\]

	Next, define the local reduction operator $\IVTkm: H^2(T) \to \VTk$ as:
	\begin{equation*}
		\IVTkm(\v):=\left(\LTkpt(\v),\Pi_{\pT}^{k}(\n_T.\nabla\v)\right)\qquad\forall \v\in H^2(T).
	\end{equation*} 
	
	Also, define the projector $\EVTm:H^2(T)\to\P^{k+2}(T)$ as $\EVTm=\RTc\circ\IVTkm$. Note that,  for all $T\in\CTh$ and $\v\in H^{2+s}(T)$ where $s>{1}/{2}$, by \cite[Lemma~3.8]{MR4699573}
	\begin{equation}\label{errlemma}
		\|\nabla^2(\v-\EVTm(\v))\|_{\sharp,T,2}+S^c_{\pT}(\IVTkm(\v),\IVTkm(\v))^{\frac{1}{2}}\le C\|\nabla^2(\v-\LTkpt(\v))\|_{\sharp,T,2}.
	\end{equation}
	
	The global interpolator $\IVhkm:H^2(\Omega)\to \Vhk$ is defined as:
	\begin{equation*}
		\IVhkm(\v):=\left(\Lhkpt(\v),(\PiFk(\n_F.\nabla\v))_{F\in\CFh}\right) \quad\forall\v\in H^2(\Omega).
	\end{equation*} 
	
	The local components of $\IVhkm(\v)$ are $\IVTkm(\v|_T)$ for all $T\in\CTh$. 
	\subsection{A priori error estimates} 
	To derive the error estimates for the system, we define the  consistency error $\delta_h((\y,\z);\cdot)\in(\Zhok)^{'}$ as:
	\begin{align}
		\delta_h((\y,\z);(\ulinev_h,\u_{\CTh})):= (\f,\u_{\CTh})-\mscr{A}_h((\IVhkm\y,\LTki\z),(\ulinevh,\u_{\CTh})),
	\end{align}
	\begin{lemma}[Consistency error]\label{C0Consistency err}
		Let $\y\in H^{2+s}(\O) \cap H^{k+3}(\CTh)$ for $s>1/2$ and $\z\in\U\cap H^{k+2}(\CTh)$ be the solutions of \eqref{wf12} and $(\uliney_h,\ulinez_h)\in \Zhok$ be the solution of \eqref{c0discrete}. Then, the consistency error satisfies the estimate
		\begin{equation}\label{C0CE}
			\|\delta_h((\y,\z);\cdot)\|_{(\Zhok)^{'}}\lesssim \Bigg( \sum_{T \in \CTh} h_T^{2(k+1)} (|y|_{H^{k+3}(T)}^2 + |z|_{H^{k+2}(T)}^2) \Bigg)^{\frac{1}{2}}.
		\end{equation}
	\end{lemma}
	\begin{proof}
			Let $(\ulinev_h,\u_{\CTh})\in \Zhok$. Since $\v_{\CTh},\u_{\CTh}\in H_0^1(\O)$, using the PDEs and integrating by part cellwise, we get 
		\begin{align*}
			(\f,\u_{\CTh}) &=\! \sum_{T\in\CTh}\! (\y,\u_T)_T+\kappa(\nabla\z,\nabla\u_T)_T -(\z,\v_T)_T + (\nabla^2\y,\nabla^2\v_T)_T -(\dnn\y,\dn\v_T)_{\pT}.
		\end{align*}
		
		The quantity $\dnn\y|_{\pT}$ is meaningful in $L^2(\pT)$, single-valued at each mesh interfaces and vanishes at boundary faces. Since  $\eta_{\pT}$ is single-valued at every mesh interface, we have
		\begin{align*}
			(\f,\u_{\CTh}) &= \sum_{T\in\CTh}\Big((\y,\u_T)_T+\kappa(\nabla\z,\nabla\u_T)_T -(\z,\v_T)_T+(\nabla^2\y,\nabla^2\v_T)_T\\ 
			&\qquad \qquad -(\dnn\y,\dn\v_T-\eta_{\pT})_{\pT} \Big).
		\end{align*}
		
		Also, from \eqref{sumdiscrete} and the definition of reconstruction operators, we deduce that
		\begin{align*}
			\mscr{A}_h(\hspace{-0.5mm}(\IVhkm\y,\LTki\z),(\ulinevh,\u_{\CTh})\hspace{-0.5mm}) \!
			&=\hspace{-0.5mm} \sum_{T\in\CTh}\hspace{-1mm}\Big(\hspace{-0.5mm} (\LTkpt\y,\u_T)_T \! + \! \kappa(\nabla\LTki\z,\!\nabla\u_T)_T \! - \! (\LTki\z,\v_T)_T \\
			& \qquad + (\nabla^2 \RTc(\y),\nabla^2\v_T)_T \! - \! (\dnn\EVTm(\y),\dn\v_T \! - \! \eta_{\pT})_{\pT} \\
			&\qquad + S^c_{\pT}(\IVTkm\y,\ulinev_T)\Big). 
		\end{align*}
		
		Define $\theta,\sigma$ cellwise such that $\theta|_T = \y|_T-\EVTm(\y|_T)$ and $\sigma|_T = \z|_T-\LTki(\z|_T)$ and  infer that 
		\begin{align*}
			\delta_h((\y,\z);(\ulinev_h,\ulineu_h))
			&=\sum_{T\in\CTh}\Big( (\y-\LTkpt\y,\u_T)_T + \kappa(\nabla\sigma,\nabla\u_T)_T - (\sigma,\v_T)_T \\
			&\qquad \ + (\nabla^2 \theta,\nabla^2\v_T)_T - (\dnn\theta,\dn\v_T-\eta_{\pT})_{\pT} -S^{c}_{\pT}(\IVTkm\y,\ulinev_T)\Big).
		\end{align*}
		
		A use of approximation properties in \eqref{errlemma} and \eqref{lagrangeerr} concludes the proof.
	\end{proof}
\begin{lemma}\label{discrete error2}
{Under the assumption made in Lemma~\ref{C0Consistency err}, the following holds true:}
	\begin{align}
		\|(\uliney_h-\IVhkm\y, \z_{\CTh}-\mathcal{L}_h^{k+1}\z)\|_{\Zhok}^2\lesssim\sum_{T \in \CTh} h_T^{2(k+1)} (|y|_{H^{k+3}(T)}^2 + |z|_{H^{k+2}(T)}^2).
	\end{align}
\end{lemma}
\begin{proof}
	Let $\ulineeh:= \uliney_h-\IVhkm\y\in \Vhok$ and $\w_{\CTh}:= \z_{\CTh} - \mathcal{L}_h^{k+1}\z \in \Usfhok$, using the coercivity of discrete problem, we get
	\begin{align*}
			\|(\ulineeh, \w_{\CTh})\|^2_{\Zhok}&\lesssim \mscr{A}_h\left((\ulineeh, \w_{\CTh}),(\ulineeh, \w_{\CTh})\right)\\
			& =\mscr{A}_h\left((\ulineyh, \z_{\CTh}), (\ulineeh, \w_{\CTh})\right)-
			 \mscr{A}_h\left((\IVhkm\y, \mathcal{L}_h^{k+1}\z), (\ulineeh, \w_{\CTh})\right)\\
			& = (\f, \w_{\CTh}) -  \mscr{A}_h\left((\IVhkm\y, \mathcal{L}_h^{k+1}\z), (\ulineeh, \w_{\CTh})\right)\\
			& = \delta_h((\y,\z);\ulineeh,\w_{\CTh}).
	\end{align*}
	
	By using Lemma~\ref{C0Consistency err}, we conclude the proof.
\end{proof}
\begin{theorem}
	With the assumption made in Lemma~\ref{C0Consistency err}, the solution satisfy 
	\begin{align*}
		\sum_{T\in\CTh} \hspace{-0.8mm} \big(\|\nabla^2(\y \hspace{-0.3mm}-\hspace{-0.3mm}\RTc(\uliney_T))\|_T^2 \hspace{-0.3mm}+\hspace{-0.3mm}\|\nabla(\z\hspace{-0.3mm}-\hspace{-0.3mm}\z_T)\|_T^2\big)\lesssim \sum_{T \in \CTh} h_T^{2(k+1)} (|y|_{H^{k+3}(T)}^2 \hspace{-0.3mm}+\hspace{-0.3mm} |z|_{H^{k+2}(T)}^2).
	\end{align*}
\end{theorem}
\begin{proof}
The proof follows by arguments similar to Theorem~\ref{HHOenergyerr}.
\end{proof}
\begin{remark}
	{In case of Cahn-Hilliard boundary conditions, similar convergence estimates can be obtained for both schemes by following the same arguments with minor modifications. To avoid repetition, we omit the detailed presentation.}
\end{remark} 
	\section{A posteriori error estimates}\label{A posteriori error estimates}
	In this section, we perform a posteriori error analysis of the non-conforming HHO method presented in Sec.~\ref{HHO Scheme} for polynomial degree $k\ge1$. Firstly, we introduce some operators useful for the analysis.
	\subsection{Interpolation operators and estimator}\label{Interpolation operators and estimator}
		For all $p\ge2$ and an arbitrary small $\epsilon \ge 0$, let $\mathcal{I}_{\text{mBS}}^p : H_0^1(\O) \cap H^{\frac{3}{2}+\epsilon}(\O) \to \P^p(\CTh)\cap H_0^1(\O)$ denotes the modified Babu$\check{\text{s}}$ka-Suri $hp$-interpolation operator defined in \cite[Sec.~2]{dong2026hp}, such that, for all $\v \in H_0^1(\O) \cap H^{\frac{3}{2}+\epsilon}(\Omega) \cap H^2(\CTh)$ and $ T\in \CTh$, the following holds:  
		\begin{align} 
			\label{approx mbs} 
			& \left(\frac{p}{h_T}\right)^2\|\v - \mathcal{I}_{\text{mBS}}^p (\v)\|_T + \left(\frac{p}{h_T}\right)^{\frac{3}{2}}\|\v - \mathcal{I}_{\text{mBS}}^p (\v)\|_{\pT} + \frac{p}{h_T}\|\v - \mathcal{I}_{\text{mBS}}^p (\v)\|_T\\
			\nonumber &+ \left(\frac{p}{h_T}\right)^{\frac{1}{2}}\|\dn(\v - \mathcal{I}_{\text{mBS}}^p (\v))\|_{\pT} + \|\nabla_h^2\mathcal{I}_{\text{mBS}}^p (\v)\|_T \lesssim \|\nabla^2 \v\|_{T}.
		\end{align}
		
		For $p\ge 1$, the modified $hp$-Karkulik–Melenk operator $\mathcal{I}_{\text{mKM}}^p : H_0^1(\O) \to \P^p(\CTh)\cap H_0^1(\O)$ is defined similar to \cite[Corollary~2.5]{dong2025hperroranalysismixedorderhybrid} and satisfies, for all $\u \in H_0^1(\O)$ and $ T\in \CTh$, 
    \begin{align}\label{approx mkm}
    	 \frac{p}{h_T}\|\u \hspace{-0.3mm}-\hspace{-0.3mm} \mathcal{I}_{\text{mKM}}^p (\u)\|_T \hspace{-0.3mm}+\hspace{-0.3mm} \left(\hspace{-0.3mm}\frac{p}{h_T}\hspace{-0.3mm}\right)^{\frac{1}{2}} \hspace{-0.63mm} \|\u \hspace{-0.3mm}-\hspace{-0.3mm} \mathcal{I}_{\text{mKM}}^p (\u)\|_{\pT} \hspace{-0.3mm}+\hspace{-0.3mm} \|\nabla\mathcal{I}_{\text{mKM}}^p (\u)\|_T \hspace{-0.3mm}\lesssim \hspace{-0.3mm} \|\nabla \u\|_{\text{es}(T)},
    \end{align}
    where $\text{es}(T)$ denotes the set of mesh elements which share at least one vertex with the cell star of $T$.
    
    For $k\ge1$, we use the following HHO interpolation operators $\mathcal{J}_{\V,h}^k :V \to \Vhok$ and $\mathcal{J}_{\U,h}^k : U \to \Uhok$ defined in \cite[Eq.~(23b)]{dong2026hp} and \cite[Appendix~7.1]{dong2025hperroranalysismixedorderhybrid}, respectively, as:
    \begin{subequations}
    \begin{align}
    	\mathcal{J}_{\V,h}^k(\v) &: = \left( \mathcal{I}_{\text{mBS}}^{k+2} (\v), (\mathcal{I}_{\text{mBS}}^{k+2} (\v)|_F)_{F\in\CFh}, (\PiFk(\n_F \! \cdot \! \nabla \v)|_F)_{F\in\CFh}  \right) && \forall v \in V,\\
    	\mathcal{J}_{\U,h}^k(\u) &: = \left( \mathcal{I}_{\text{mKM}}^{k} (\u), (\mathcal{I}_{\text{mKM}}^{k} (\u)|_F)_{F\in\CFh} \right) && \forall u \in U.
    \end{align}
    	\end{subequations}

    We will use the notation $\hbar_T := \frac{h_T}{k+2}$ and $\tilde{h}_T:= \frac{h_T}{k+1} $ and define the local error indicators and data oscillation term as: 
	\begin{align*}
		\boldsymbol{\eta}_{T,\res}& :=\hbar^2_T \|\z_{\CTh}-  \Delta^2\RT(\y_T)\|_T+  \tilde{h}_T  \| \Pi_T^{k+1}(\f) + \y_{\CTh} + \kappa\Delta \GT(\ulinez_T)\|_T,\\
		\boldsymbol{\eta}_{T,\sta} &:=S^{\Delta}_{\pT}(\uliney_T,\uliney_T)^{\frac{1}{2}} + S^{\nabla}_{\partial T}(\ulinez_T,\ulinez_T)^{\frac{1}{2}} + \big( \hbar_T\big)^{-\frac{1}{2}} \|\dn\y_{\CTh}\|_{\pT^b},\\
		\boldsymbol{\eta}_{T,\nor}& :=\left( \hbar_T\right)^{\frac{1}{2}} \| [\![ \dnt \RT(\uliney_T)]\!]\|_{\pT^i} +  \left( \hbar_T \right)^{\frac{3}{2}} \| [\![\dn\Delta \RT(\uliney_T)]\!]\|_{\pT^i}\\ 
		 &\quad + \big( \tilde h_T\big)^{\frac{1}{2}}\| [\![ \dn \GT(\ulinez_T)]\!]\|_{\pT^i}, \\
		\boldsymbol{\eta}_{T,\tan}& :=  \left( \hbar_T\right)^{\frac{1}{2}}\Big(\| [\![ \nabla \z_{\CTh}]\!] \times \n_T\|_{\pT}   + \| [\![\dnt \y_{\CTh}]\!] \|_{\pT}   + \| [\![\dtt \y_{\CTh}]\!]  \|_{\pT} \Big),\\
		\mathcal{O}_T(f) &:=\tilde{h}_T \|\f - \Pi_T^{k+1}(f)\|_T,
	\end{align*}
	where we have replaced $h_T$ by ${h_T}/{(k+2)^2}$ and ${h_T}/{(k+1)^2}$ in the stabilization terms $S^{\Delta}(\cdot, \cdot)$ and $S^{\nabla}(\cdot, \cdot)$, respectively, defined in \eqref{stab4th} and \eqref{stab2nd}, respectively. These parameters are used to ensure that the $hp$-analysis doesn't fall beyond the present scope. Similar parameter choices have been considered for DG schemes in \cite{MR2048235}, and for HHO methods applied to biharmonic and Poisson problems in \cite{dong2026hp} and \cite{dong2025hperroranalysismixedorderhybrid}, respectively
	
	We decompose the error $e=  (\ey, \ez) = (\y-\y_{\CTh}, \z-\z_{\CTh})$ into two parts as:
	\begin{equation}\label{cderrors}
     (\ey, \ez) = (\y-\y_c, \z-\z_c) + (\y_c-\y_{\CTh}, \z_c-\z_{\CTh}) := (\ey_c, \ez_c) + (\ey_d, \ez_d),
  \end{equation} 
  where $\ey_c\in V$ and $\ez_c \in \U$ denote the conforming error and $\ey_d$ and $ \ez_d$ denote non-conforming error.
  
   \subsection{Bound for conforming error}
   In this subsection we derive a bound for the conforming errors $\ey_c$ and $\ez_c$. By $\nabla_{\CTh} \cdot$ and $\nabla^2_{\CTh} \cdot$, we denote the broken gradient and broken Hessian operators acting cellwise on $H^1(\CTh)$ and $H^2(\CTh)$, respectively.
 \begin{lemma}\label{conferr in y and z}
 	For the conforming errors $\ey_c$ and $\ez_c$ defined in \eqref{cderrors}, we have
 	\begin{align}
 		\|\nabla^2 \ey_c\|^2 + \|\nabla \ez_c\|^2 &\lesssim\sum_{T\in\CTh}\Big( \boldsymbol{\eta}_{T,\res}^2 + \boldsymbol{\eta}_{T,\nor}^2 + \mathcal{O}_T(\f)^2 + (k+2) S_{\pT}^{\Delta}(\uliney_T, \uliney_T)\\
 	   \notag	& \qquad \qquad +  (k+1) \kappa \, S_{\pT}^{\nabla}(\ulinez_T, \ulinez_T )  + \|\nabla^2_{\CTh} \ey_d\|^2 + \|\ey_d\| ^2 \\
 	   \nonumber &\qquad \qquad + \|\nabla_{\CTh} \ez_d\|^2 + \|\ez_d\| ^2\Big).
 	\end{align}
 \end{lemma}
 \begin{proof}
 	Firstly, consider
 	\begin{align}\label{eyc1}
 		(\nabla^2 \ey_c,\nabla^2 \ey_c) - ( \ez_c, \ey_c) & = (\nabla^2_{\CTh} \ey,\nabla^2 \ey_c) - ( \ez,  \ey_c) - (\nabla^2_{\CTh} \ey_d,\nabla^2 \ey_c) + ( \ez_d, \ey_c).
 	\end{align}
 	
 	Now, we bound first two terms on the right hand side in the following manner by adding and subtracting the global discrete formulation:
 	\begin{align}
 		\label{conferry}
 		(\nabla^2 \ey,\nabla^2 \ey_c) \hspace{-0.45mm}-\hspace{-0.45mm} ( \ez,  \ey_c) &\hspace{-0.45mm}= 0 - (\nabla^2_{\CTh} \y_{\CTh}, \nabla^2 \ey_c) + (\z_{\CTh},  \ey_c) \\
 		\nonumber &\hspace{-0.45mm}= \sum_{T\in\CTh}\Big( (\z_T,  \ey_c)_T - (\nabla^2 \y_{T}, \nabla^2 \ey_c)_T \\
 		\nonumber & \qquad  \quad - (\nabla^2 \RT (\uliney_T), \nabla^2 \ey_c)_T \hspace{-0.45mm}+\hspace{-0.45mm} (\nabla^2 \RT (\uliney_T), \nabla^2 \ey_c)_T \Big)\\
 		\nonumber
 		&\hspace{-0.45mm}= \hspace{-0.45mm}\sum_{T\in\CTh} \hspace{-0.95mm} \Big( \hspace{-0.45mm} (\z_T,  \ey_c \hspace{-0.45mm}-\hspace{-0.45mm} \w_T)_T \hspace{-0.45mm}-\hspace{-0.45mm} (\nabla^2 \RT (\uliney_T), \nabla^2 (\ey_c \hspace{-0.45mm}-\hspace{-0.45mm} \RT(\ulinew_T)\hspace{-0.45mm})_T \\
 		\notag &\qquad \quad + (\nabla^2 (\RT (\uliney_T) - \y_T), \nabla^2 \ey_c)_T + S_{\pT}^{\Delta}(\uliney_T, \ulinew_T) \Big).
 	\end{align}
 	
 	Using similar arguments, we derive
 		\begin{align}\label{ezc1}
 		\kappa(\nabla \ez_c,\nabla \ez_c) + ( \ey_c, \ez_c) & = \kappa(\nabla \ez,\nabla \ez_c) + ( \ey,  \ez_c) - \kappa (\nabla_{\CTh} \ez_d,\nabla \ez_c) - ( \ey_d, \ez_c),
 	\end{align}
 	and
 	\begin{align}\label{conferrz}
 		\kappa(\nabla_{\CTh} \ez, \nabla \ez_c) + ( \ey,  \ez_c) &= (\f, \ez_c) - (\y_{\CTh},  \ez_c) - \kappa(\nabla_{\CTh} \z_{\CTh}, \nabla \ez_c) \\
 		\notag &\hspace{-1cm}=  \sum_{T\in\CTh}\Big( (\f \hspace{-0.3mm}-\hspace{-0.3mm} \y_T,  \ez_c \hspace{-0.3mm}-\hspace{-0.3mm} \u_T)_T \hspace{-0.3mm}-\hspace{-0.3mm} \kappa (\nabla \GT (\ulinez_T), \nabla ( \ez_c \hspace{-0.3mm}-\hspace{-0.3mm} \GT(\ulineu_T))_T \\
 		\notag & \quad \ + \kappa (\nabla (\GT (\ulinez_T) - \z_T), \nabla\ez_c)_T + S_{\pT}^{\nabla}(\ulinez_T, \ulineu_T) \Big).
 	\end{align}
 	By summing \eqref{eyc1} and \eqref{ezc1}, we arrive at 
 	\begin{align}\label{conferrexp}
 		\|\nabla^2 \ey_c\|^2 \hspace{-0.3mm}+\hspace{-0.3mm} \kappa\|\nabla \ez_c\|^2 & \hspace{-0.3mm}=\hspace{-0.3mm} (\nabla^2 \ey,\nabla^2 \ey_c) - ( \ez,  \ey_c) + \kappa(\nabla \ez,\nabla \ez_c) + ( \ey,  \ez_c) \\
 		\nonumber &\quad  - (\nabla^2_{\CTh} \ey_d,\nabla^2 \ey_c) + ( \ez_d, \ey_c) -\kappa (\nabla_{\CTh} \ez_d,\nabla \ez_c) - ( \ey_d, \ez_c).
 	\end{align}
 	First four terms on the right hand side can be approximated using \eqref{conferry} and \eqref{conferrz} as 
 	\begin{align*}
 		&(\nabla^2 \ey,\nabla^2 \ey_c) - ( \ez,  \ey_c) + \kappa(\nabla \ez,\nabla \ez_c) + ( \ey,  \ez_c)  \\
 		&= \sum_{T\in\CTh}\Big( (\z_T,  \ey_c - \w_T)_T + (\f - \y_T,  \ez_c - \u_T)_T - (\nabla^2 \RT (\uliney_T), \nabla^2 (\ey_c - \RT(\ulinew_T))_T \\
 		& \qquad \qquad + (\nabla^2 (\RT (\uliney_T) - \y_T), \nabla^2 \ey_c)_T  - \kappa (\nabla \GT (\ulinez_T), \nabla ( \ez_c- \GT(\ulineu_T))_T \\
 		& \qquad \qquad  + \kappa (\nabla (\GT (\ulinez_T) - \z_T), \nabla\ez_c)_T + S_{\pT}^{\Delta}(\uliney_T, \ulinew_T) + S_{\pT}^{\nabla}(\ulinez_T, \ulineu_T) \Big).
 	\end{align*}
 	
 	Now, we choose 
 	\[
 	\ulinew_h = \mathcal{J}_{\V,h}^k(\ey_c) \in \Vhok, \qquad 
 	\ulineu_h = \mathcal{J}_{\U,h}^k(\ez_c)\in \Uhok,
 	\]
 	and set 
 	\[
 	\zeta:= \ey_c - \mathcal{I}_{\text{mBS}}^{k+2}(\ey_c), \qquad 
 	\Upsilon:= \ez_c - \mathcal{I}_{\text{mKM}}^{k}(\ez_c).
 	\]
 	
 	By the definition of reconstruction operators, we have 
 	\begin{align*}
 		(\nabla^2 \RT (\uliney_T), \nabla^2 (\ey_c - \RT(\ulinew_T))_T &=  (\nabla^2 \RT (\uliney_T), \nabla^2\zeta)_T\\
 		&\quad \ + (\dnn \RT(\uliney_T), \dn \mathcal{I}_{\text{mBS}}^{k+2}(\ey_c)- \Pi_{\pT}^k(\dn\ey_c) )_{\pT},\\
 		(\nabla \GT (\ulinez_T), \nabla ( \ez_c- \GT(\ulineu_T))_T &= (\nabla \GT (\ulinez_T), \nabla\Upsilon)_T.
 	\end{align*}

 	On substituting these choices, we further have
 	\begin{align*}
 		(\nabla^2_{\CTh} \ey,\nabla^2 \ey_c) \hspace{-0.3mm}-\hspace{-0.3mm} ( \ez,  \ey_c) \hspace{-0.3mm}+\hspace{-0.3mm} \kappa(\nabla_{\CTh} \ez,\nabla \ez_c) + ( \ey,  \ez_c) &\hspace{-0.3mm}=\hspace{-0.3mm} \sum_{T\in\CTh} \hspace{-0.8mm}\Big( \hspace{-0.3mm}(\z_T, \zeta)_T \hspace{-0.3mm}+\hspace{-0.3mm} (\f \hspace{-0.3mm}-\hspace{-0.3mm} \y_T, \Upsilon)_T  \\
 		& \hspace{-6.3cm} - (\nabla^2 \RT (\uliney_T), \nabla^2(\zeta))_T 
 		-(\dnn \RT(\uliney_T), \dn \mathcal{I}_{\text{mBS}}^{k+2}(\ey_c)- \Pi_{\pT}^k(\dn\ey_c) )_{\pT}\\
 	& \hspace{-6.3cm}	+ (\nabla^2 (\RT (\uliney_T) - \y_T), \nabla^2 \ey_c)_T  - \kappa (\nabla \GT (\ulinez_T), \nabla\Upsilon)_T \\
 	& \hspace{-6.3cm}	+ \kappa (\nabla (\GT (\ulinez_T) \hspace{-0.2mm} - \hspace{-0.2mm} \z_T), \nabla\ez_c)_T \hspace{-0.2mm} + \hspace{-0.2mm} S_{\pT}^{\Delta}(\uliney_T, \mathcal{J}_{\V,h}^k(\ey_c) ) \hspace{-0.2mm} + \hspace{-0.2mm} S_{\pT}^{\nabla}(\ulinez_T, \mathcal{J}_{\U,h}^k(\ez_c)) \Big).
 	\end{align*}
 	
  Using integration by parts and the fact that $S_{\pT}^{\nabla}(\ulinez_T, \mathcal{J}_{\U,h}^k(\ez_c))=0$,
  \begin{align*}
  	(\nabla^2_{\CTh} \ey,\nabla^2 \ey_c) - ( \ez,  \ey_c) + \kappa(\nabla_{\CTh} \ez,\nabla \ez_c) + ( \ey,  \ez_c) &= \sum_{T\in\CTh}\Big( (\z_T- \Delta^2 \RT (\uliney_T),  \zeta)_T  \\
  	&  \hspace{-4cm} + (\f - \y_T + \kappa\Delta \GT(\ulinez_T),  \Upsilon)_T - (\dn\Delta \RT (\uliney_T), \zeta)_{\pT^i} \\
   &  \hspace{-4cm}	- (\dnt \RT (\uliney_T), \dt\zeta)_{\pT^i} \hspace{-0.3mm}-\hspace{-0.3mm} (\dnn \RT(\uliney_T), \dn \ey_c \hspace{-0.3mm}-\hspace{-0.3mm} \Pi_{\pT}^k(\dn\ey_c) )_{\pT}\\
  	&  \hspace{-4cm} + (\nabla^2 (\RT (\uliney_T) \hspace{-0.3mm}-\hspace{-0.3mm} \y_T), \nabla^2 \ey_c)_T \hspace{-0.3mm}-\hspace{-0.3mm} \kappa (\dn \GT (\ulinez_T), \Upsilon)_{\pT^i} \\
  	&  \hspace{-4cm}  + \kappa (\nabla (\GT (\ulinez_T) - \z_T), \nabla\ez_c)_T + S_{\pT}^{\Delta}(\uliney_T, \mathcal{J}_{\V,h}^k(\ey_c) ) \Big).
  \end{align*}
 The term $(\dnn \RT(\uliney_T), \dn \ey_c- \Pi_{\pT}^k(\dn\ey_c) )_{\pT}$ vanishes using the definition of $L^2$-projection. Also, $\zeta,\ \dt \zeta$ and $\Upsilon$ are  single-valued on every mesh interface. Therefore, invoking Cauchy-Schwarz inequality, we conclude that
 \begin{align*}
 	&|(\nabla^2_{\CTh} \ey,\nabla^2 \ey_c) - ( \ez,  \ey_c) + \kappa(\nabla_{\CTh} \ez,\nabla \ez_c) + ( \ey,  \ez_c)| \\
 	& \lesssim \sum_{T\in\CTh}\Big( \| \z_T- \Delta^2 \RT (\uliney_T)\|_T \| \zeta\|_T + \|\f \hspace{-0.2mm} -\hspace{-0.2mm} \y_T + \kappa\Delta \GT(\ulinez_T)\|_T \|\Upsilon\|_T \\
     & \qquad \qquad  + \|\nabla^2 (\RT (\uliney_T) \hspace{-0.2mm} - \hspace{-0.2mm} \y_T)\|_T \|\nabla^2 \ey_c\|_T + \kappa \|\nabla (\GT (\ulinez_T) - \z_T)\|_T \|\nabla\ez_c\|_T  \\
 	  &  \qquad \qquad - \|[\![\dnt \RT (\uliney_T)]\!]\|_{\pT^i} \|\dt\zeta\|_{\pT^i} + \|[\![\dn\Delta \RT (\uliney_T)]\!]\|_{\pT^i} \|\zeta\|_{\pT^i}  \\
 	  & \qquad \qquad  + \kappa \|[\![\dn \GT (\ulinez_T)]\!]\|_{\pT^i} \|\Upsilon\|_{\pT^i} + S_{\pT}^{\Delta}(\uliney_T, \uliney_T )^{\frac{1}{2}}  S_{\pT}^{\Delta}(\mathcal{J}_{\V,h}^k(\ey_c), \mathcal{J}_{\V,h}^k(\ey_c) )^{\frac{1}{2}}  \Big).
 \end{align*}
 
 Using the approximation properties \eqref{approx mbs}, \eqref{approx mkm}, \cite[Lemma~3.1]{dong2025hperroranalysismixedorderhybrid} and \cite[Lemmas~3.1, 3.5]{dong2026hp}, we obtain the bounds
 \begin{align*}
 	\|\nabla (\GT (\ulinez_T) - \z_T)\|_T^2 &\lesssim S_{\pT}^{\nabla}(\ulinez_T, \ulinez_T), &&
 	\|\nabla^2 (\RT (\uliney_T) - \y_T)\|_T^2 \lesssim  S_{\pT}^{\Delta}(\uliney_T, \uliney_T),\\ S_{\pT}^{\Delta}(\mathcal{J}_{\V,h}^k(\ey_c), \mathcal{J}_{\V,h}^k(\ey_c) ) &\lesssim (k+2) \| \nabla^2 \ey_c\|_T^2,
 \end{align*}
 which further yield
 \begin{align*}
 	&|(\nabla_{\CTh}^2 \ey,\nabla^2 \ey_c) - ( \ez,  \ey_c) + \kappa(\nabla_{\CTh} \ez,\nabla \ez_c) + ( \ey,  \ez_c)| \\
 	& \lesssim \Bigg( \hspace{-0.5mm}\sum_{T\in\CTh}\Big( \hbar_T^4\| \z_T \hspace{-0.3mm}-\hspace{-0.3mm} \Delta^2 \RT (\uliney_T)\|_T^2 \hspace{-0.3mm}+\hspace{-0.3mm} \tilde h_T^2\|\f - \y_T \hspace{-0.3mm}+\hspace{-0.3mm} \kappa\Delta \GT(\ulinez_T)\|_T^2
 	+   \kappa S_{\pT}^{\nabla}(\ulinez_T, \ulinez_T )\\
 	 &\qquad \qquad \quad + \hbar_T^3\|[\![\dn\Delta \RT (\uliney_T)]\!]\|_{\pT^i}^2  +\hbar_T \|[\![\dnt \RT (\uliney_T)]\!]\|_{\pT^i}^2  + (k+2) S_{\pT}^{\Delta}(\uliney_T, \uliney_T ) \\
     & \qquad \qquad \quad + \kappa \tilde h_T \|[\![\dn \GT (\ulinez_T)]\!]\|_{\pT^i}^2   \Big)\Bigg)^{\frac{1}{2}}\big(\|\nabla^2 \ey_c\|^2+ \kappa \|\nabla\ez_c\|^2\big)^{\frac{1}{2}}.
 \end{align*}
 
 Using expression for $\boldsymbol{\eta}_{T,\res},\boldsymbol{\eta}_{T,\nor}$ and triangle inequality in second term, we get 
  \begin{align*}
 	&|(\nabla_{\CTh}^2 \ey,\nabla^2 \ey_c) - ( \ez,  \ey_c) + \kappa(\nabla_{\CTh} \ez,\nabla \ez_c) + ( \ey,  \ez_c)| \\
 	&\lesssim \Bigg( \sum_{T\in\CTh}\Big( \boldsymbol{\eta}_{T,\res}^2 + \boldsymbol{\eta}_{T,\nor}^2 +  \kappa S_{\pT}^{\nabla}(\ulinez_T, \ulinez_T ) + (k+2) S_{\pT}^{\Delta}(\uliney_T, \uliney_T ) \\
   &\qquad \qquad \quad + \mathcal{O}_T(\f)^2   \Big)\Bigg)^{\frac{1}{2}}\big(\|\nabla^2 \ey_c\|^2+ \kappa \|\nabla\ez_c\|^2\big)^{\frac{1}{2}}.
 \end{align*}
 
 The rest four terms in \eqref{conferrexp} are bounded using Cauchy-Schwarz inequality and Poincar\'e inequality. Combining this with the above estimate concludes the proof.
 \end{proof}
 \subsection{Bound for non-conforming error} 
 We start by defining conforming parts $\y_c\in \V$ and $\z_c\in \U$ of $\y_{\CTh}$ and $\z_{\CTh}$, respectively. The construction of $\z_c$ follows from the idea presented in \cite[Sec.~5.1.2]{dong2025hperroranalysismixedorderhybrid}. We solve a minimization problem in $H_0^1(\omega_{\a})$, where  $\omega_{\a}$ is the star associated with the vertex $\a$. Let $\CF_{\a}$ and $\bar\CF_{\a}$ denotes the set of mesh faces that contain the vertex $\a$ and the set of mesh faces on $\partial \omega_\a$ that do not contain $\a$, respectively. Let $\psi_{\a}$ be a hat function that is equal to $1$ at $\a$ and has support in $\omega_\a$. The hat function satisfies the partition of unity property
 \begin{equation*}
 	\sum_{\a\in \mathcal{V}_h} \psi_{\a} = 1 \quad \text{on } \Omega.
 \end{equation*}
 
 At each vertex patch $\omega_\a$, we define $\z_c^{\a} \in H_0^1(\omega_\a)$ as the unique solution of the minimization problem
 \begin{equation*}
 	\z_c^{\a} := \arg\min\limits_{\rho_{\a}\in H_0^1(\omega_a)} \|\nabla\rho_{\a}- \nabla_{\CTh}(\psi_{\a}\z_{\CTh})\|_{\omega_\a}.
 \end{equation*}
 
 Then, by extending $\z_c^{\a}$ by zero to $\Omega$, we set $ 	\z_c := \sum_{\a\in \mathcal{V}_h}\z_c^{\a} . $
	\begin{lemma}[Non-conforming error in $\z$] \label{non conforming error in z}The following estimate holds true:
		\begin{align}
			\|\nabla_{\CTh} \ez_d\| ^2 +  \|\ez_d\|^2 &\lesssim \sum_{T\in\CTh} S^{\nabla}_{\pT} (\ulinez_T, \ulinez_T) + \tilde{h}_T\| [\![\nabla\z_{\CTh}]\!]\times \n_T\|_{\pT}.
		\end{align}
	\end{lemma}
\begin{proof}
	To prove this, first note that 
	\begin{equation}
		\ez_d = \z_c -\z_{\CTh} = \sum_{\a\in \mathcal{V}_h} \z_c^{\a} - \psi_\a \u_{\CTh} := \sum_{\a\in \mathcal{V}_h} \ez_{d,\a}.
	\end{equation}
	
	Since $\nabla_{\CTh} \ez_{d,\a}\in \L^2(\omega_\a)$, we invoke the Helmholtz decomposition 
	\begin{equation*}
		\nabla_{\CTh} \ez_{d,\a} = \nabla \xi + \textbf{curl}\boldsymbol{\phi}  \quad \text{in} \ \omega_\a,
	\end{equation*}
	where $\xi \in H_0^1(\omega_\a)$ and $\boldsymbol{\phi}\in \boldsymbol{H}^1(\omega_\a)$. Now, following the arguments developed in \cite[Lemma~5.3]{dong2025hperroranalysismixedorderhybrid}, we obtain
\begin{align} \label{H1 nonconferr in z}
	\| \nabla_{\CTh} \ez_{d,\a} \|^2_{\omega_a} \lesssim \sum_{T\subset \omega_\a} S^{\nabla}_{\pT}(\ulinez_T, \ulinez_T) + \tilde{h}_T\| [\![ \nabla \z_{\CTh}]\!] \times \n_T\|_{\pT}^{2}.
	\end{align}
	
	The bound for $	\|\ez_{d,\a} \|^2_{\omega_a}$ follows from the broken Poincar\'e-Friedrichs inequality \cite[Corollary~4.1]{MR4044442} (see \cite{MR1974504} for more details)
	\begin{align*}
		\|\ez_{d,\a}\|_{\omega_\a} \lesssim  h_{\omega_\a}\Bigg(\|\nabla_{\CTh} \ez_{d,\a}\|_{\omega_\a} + \Bigg(\sum_{F\in\CF_{\a}}h_F^{-1}\|\Pi_{\pT}^0 [\![ \ez_{d,\a}]\!]\|^2_{F} \Bigg)^{{1}/{2}}\Bigg).
	\end{align*}
	
	Using the stability of orthogonal $L^2$-projection, mesh shape regularity and the bound $\|\psi_\a\|_{\infty, \omega_a}\le 1$, we deduce that 
		\begin{align} \label{L2 nonconferr in z}
	\sum_{\a \in \mathcal{V}_h}	\|\ez_{d,\a}\|_{\omega_\a} &\lesssim  \sum_{\a \in \mathcal
		{V}_h} h_{\omega_\a}\Bigg(\|\nabla_{\CTh} \ez_{d,\a}\|_{\omega_\a} + \Bigg(\sum_{F \in \CF_{\a}}h_F^{-1}\| [\![ \z_{\CTh}]\!]\|^2_{F} \Bigg)^{\frac{1}{2}}\Bigg).\\
       	& \lesssim   \sum_{\a \in \mathcal{V}_h}h_{\omega_\a}\|\nabla_{\CTh} \ez_{d,\a}\|_{\omega_\a} +  \sum_{T\in \CTh} S^{\nabla}_{\pT}(\ulinez_T, \ulinez_T)^{\frac{1}{2}}, \notag
	\end{align}
	where in the last step, after adding and subtracting $\z_F$, we have used triangle inequality  followed by definition of $\S^{\nabla}(\cdot, \cdot)$. Since, 
	$$	\|\nabla_{\CTh} \ez_d\|^2  +  \|\ez_d\|^2  \le (d+1)\sum\limits_{\a \in \mathcal{V}_h}( 	\|\nabla_{\CTh} \ez_{d,\a}\|^2_{\omega_\a} + 	\|\ez_{d,\a}\|^2_{\omega_\a} ).$$
	
	 Therefore, combining this with bounds \eqref{H1 nonconferr in z} and \eqref{L2 nonconferr in z} concludes the proof.
\end{proof}

Now, to bound the non-conforming error $\ey_d$, we follow the idea developed in \cite[Sec.~5.3]{dong2026hp}. For this, we consider $C^1$-composite finite element with polynomial order five on the Alfeld split of the mesh $\CTh$ introduced in \cite{Walk14}. Since, the Alfeld split consists of subdividing each mesh cell $T\in\CTh$ into $(d+1)$ sub-simplices by connecting each vertex of $T$ to its centroid. 

We denote a generic sub-cell obtained in this manner by $\tilde{T}$, and the global resulting sub-mesh by $\ACTh$. Let $\phi_{\a}$ denote the $C^1$-basis function associated with vertex $\a\in \mathcal{V}_h$ and having support in $\omega_\a$. Also, let $\phi_T$ be the basis function associated with $T$ and having support inside $T$. These basis functions satisfy the following property:
\begin{equation}\label{pouy}
	\sum_{\a \in \mathcal{V}_h} \phi_{\a} + \sum_{T \in \CTh} \phi_{T} = 1 \quad \text{on } \Omega.
\end{equation}
 At each vertex patch $\omega_\a$, we define $\y_c^{\a} \in H_0^2(\omega_\a)$ as the unique solution of the minimization problem
\begin{equation*}
	\y_c^{\a} := \arg\min\limits_{\rho_{\a}\in H_0^2(\omega_a)} \|\nabla^2\rho_{\a}- \nabla^2_{\CTh}(\phi_{\a}\y_{\CTh})\|_{\omega_\a}.
\end{equation*}
Then by extending $\y_c^{\a}$ by zero to $\Omega$, we set 
\begin{equation}\label{yc construction}
	\y_c := \sum_{\a \in \mathcal{V}_h} \y_c^{\a} + \sum_{T \in \CTh} \phi_{T} \y_{\CTh} \in H_0^2(\Omega) \subset \V.
\end{equation}
{To bound the non-conforming error in $\y$, we use the following Poincar\'e type inequality obtained from Poincar\'e inequality for piecewise $H^1$-functions, combined with an integration by parts argument and a subsequent use of the broken Poincar\'e inequality.}
\begin{remark}\label{poincareH2}
Let $\v\in \{ \v \in L^2(\omega_{\a}) \ |\ \v|_T \in H^2(T),\ v|_F=0\ \  \forall F\in \bar{\CF}_{\a} \backslash \CF_{\a}\}$. Then the following holds:
\begin{align} 
	\|\v\|^2_{\omega_{\a}} \lesssim \sum_{T\subset \omega_{\a}}\|\nabla^2\v\|^2_{T} + \sum_{F\in \CF_{\a}}\Big( h_F^{-1} \|[\![\dn \v]\!]\|_F^2 + h_F^{-3} \|[\![\v]\!]\|_F^2  \Big).
\end{align}
\end{remark}
	\begin{lemma}[Non-conforming error in $\y$]\label{non conforming error in y} It holds that, 
	\begin{align*}
			\|\nabla^2_{\CTh} \ey_d\|^2 + \|\ey_d\|^2 & \lesssim \sum_{T \in \CTh} \Big( \hbar_T\big( \| [\![\dnt \y_{\CTh}]\!] \|_{\pT}^2 \hspace{-0.35mm}+\hspace{-0.35mm} \| [\![\dtt \y_{\CTh}]\!]  \|_{\pT}^2\big) \\
			&\qquad \quad+ S^{\Delta}_{\pT}(\uliney_T,\uliney_T) +  \big( \hbar_T\big)^{-\frac{1}{2}} \|\dn\y_{\CTh}\|_{\pT^b}\Big).
	\end{align*}
	\end{lemma}
	\begin{proof}
		Firstly, using the partition of unity \eqref{pouy}, we get
		\begin{equation*}
			\ey_d = \y_c -\y_{\CTh} = \sum_{\a\in \mathcal{V}_h} \y_c^{\a} - \phi_\a \y_{\CTh} =: \sum_{\a\in \mathcal{V}_h} \ey_{d,\a}.
		\end{equation*}
		
	Now, it follows from \cite[Lemma~5.4]{dong2026hp} by invoking Helmholtz decomposition that 
			\begin{align}
			\|\nabla^2_{\CTh} \ey_{d,\a}\|^2 &\lesssim \sum_{T \subset \omega_\a}  \hbar_T\big( \| [\![\dnt \y_{\CTh}]\!] \|_{\pT}^2   + \| [\![\dtt \y_{\CTh}]\!]  \|_{\pT}^2\big) + S^{\Delta}_{\pT}(\uliney_T,\uliney_T).
		\end{align}
	{{The bound for  $\|\ey_{d,\a}\|^2$ follows by Remark~\ref{poincareH2} and arguments developed in Lemma~\ref{non conforming error in z}. By summing these bounds over each vertex $\a \in \mathcal{V}_h$, we conclude the proof}}.
	\end{proof}
	\begin{theorem}[Reliability]\label{reliability}
		For  $k\ge1$, the following estimate holds true
		\begin{align}\label{main_result}
			\|\nabla_{\CTh}^2\ey\|^2 + \|\nabla_{\CTh} \ez\|^2 \lesssim (\e_{\boldsymbol{\eta}})^2,
		\end{align}
		where the global estimator is defined as:
		\[
		(\e_{\boldsymbol{\eta}})^2 := \sum_{T\in \CTh} \Big( \boldsymbol{\eta}^2_{T,\res} + \boldsymbol{\eta}^2_{T,\nor} + \boldsymbol{\eta}^2_{T,\tan} + (k+2)\boldsymbol{\eta}^2_{T,\sta} +  O_T^2(\f) \Big).
		\]
	\end{theorem}
	\begin{proof}
		The proof follows by combining Lemmas~ \ref{conferr in y and z}, \ref{non conforming error in z} and \ref{non conforming error in y}.
	\end{proof}
	\begin{remark} The assumption $k\ge1$ in Sec.~\ref{Interpolation operators and estimator} is needed to invoke the modified Karkulik-Melenk interpolation operator. For the case $k=0$, we can use nodal averaging operator {\cite[Thm.~2.2]{MR2034620}} to bound the non-conforming error in $\z$.
	\end{remark}
	\begin{remark}
		The a posteriori error estimates discussed in Sec.~\ref{A posteriori error estimates} can be extended for the Cahn-Hilliard type boundary conditions by suitable modification in the construction of conforming parts and $hp$-interpolation operators. For each vertex $\a\in \mathcal{V}_h$, we define 
		\begin{subequations}
			\begin{align*}
			\z_c^\a &:= \arg \min_{{\rho_\a} \in H^1_{\star}(\omega_\a)} \| \nabla \rho_\a -\nabla_{\CTh}(\psi_{\a}\z_{\CTh})\|_{\omega_\a},\\
			\y_c^\a &:= \arg \min_{{\rho_\a} \in H^2_{\star}(\omega_\a)} \| \nabla^2 \rho_\a -\nabla^2_{\CTh}(\phi_{\a}\y_{\CTh})\|_{\omega_\a},
		\end{align*}
		\end{subequations}
		where
		\begin{subequations}
			\begin{align*}
				H^1_{\star}(\omega_\a) &: = \{\v \in H^1(\omega_\a)\ |\ \v|_{\partial\omega_\a \cap \Omega} = 0\}\\
				H^2_{\star}(\omega_\a) &: = \{\v \in H^2(\omega_\a)\ |\ \dn\v|_{\partial\omega_\a} = 0,\ \v|_{\partial\omega_\a \cap \Omega} = 0\}.
			\end{align*}
		\end{subequations}
			Then, by extending $\y_c^\a$ and $\z_c^\a$ to zero outside $\omega_\a$, we set 
			\begin{align*}
				\z_c &:=\tilde\z_c - |\Omega|^{-1}(\tilde\z_c,1), &&	\y_c :=\tilde\y_c - |\Omega|^{-1}(\tilde\y_c,1),
			\end{align*}
		where 
		\begin{align*}
			\tilde\z_c &:=\sum_{\a \in \mathcal{V}_h}\z_c^\a, &&\tilde\y_c := \sum_{\a \in \mathcal{V}_h} \y_c^\a +  \sum_{T \in \CTh}\phi_T\y_{\CTh}.
		\end{align*}	
		
	\end{remark}
	\section{Numerical experiments}\label{Numerical experiments}
	In this section, we present some numerical examples to validate the theoretical results derived in Secs.~\ref{HHO Scheme}, \ref{C0 HHO Scheme} and \ref{A posteriori error estimates}. The computational domain $\O$ is partitioned into a mesh $\mathcal{T}_h$ composed of rectangles and polygonal elements for the fully non-conforming HHO method, and a simplicial discretization consisting of triangles for $C^0$-conforming HHO scheme. The performance of two approaches is investigated for both boundary conditions, allowing for a comparative assessment of the accuracy, robustness, and computational efficiency across varying mesh configurations. We denote the errors as:
	\begin{align*}
		\e_{Z} &:=  \left(\sum_{T\in\CTh} (\|\nabla^2(\y-\RT(\uliney_T))\|_T^2 + \|\nabla(\z-\GT(\ulinez_T))\|_T^2) \right)^{1/2} ,\\
		\e_{Z} & := \left(\sum_{T\in\CTh} \big(\|\nabla^2(\y-\RTc(\uliney_T))\|_T^2 + \|\nabla(\z-\z_T)\|_T^2\big) \right)^{1/2},
	\end{align*}
	for fully non-conforming HHO and $C^0$-HHO schemes, respectively.
	\subsection{Accuracy test}\label{Example 5.1}
	\textbf{Case~1:} In the first example, we consider the unit square domain $\Omega=(0,1)^2$ with $\kappa = 1, 0.001, 0.00001$ and design the source function $\f$ in such a way that the analytical solution is
$
		\y(\x) = x_1^{6}(1-x_1)^{6}x_2^{6}(1-x_2)^{6},
$
	and impose the simply supported boundary conditions \eqref{boundcndn1} for both the unknowns. 
	
	\textbf{Case~2:} Next, consider the second type of boundary conditions \eqref{boundcndn2} and manufacture the load $\f$ such that
$
		\y(\x) =  \beta(\sin^{4}(2 \pi x_1) \sin^{4}(2 \pi x_2)-9/64),
$
	where the parameter $\beta = 10^{-6}$ is a re-scaled constant and the mobility coefficient $\kappa = 1, 0.01, 0.0001$. 
	
	For Case~1 (Figs.~\ref{fig:simplyrectkappa1}, \ref{fig:simplyrectkappa2}, \ref{fig:simplyrectkappa3}, \ref{fig:simplycokappa1}, \ref{fig:simplycokappa2} and \ref{fig:simplycokappa3}) and Case~2 (Figs.~\ref{fig:cahnHHOkappa1}, \ref{fig:cahnHHOkappa2}, \ref{fig:cahnHHOkappa3}, \ref{fig:cahncokappa1}, \ref{fig:cahncokappa2} and \ref{fig:cahncokappa3}), we observe that the proposed schemes presented in Sec.~\ref{HHO Scheme} and Sec.~\ref{C0 HHO Scheme} exhibit a robust and consistent performance across the considered values of mobility parameter $\kappa$. In particular, both schemes attain the expected convergence of $O(h)$ and $O(h^2)$ for the primary variables with polynomial degrees $k=0$ and $k=1$, respectively. These results remain essentially unaffected by variations in $\kappa$, demonstrating the robustness of the methods in low-mobility regimes. Moreover, Figs.~\ref{fig:simplyrectkappa1}, \ref{fig:simplyrectkappa2} and \ref{fig:simplyrectkappa3} reveal a close correspondence between the estimator $\e_{\boldsymbol{\eta}}$ and the error $\e_{Z}$, as indicated by their nearly parallel alignment. The effectivity index in Figs.~\ref{fig:effrectkappa1}, \ref{fig:effrectkappa2} and \ref{fig:effrectkappa3} remains stable across all values of $\kappa$, providing further numerical evidence for the reliability and parameter robustness of the proposed estimator.

	\begin{figure}
	\centering
		\centering
		\subfloat[ Error history \label{fig:simplyrectkappa1}]{\includegraphics[width=0.49\textwidth]{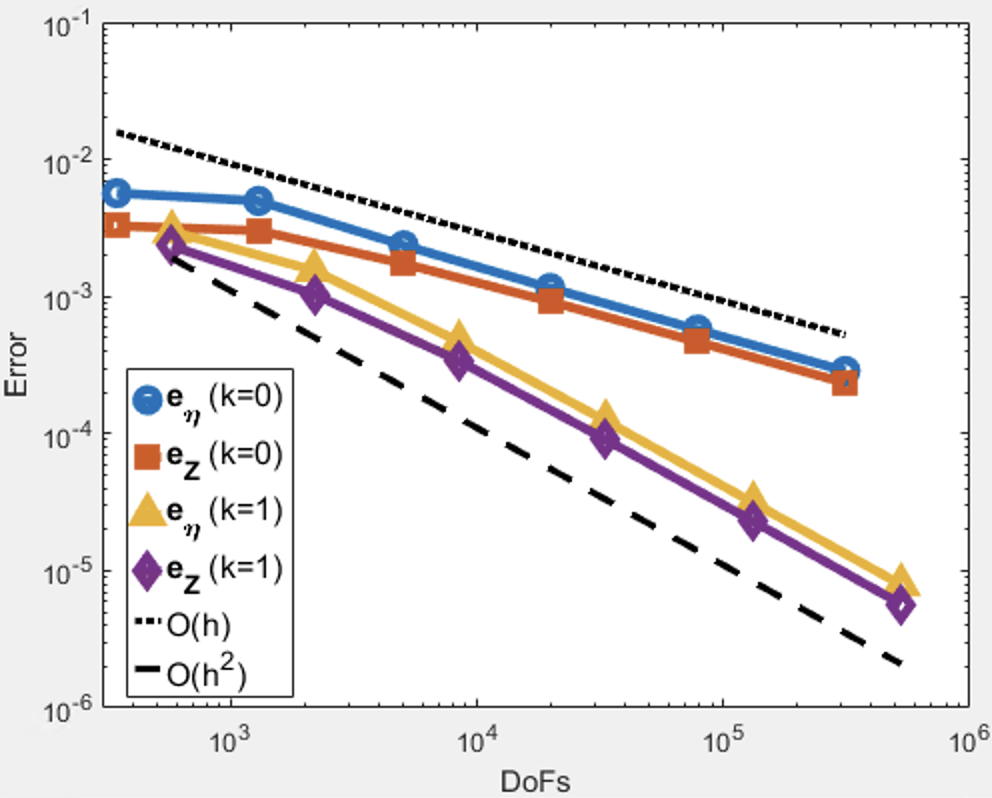}}
		\hfill
		\subfloat[Effectivity index\label{fig:effrectkappa1}]{\includegraphics[width=0.49\textwidth]{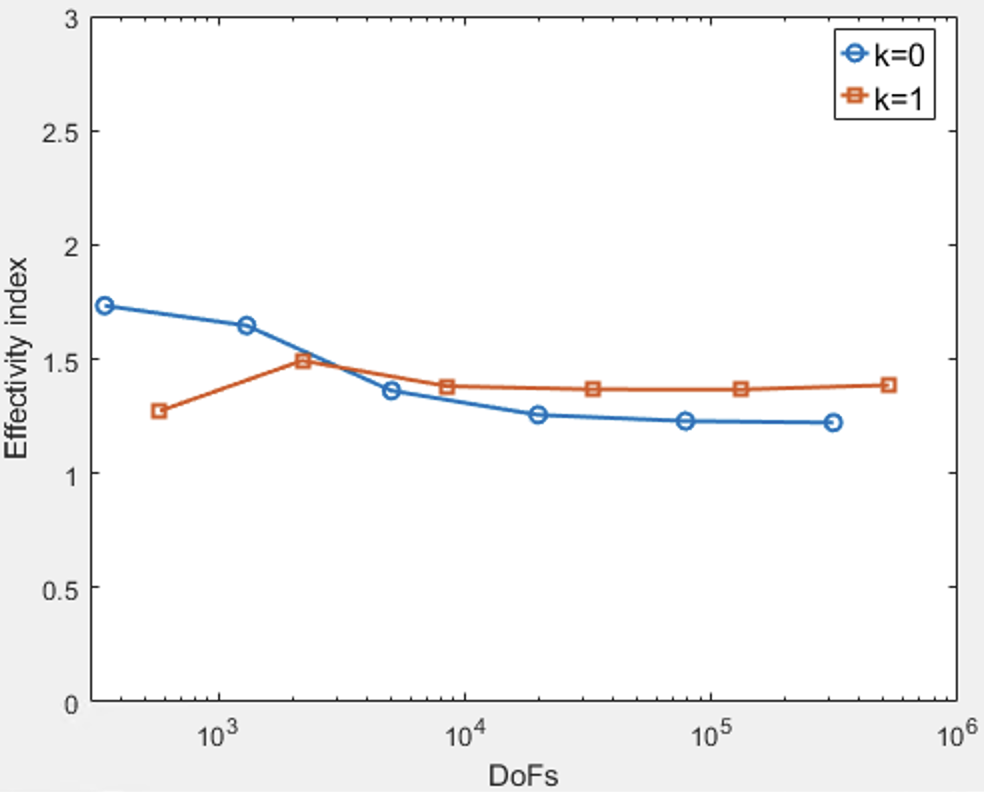}}\\
		\subfloat[ $\y_{\mathcal{T}_h}$\label{fig:simplyrectkappa1Yh}]{\includegraphics[width=0.49\textwidth]{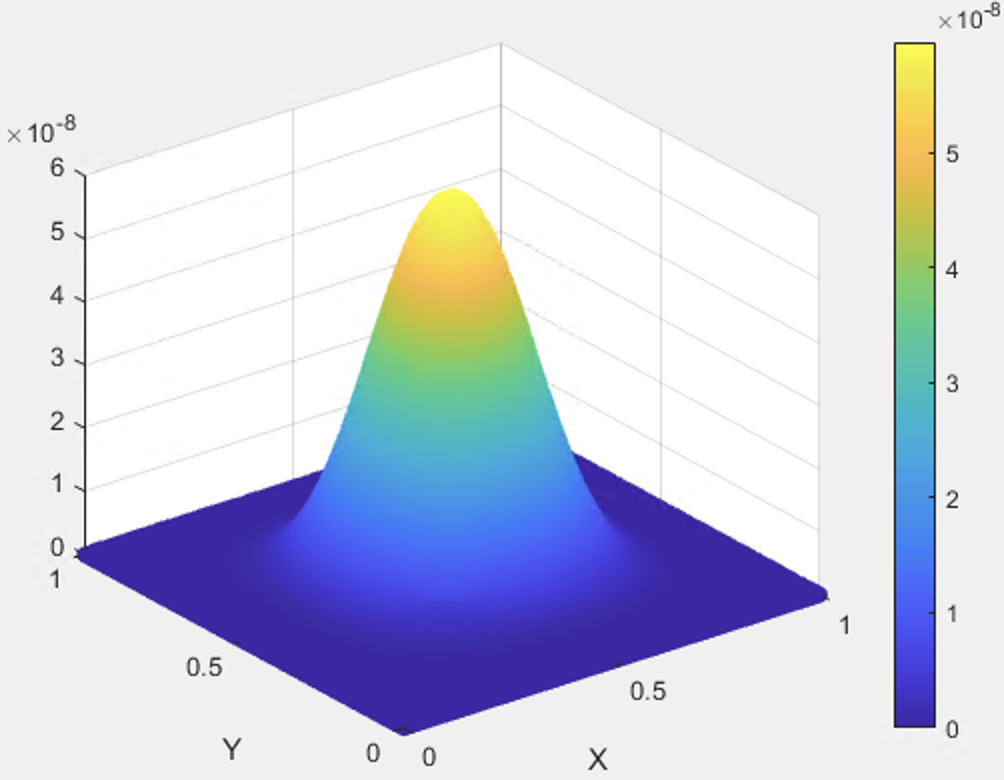}}
		\hfill
		\subfloat[ $\z_{\mathcal{T}_h}$\label{fig:simplyrectkappa1Zh}]{\includegraphics[width=0.49\textwidth]{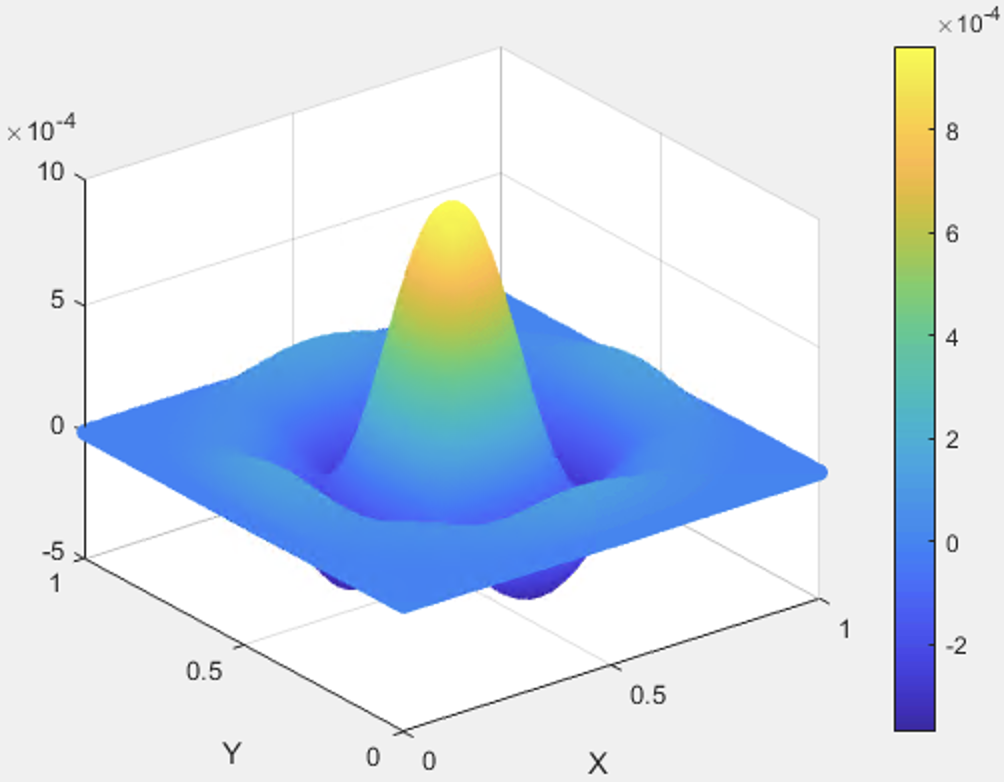}}
	\vspace{-0.5mm}
	\caption{Example~\ref{Example 5.1} (Case~1): Results for $\kappa = 1$ using non-conforming HHO method and approximate solutions for $k=1$ on a rectangular mesh with $1024$ elements.}
	\label{FIGURE ex1_HHO_case1}
\end{figure}
	\begin{figure}
	\centering
	\subfloat[ Error history\label{fig:simplyrectkappa2}]{\includegraphics[width=0.49\textwidth]{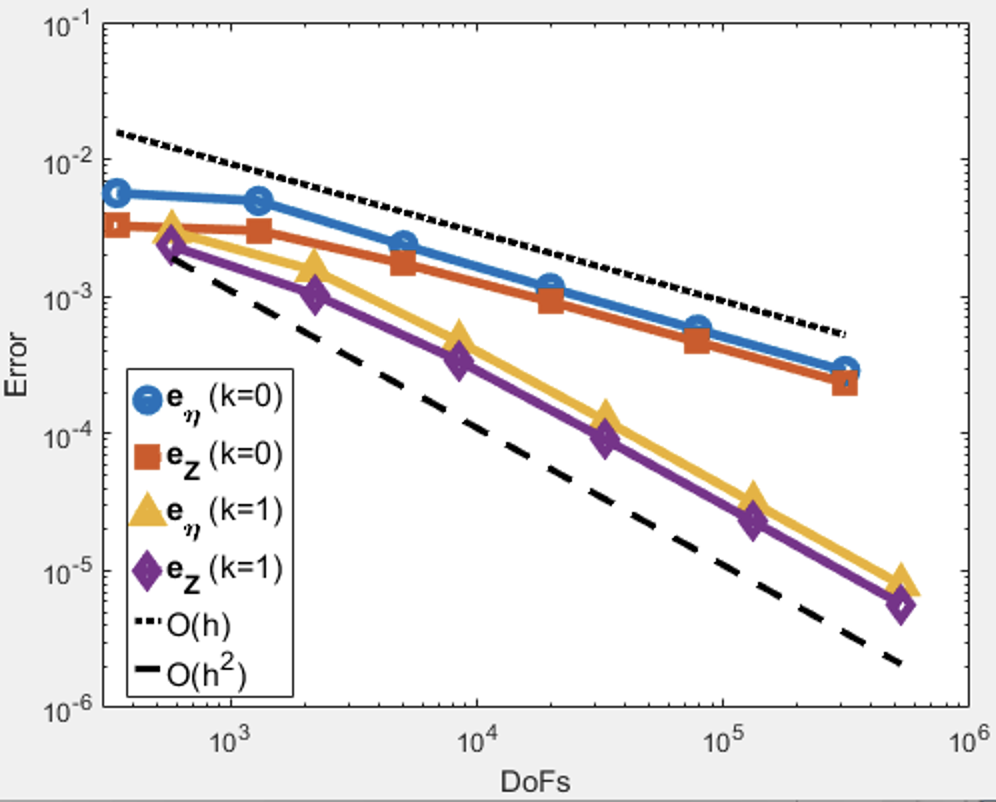}}
	\hfill
	\subfloat[Effectivity index\label{fig:effrectkappa2}]{\includegraphics[width=0.49\textwidth]{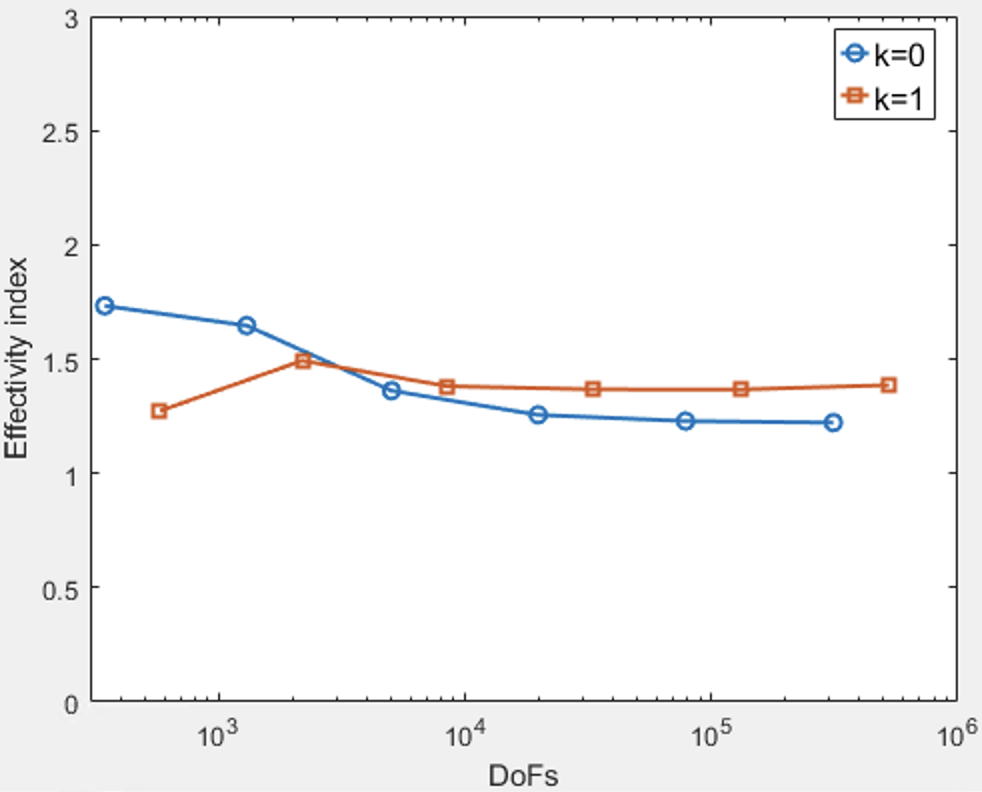}}\\
	\vspace{-0.5mm}
	\caption{Example~\ref{Example 5.1} (Case~1): Computational results for $\kappa = 10^{-3}$ using non-conforming HHO method  on rectangular meshes.}
	\label{FIGURE ex1_HHO_case2}
\end{figure}
	\begin{figure}
	\centering
	\subfloat[ Error history\label{fig:simplyrectkappa3}]{\includegraphics[width=0.49\textwidth]{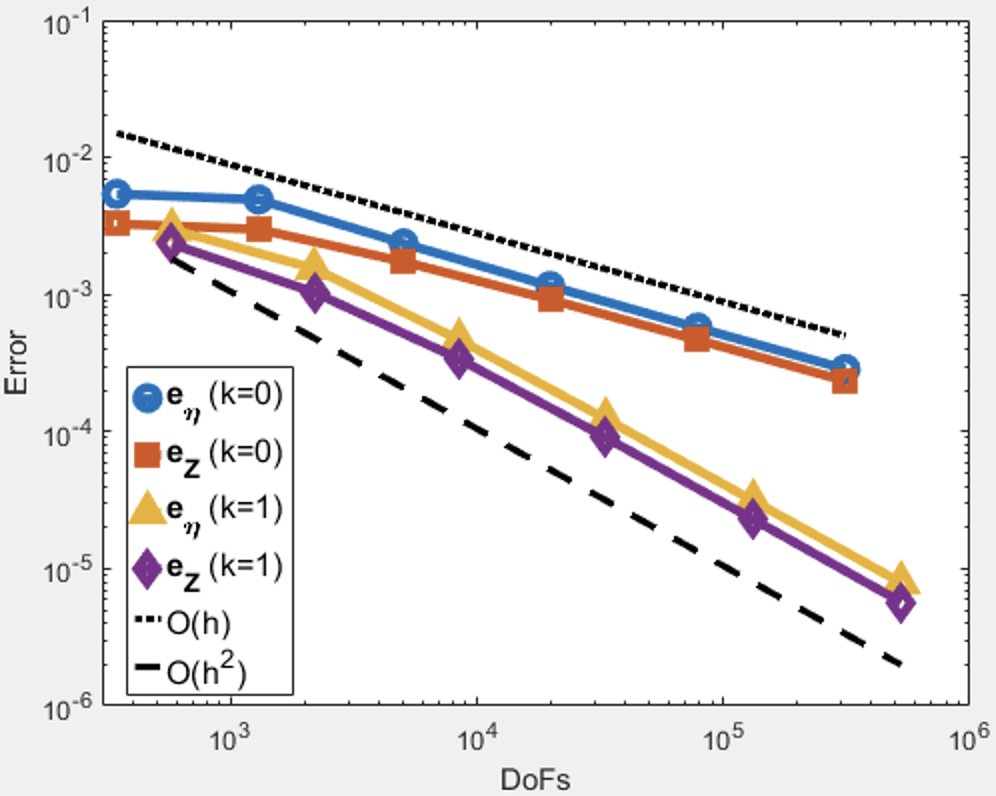}}
	\hfill
	\subfloat[Effectivity index\label{fig:effrectkappa3}]{\includegraphics[width=0.49\textwidth]{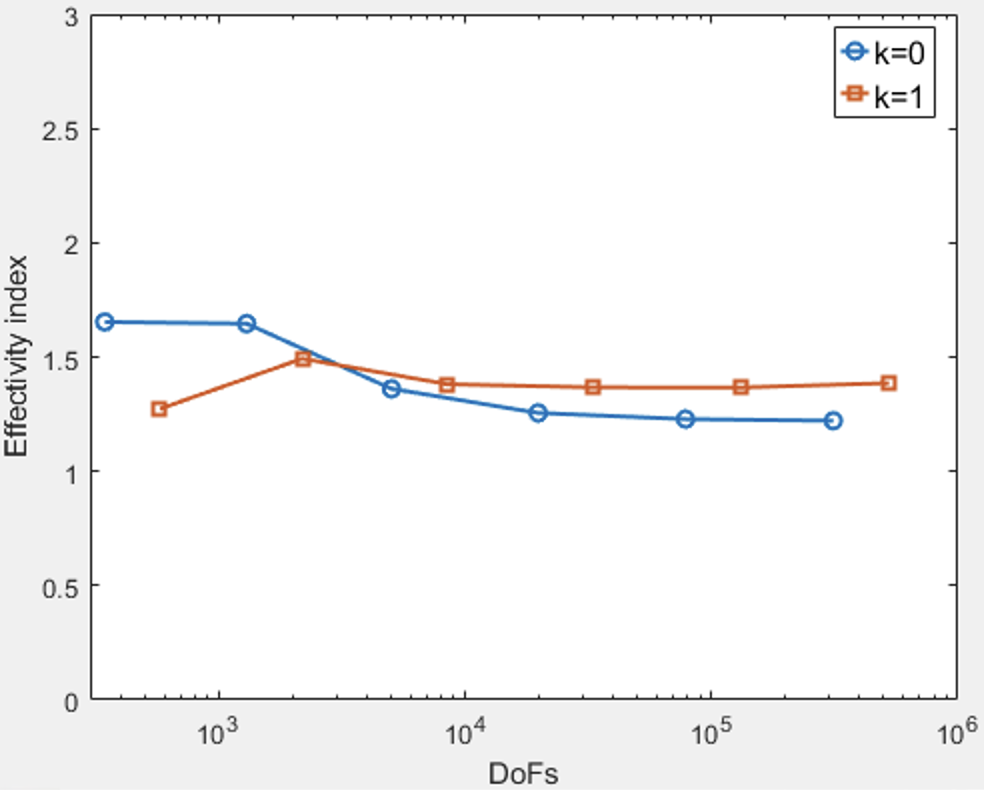}}\\

	\vspace{-0.5mm}
	\caption{Example~\ref{Example 5.1} (Case~1): Numerical results for $\kappa = 10^{-5}$ using non-conforming HHO method on rectangular meshes.}
	\label{FIGURE ex1_HHO_case3}
\end{figure}

\begin{figure}
	\centering
\subfloat[ $\kappa = 1$ \label{fig:simplycokappa1}]{\includegraphics[width=0.49\textwidth]{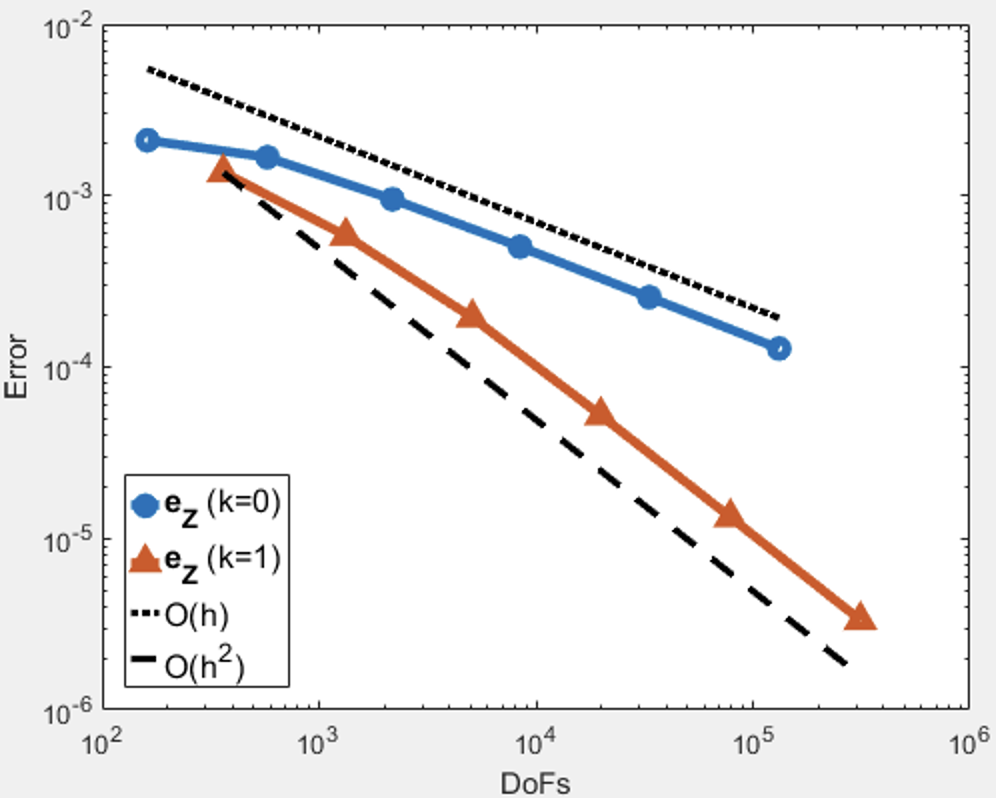}}
\hfill
\subfloat[ $\kappa = 10^{-3}$\label{fig:simplycokappa2}]{\includegraphics[width=0.49\textwidth]{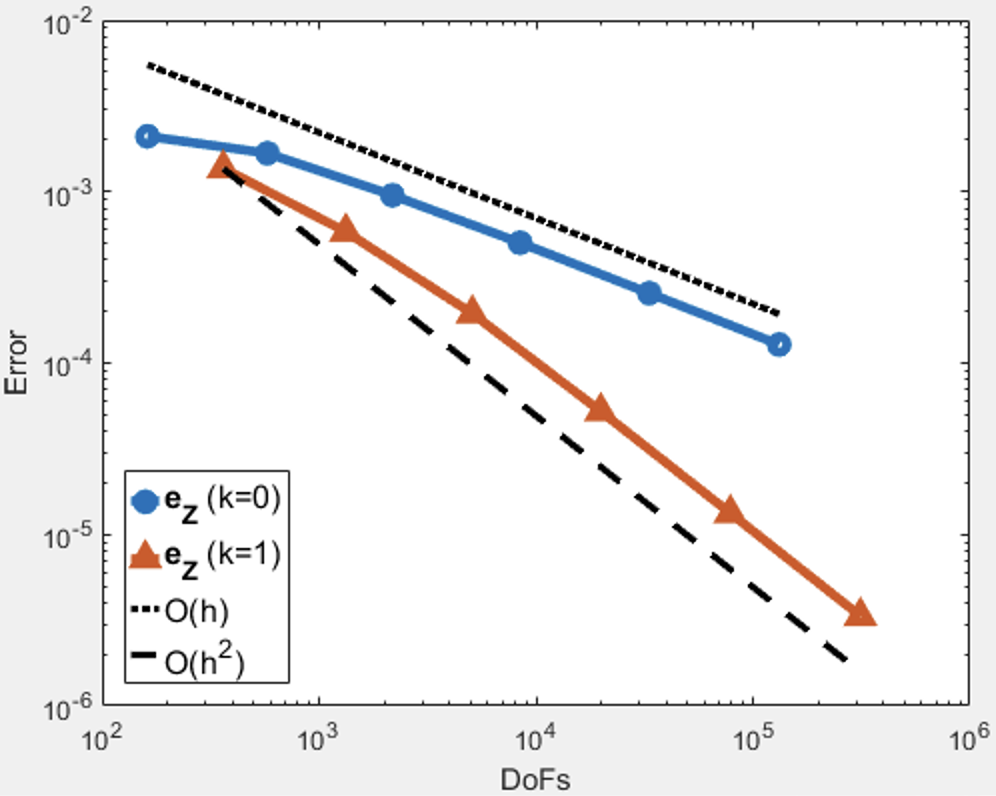}}

	\vspace{-0.5mm}
	\caption{Example~\ref{Example 5.1} (Case~1): Error histories using $C^0$-conforming HHO method.}
	\label{FIGURE ex_C0HHO}
\end{figure}
\begin{figure}
	\centering
\subfloat[ Error history\label{fig:simplycokappa3}]{\includegraphics[width=0.49\textwidth]{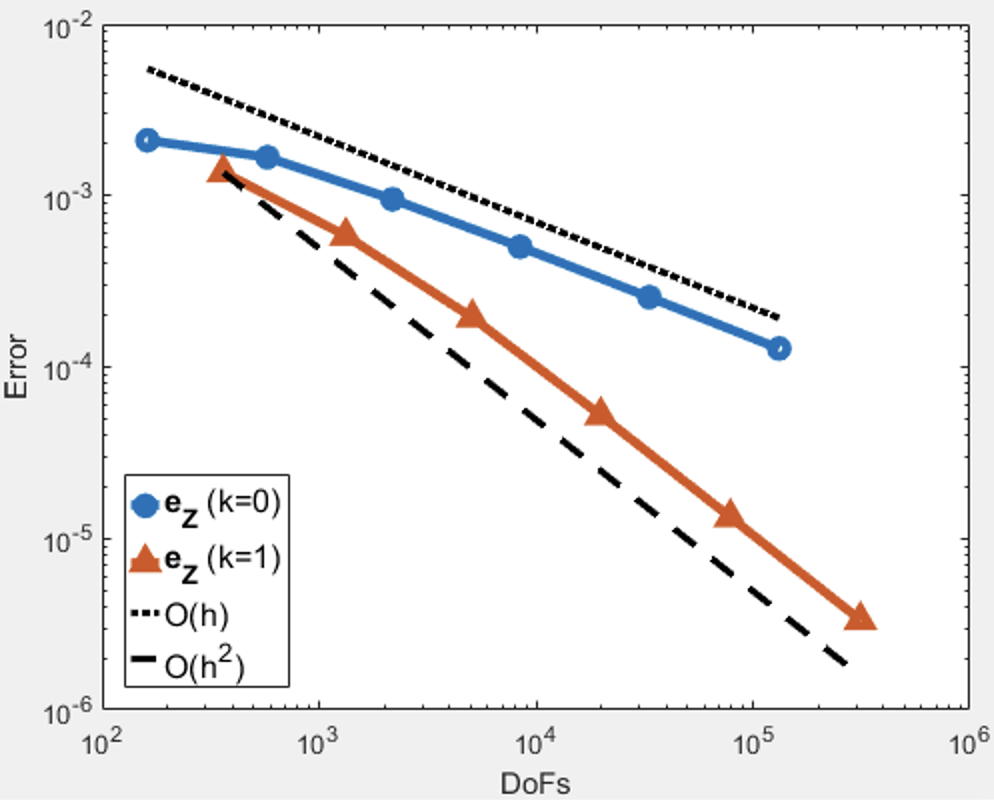}}\\
\hfill
\subfloat[ $\y_{\mathcal{T}_h}$\label{fig:simplycokappa3Yh}]{\includegraphics[width=0.49\textwidth]{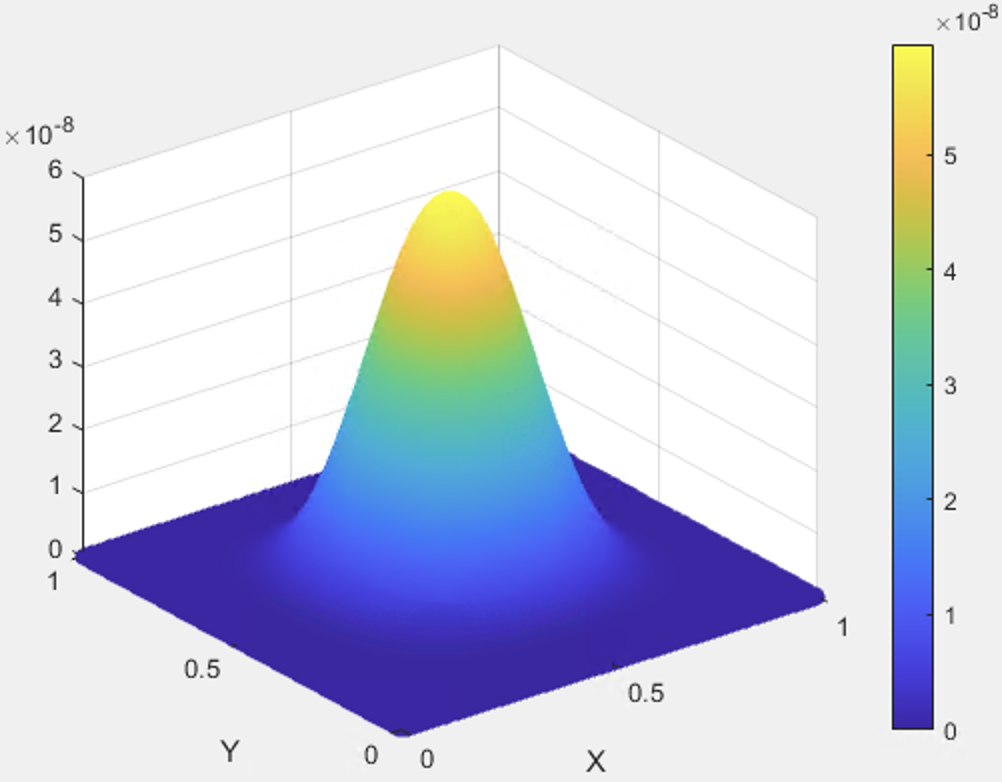}}
\hfill
\subfloat[ $\z_{\mathcal{T}_h}$\label{fig:simplycokappa3Zh}]{\includegraphics[width=0.49\textwidth]{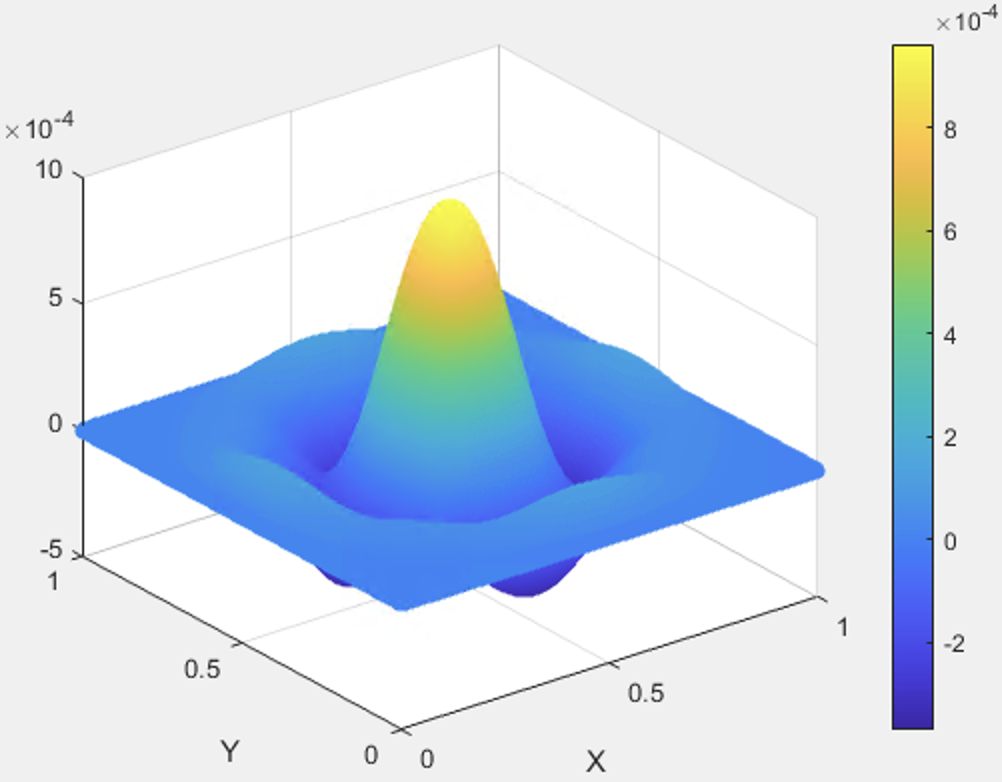}}
	\vspace{-0.5mm}
	\caption{Example~\ref{Example 5.1} (Case~1): Numerical observation for $\kappa = 10^{-5}$  and approximate solutions using $C^0$-conforming HHO method for $k=1$ on a mesh with $8096$ triangles.}
	\label{FIGURE ex_C0HHO_case3}
\end{figure}
\begin{figure}
	\centering
	\subfloat[ Error history\label{fig:cahnHHOkappa1}]{\includegraphics[width=0.49\textwidth]{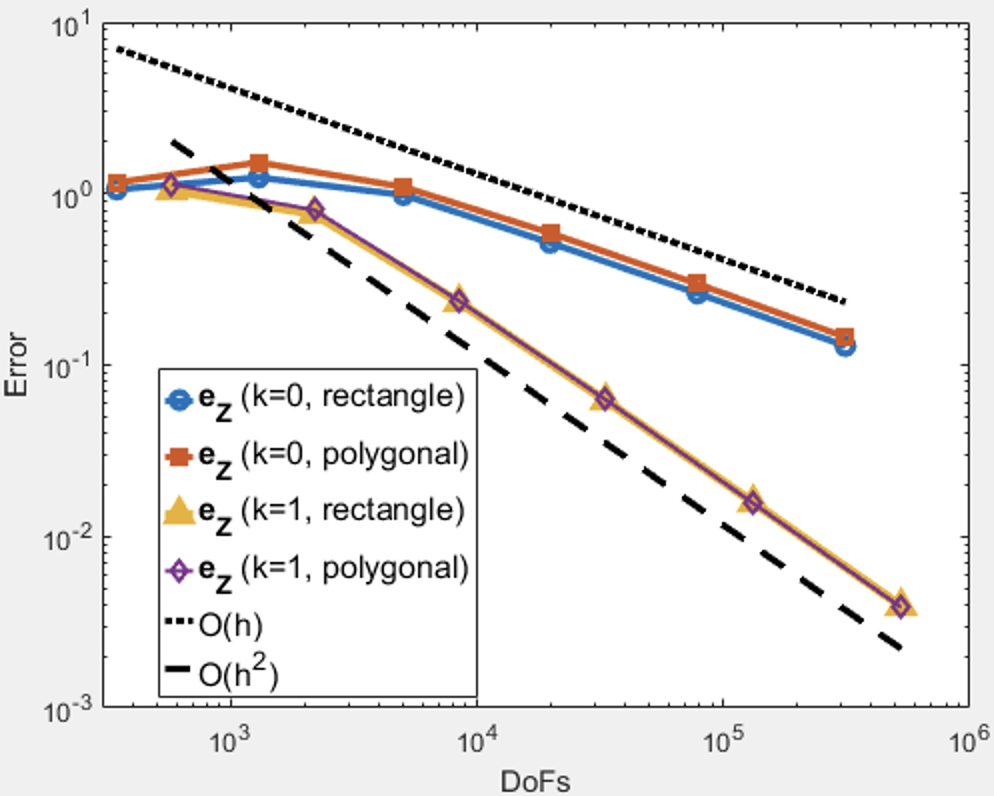}}\\
	\subfloat[ $\y_{\mathcal{T}_h}$\label{fig:cahnHHOkappa1Yh}]{\includegraphics[width=0.49\textwidth]{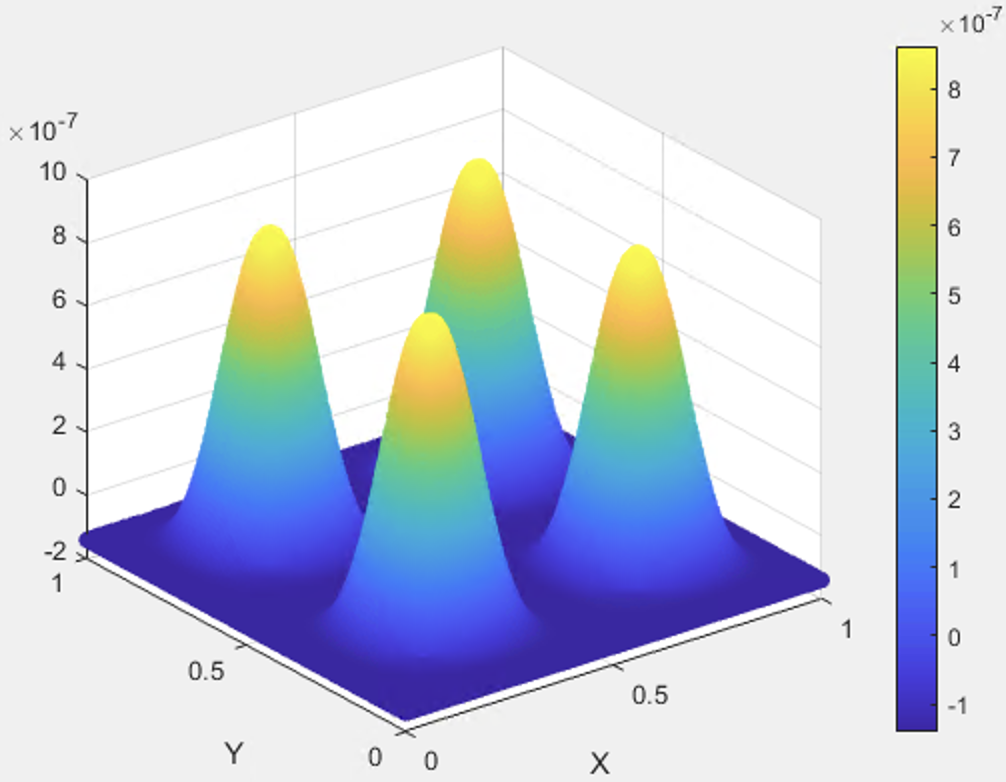}}
	\hfill
	\subfloat[ $\z_{\mathcal{T}_h}$\label{fig:cahnHHOkappa1Zh}]{\includegraphics[width=0.49\textwidth]{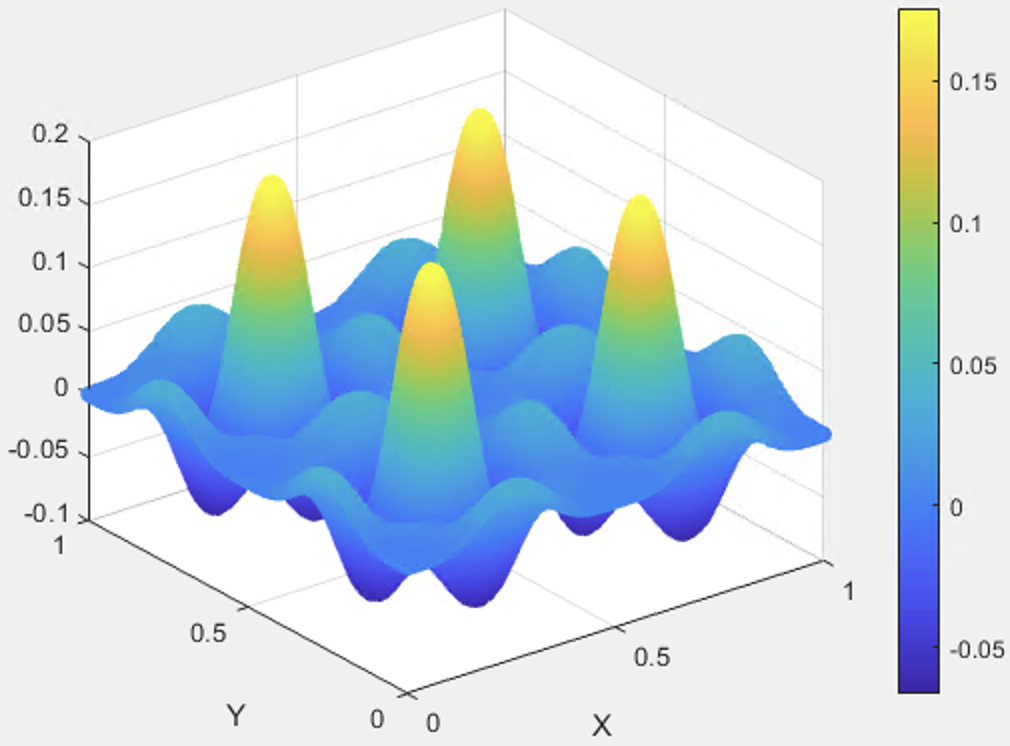}}

	\vspace{-0.5mm}
	\caption{Example~\ref{Example 5.1} (Case~2): Numerical findings for $\kappa = 1$ and approximate solutions using non-conforming HHO method for $k=1$ on a mesh with $1024$ rectangles.}
	\label{FIGURE ex2_HHO_case1}
\end{figure}
\begin{figure}
	\centering

		\subfloat[ $\kappa= 10^{-2}$ \label{fig:cahnHHOkappa2}]{\includegraphics[width=0.49\textwidth]{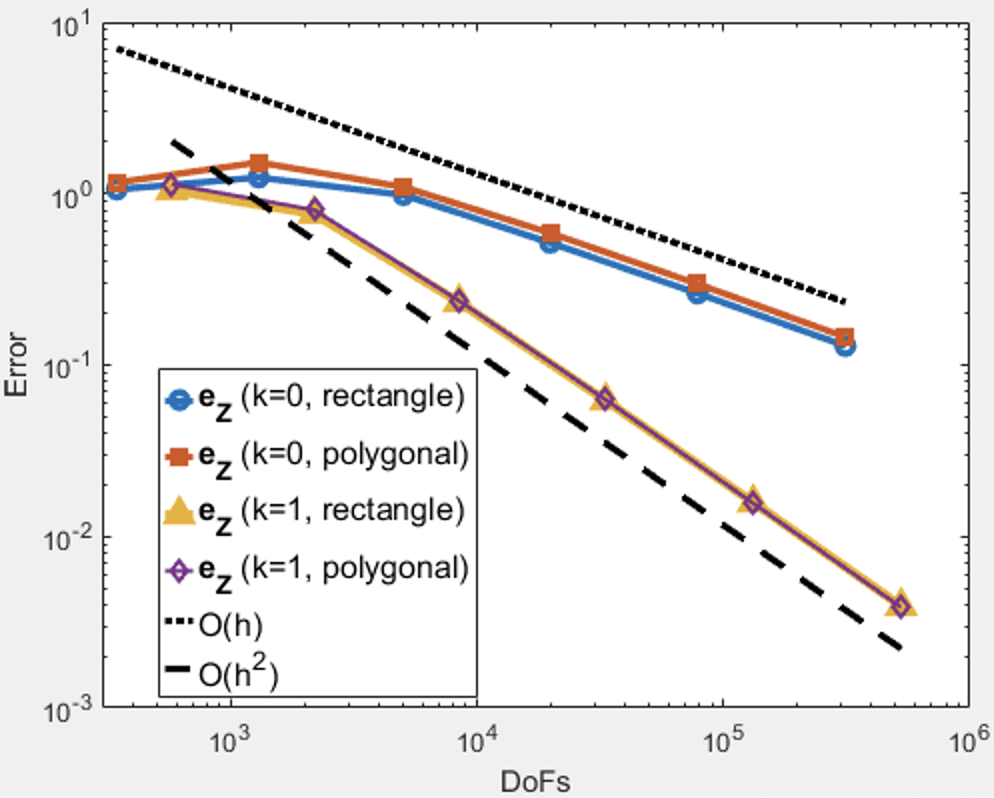}}
	\hfill
			\subfloat[ $\kappa= 10^{-4}$ \label{fig:cahnHHOkappa3}]{\includegraphics[width=0.49\textwidth]{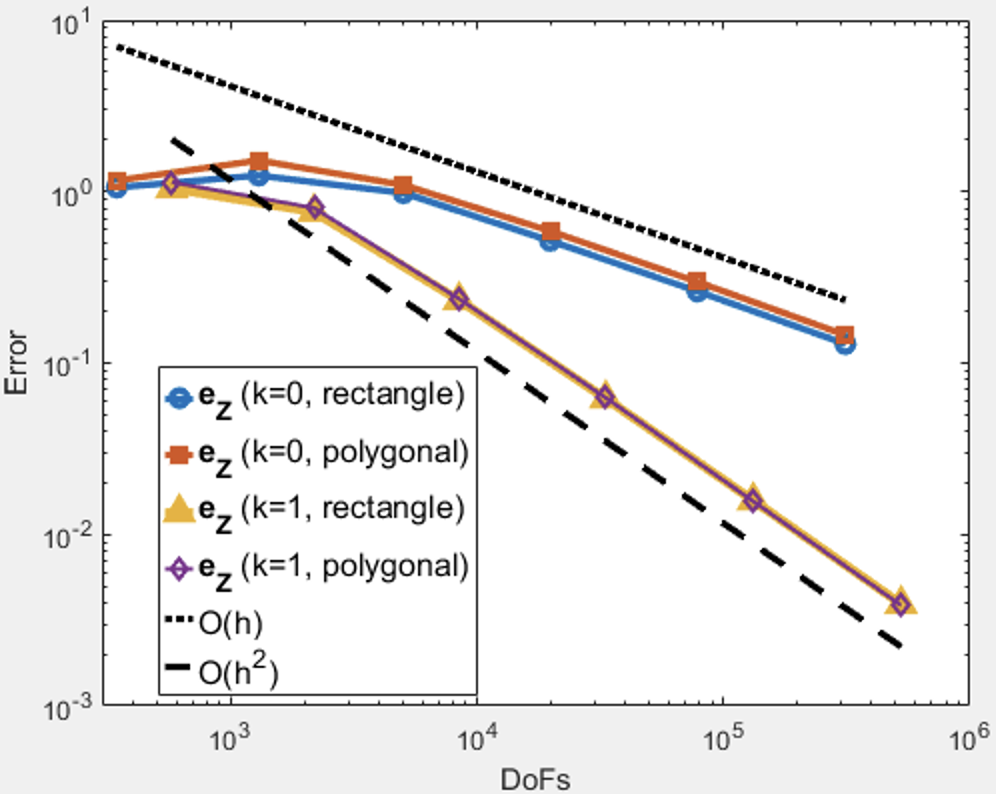}}
	\vspace{-0.5mm}
	\caption{Example~\ref{Example 5.1} (Case~2): Error histories using non-conforming HHO method.}
	\label{FIGURE ex2_HHO_case2}
\end{figure}
%
%

\begin{figure}
	\centering
	
	\subfloat[ $\kappa= 1$ \label{fig:cahncokappa1}]{\includegraphics[width=0.49\textwidth]{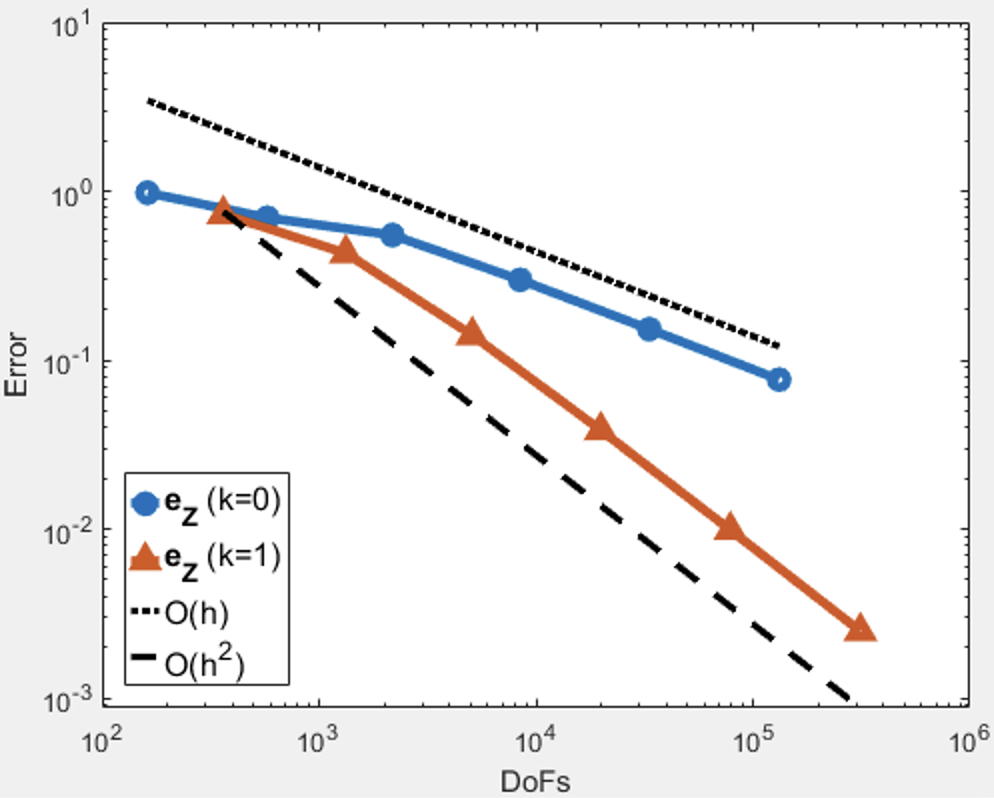}}
	\hfill
	\subfloat[ $\kappa= 10^{-2}$ \label{fig:cahncokappa2}]{\includegraphics[width=0.49\textwidth]{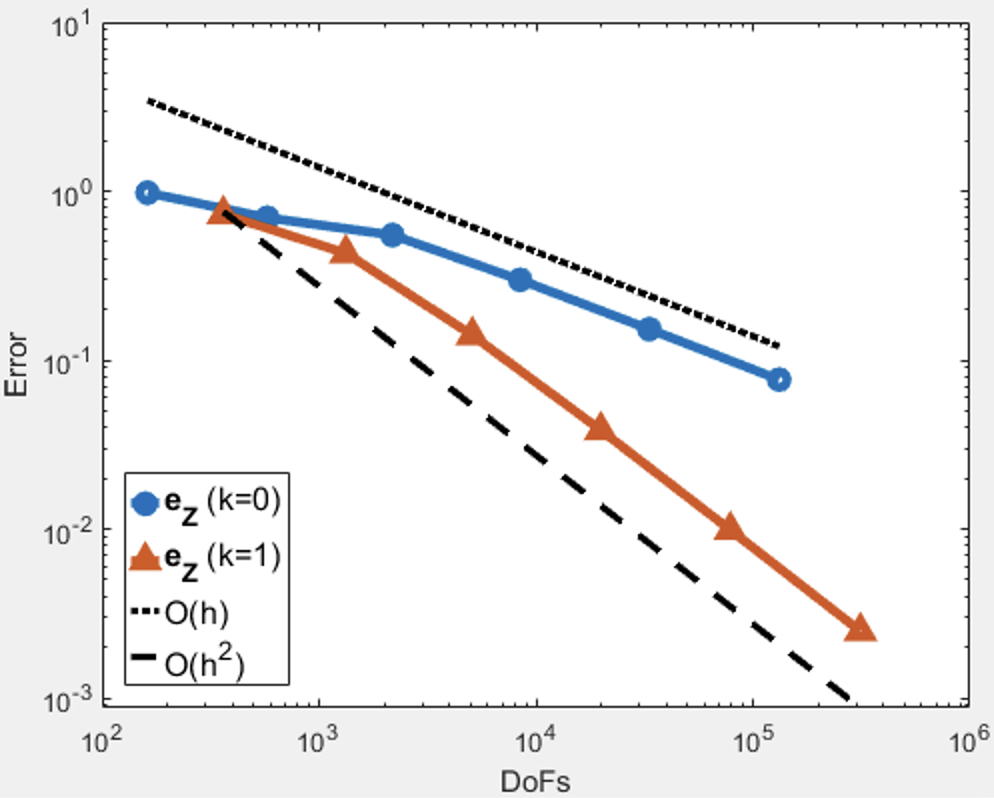}}
	
	\vspace{-0.5mm}
	\caption{Example~\ref{Example 5.1} (Case~2): Convergence results using $C^0$-conforming HHO method.}
	\label{FIGURE ex2_c0HHO_case1}
\end{figure}
%
\begin{figure}
	\centering
	\subfloat[ Error history\label{fig:cahncokappa3}]{\includegraphics[width=0.49\textwidth]{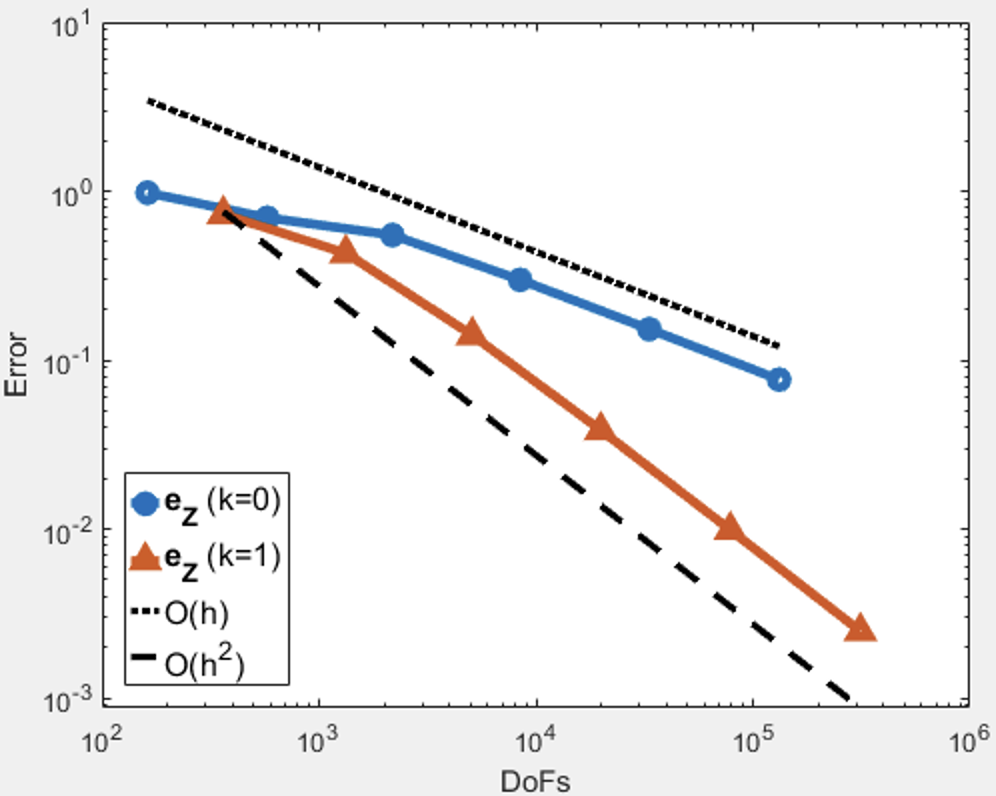}}\\
	\subfloat[ $\y_{\mathcal{T}_h}$\label{fig:cahncokappa3Yh}]{\includegraphics[width=0.49\textwidth]{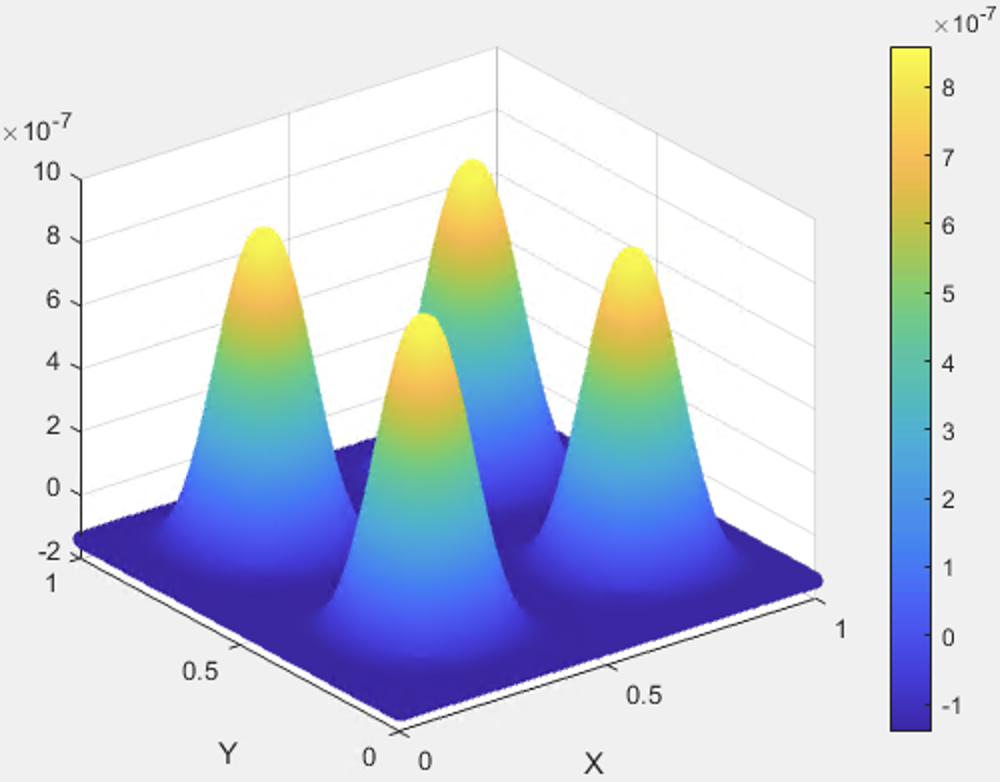}}
	\hfill
	\subfloat[ $\z_{\mathcal{T}_h}$\label{fig:cahncokappa3Zh}]{\includegraphics[width=0.49\textwidth]{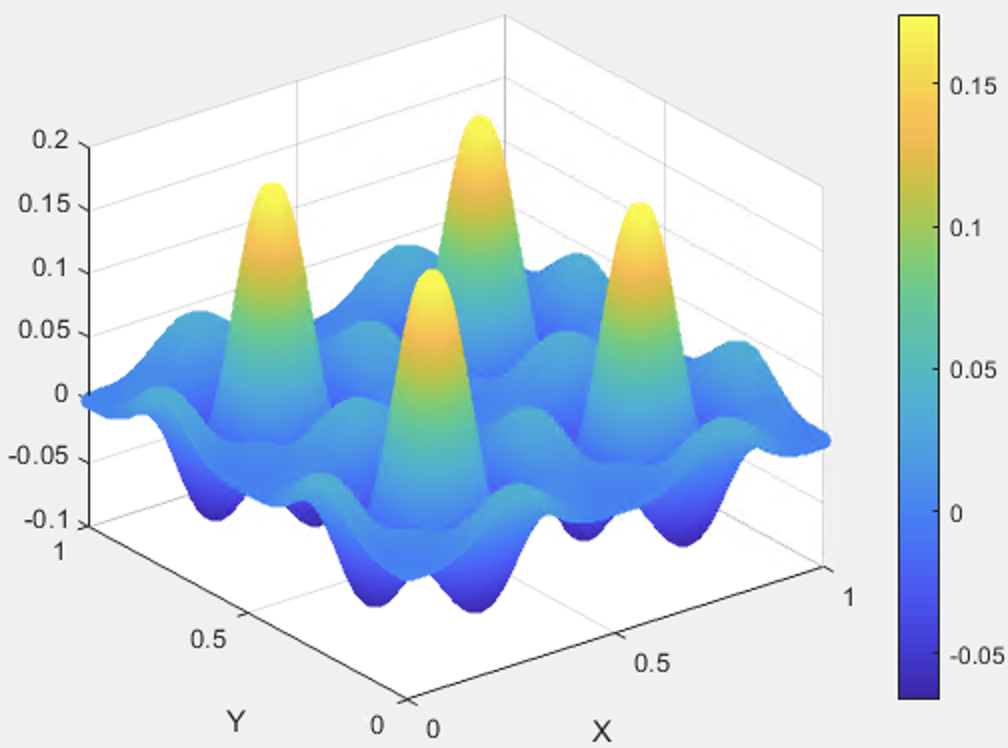}}
	\vspace{-0.5mm}
	\caption{Example~\ref{Example 5.1} (Case~2): Numerical results  for $\kappa = 10^{-4}$ and approximate solutions using $C^0$-conforming HHO method for $k=1$ on a mesh with $8096$ triangles.}
	\label{FIGURE ex2_c0HHO_case3}
\end{figure}
	\subsection{Adaptivity test for a singular solution}\label{Example 5.3}
	In this example, we consider the simply supported boundary conditions and select $\f$ on $\Omega = (0,1)^2$ with $\kappa = 1$ such that $\y(\x) = x_1^5 x_2^5 (x_1-1)^5(x_2-1)^5(x_1^2 + x_2^2)^{-{5}/{2}}$.
	
	We test an $h$-adaptive test algorithm driven by the a posteriori error estimator in Sec.~\ref{A posteriori error estimates}. The adaptive algorithm starts from a coarse mesh and uses the estimator to mark the mesh cells for refinement through a bulk-chasing criterion (also known as D\"olfer's marking \cite{MR1393904}). The adaptive algorithm is classically described as 
	\begin{align*}
		\text{Solve} \longrightarrow  \text{Estimate} \longrightarrow  \text{Mark} \longrightarrow  \text{Refine}.
	\end{align*}
	We test the convergence rate of this algorithm with $k=0,1$ and bulk-chasing criterion set to $40\%$.
	Figs.~\ref{fig:adaptiveunitsqk0sc} and \ref{fig:adaptiveunitsqk=1sc} illustrate the estimator decay, while Figs.~\ref{fig:effectivityunitsqk0sc} and \ref{fig:effectivityunitsqk1sc} demonstrate effectiveness of the proposed adaptive strategy in resolving the singular behavior of the exact solution. Starting from a coarse mesh, the estimator $\e_{\boldsymbol{\eta}}$ successfully identifies the regions requiring further refinement, leading to a progressively finer mesh around the singularity. For both $k=0$ and $k=1$, after sufficient refinement, the energy error $\e_{Z}$ and the estimator $\e_{\boldsymbol{\eta}}$ exhibit the expected optimal decay, while the effectivity index remains nearly constant. These results confirm the reliability of the proposed estimator and the effectiveness of the resulting adaptive algorithm in recovering the optimal convergence behavior.
	
	\subsection{Adaptivity for Cahn-Hilliard-type boundary conditions}\label{cahn_adaptive_ex}
	In this example, we test the adaptive algorithm using the same a posteriori error estimator as in Theorem~\ref{reliability}. We consider $\Omega=(0,1)^2$ with $\kappa=1$ and construct $\f$ such that
	$
	\y(\x)=x_1^6x_2^6(x_1-1)^6(x_2-1)^6(x_1^2+x_2^2)^{-3}.
	$
	The presence of the factor $(x_1^2+x_2^2)^{-3}$ introduces a singularity at the origin, resulting in limited regularity of the exact solution. Consequently, the optimal convergence rates are not achieved under uniform mesh refinement.. To overcome this difficulty, we employ the proposed adaptive strategy, which automatically identifies the region of reduced regularity and concentrates the mesh refinement in its vicinity. Fig.~\ref{fig:cahnadaptivityk1} and \ref{fig:cahneffectivityk1} presents the convergence history of the adaptive algorithm for $k=1$ with a bulk-chasing criterion of $40\%$, demonstrating the effectiveness of adaptive refinement in recovering the desired convergence behaviour despite the limited regularity.
%

	\begin{figure}
	\centering
	\subfloat[Estimator and energy error\label{fig:adaptiveunitsqk0sc}]{\includegraphics[width=0.495\textwidth]{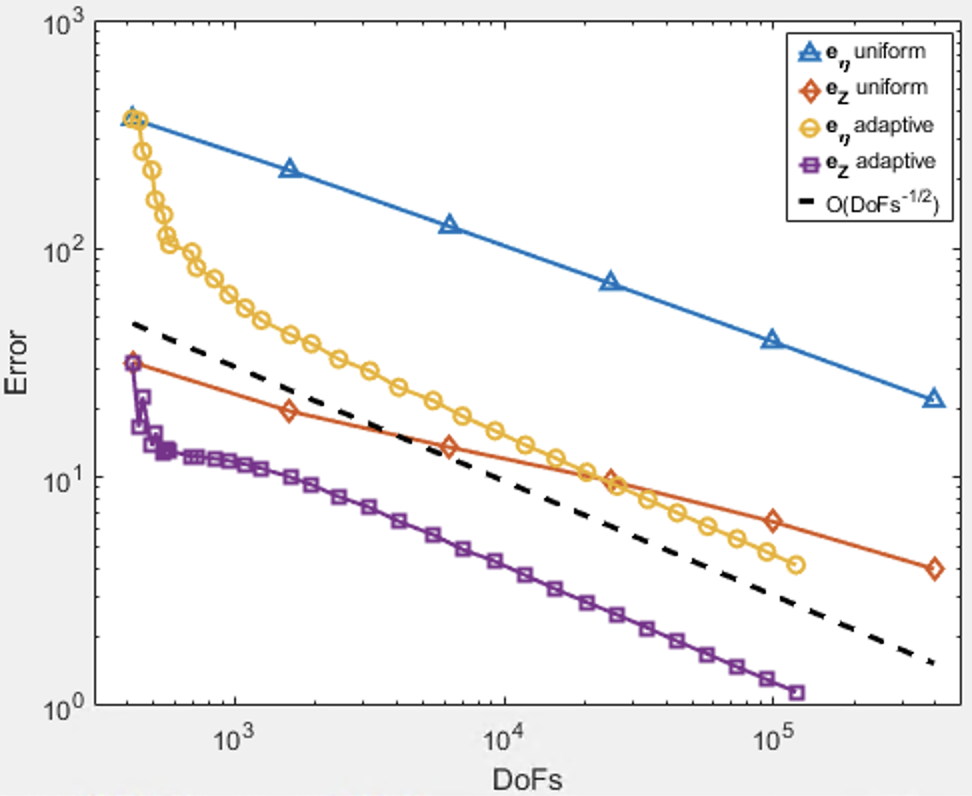}}
	\hfill
	\subfloat[Effectivity index\label{fig:effectivityunitsqk0sc}]{\includegraphics[width=0.495\textwidth]{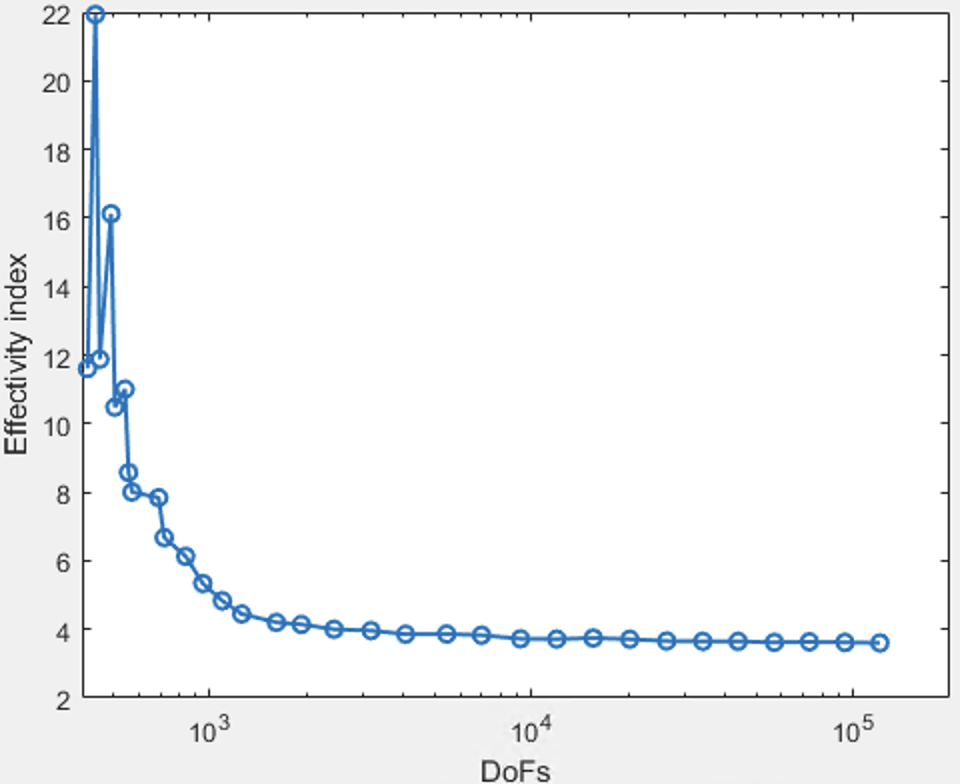}}\\
	\subfloat[ $\y_{\mathcal{T}_h}$ \label{fig:adaptiveunitsqk0scYh}]{\includegraphics[width=0.49\textwidth]{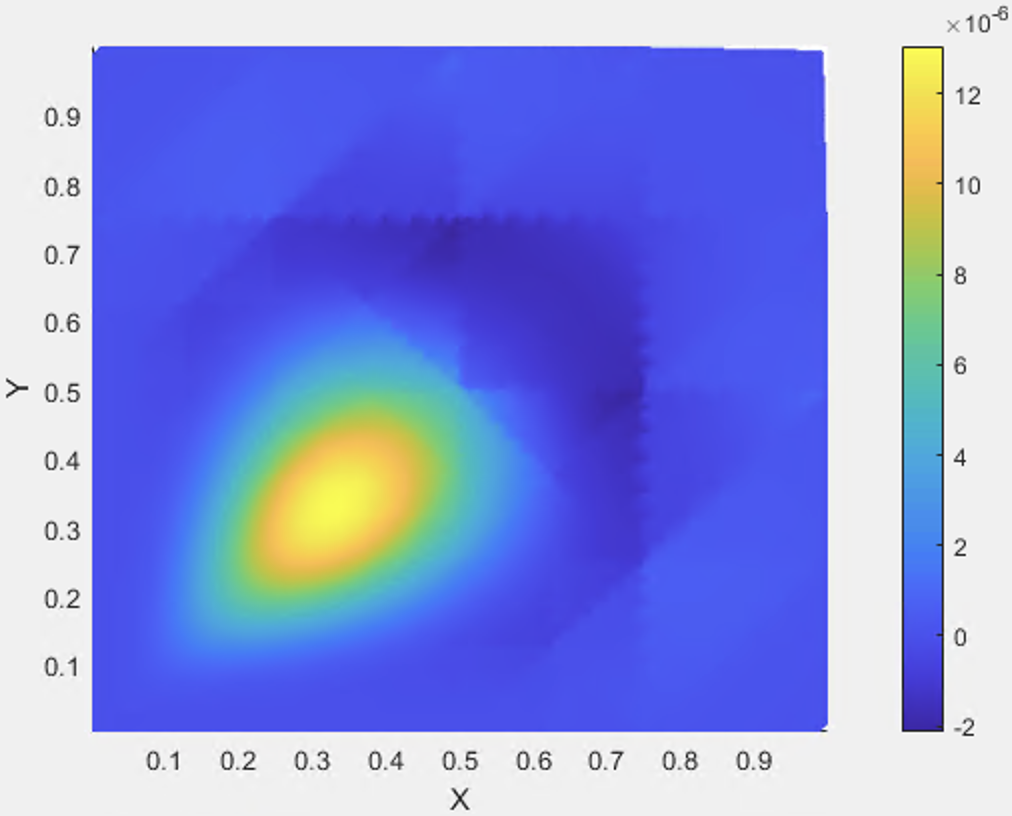}}
	\hfill
	\subfloat[$\z_{\mathcal{T}_h}$ \label{fig:adaptiveunitsqk0scZh}]{\includegraphics[width=0.49\textwidth]{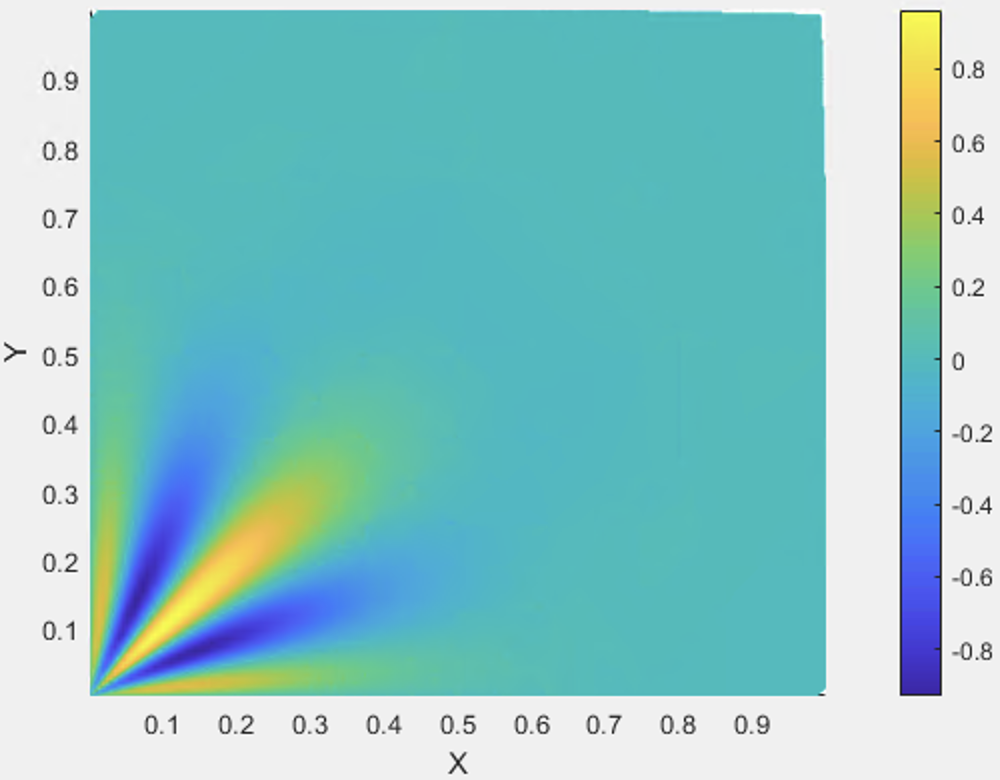}}\\
	\subfloat[ \label{fig:mesh3634lv28unitsqk0sc}]{\includegraphics[width=0.49\textwidth]{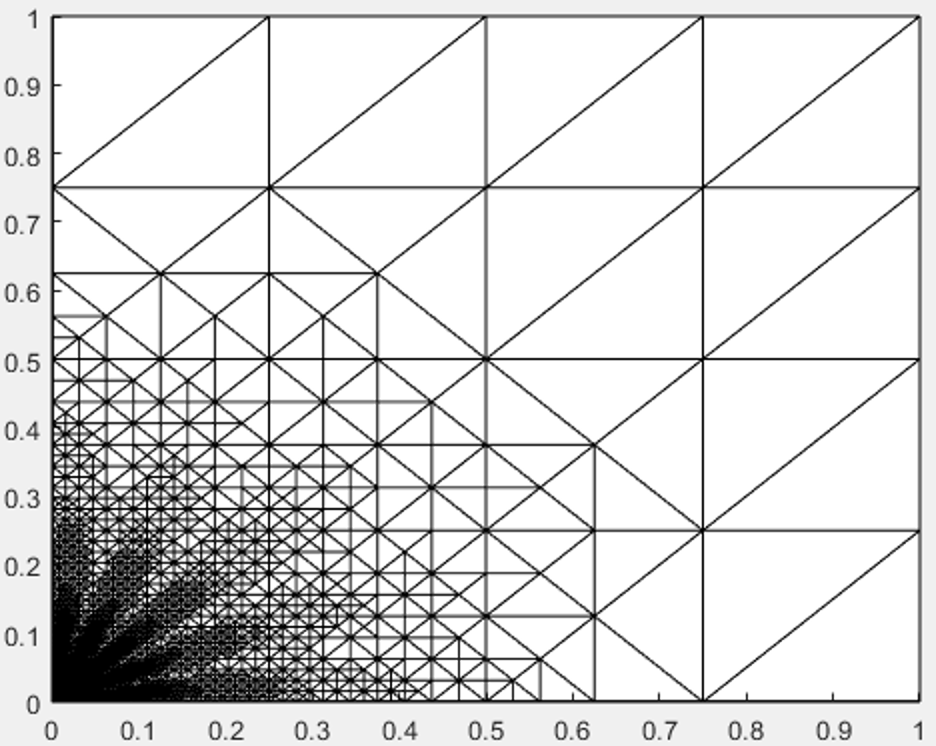}}

	\vspace{-0.5mm}
	\caption{Example~\ref{Example 5.3}: Convergence results and computed solutions \ref{fig:adaptiveunitsqk0scYh}-\ref{fig:adaptiveunitsqk0scZh} on an adaptive mesh \ref{fig:mesh3634lv28unitsqk0sc} with $3634$ triangles for $k=0$.}
	\label{FIGURE ex3_k0}
\end{figure}
	\begin{figure}
	\centering
		\subfloat[Estimator and energy error\label{fig:adaptiveunitsqk=1sc}]{\includegraphics[width=0.495\textwidth]{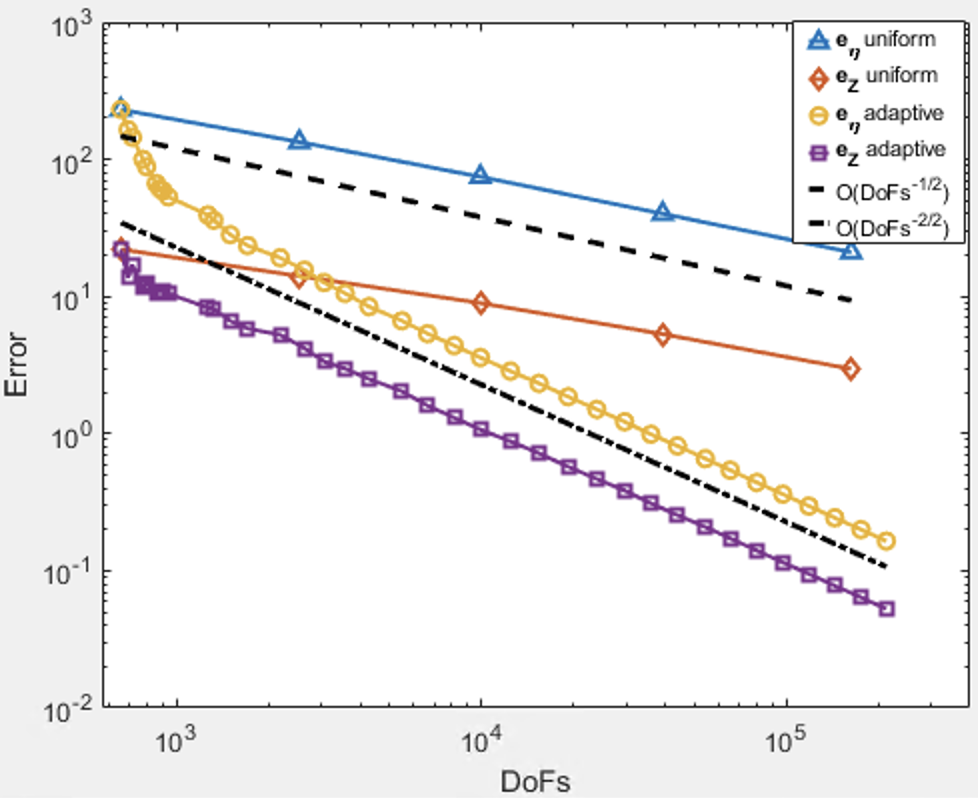}}
	\hfill
	\subfloat[Effectivity index\label{fig:effectivityunitsqk1sc}]{\includegraphics[width=0.495\textwidth]{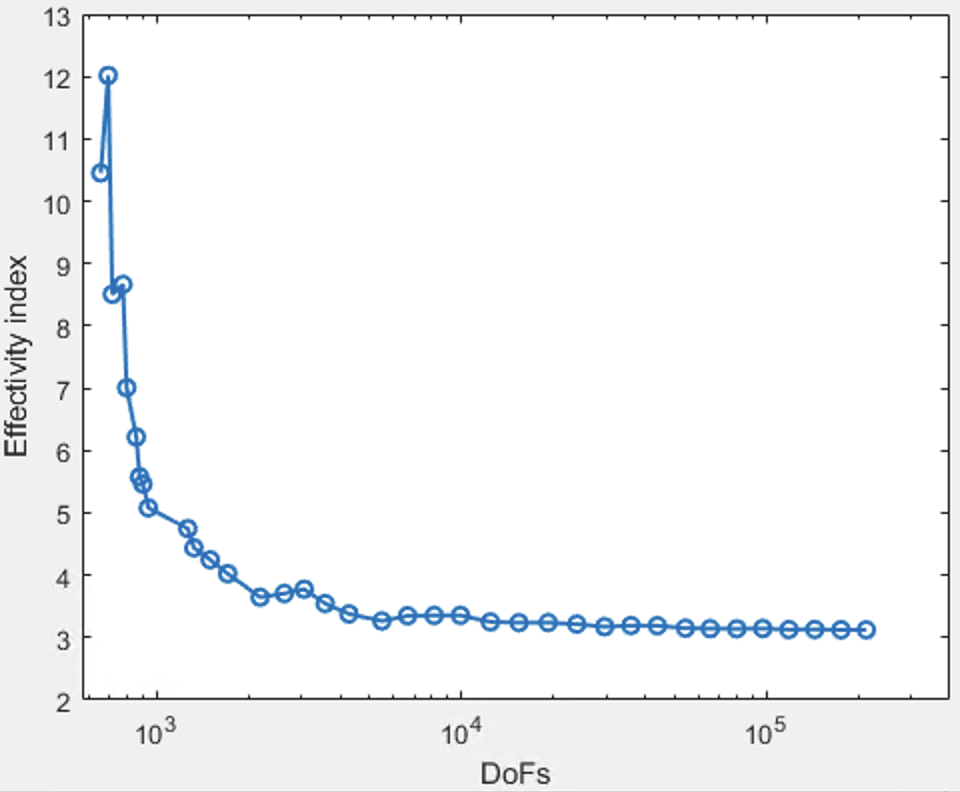}}\\
	\subfloat[ $\y_{\mathcal{T}_h}$ \label{fig:adaptiveunitsqk=1scYh}]{\includegraphics[width=0.49\textwidth]{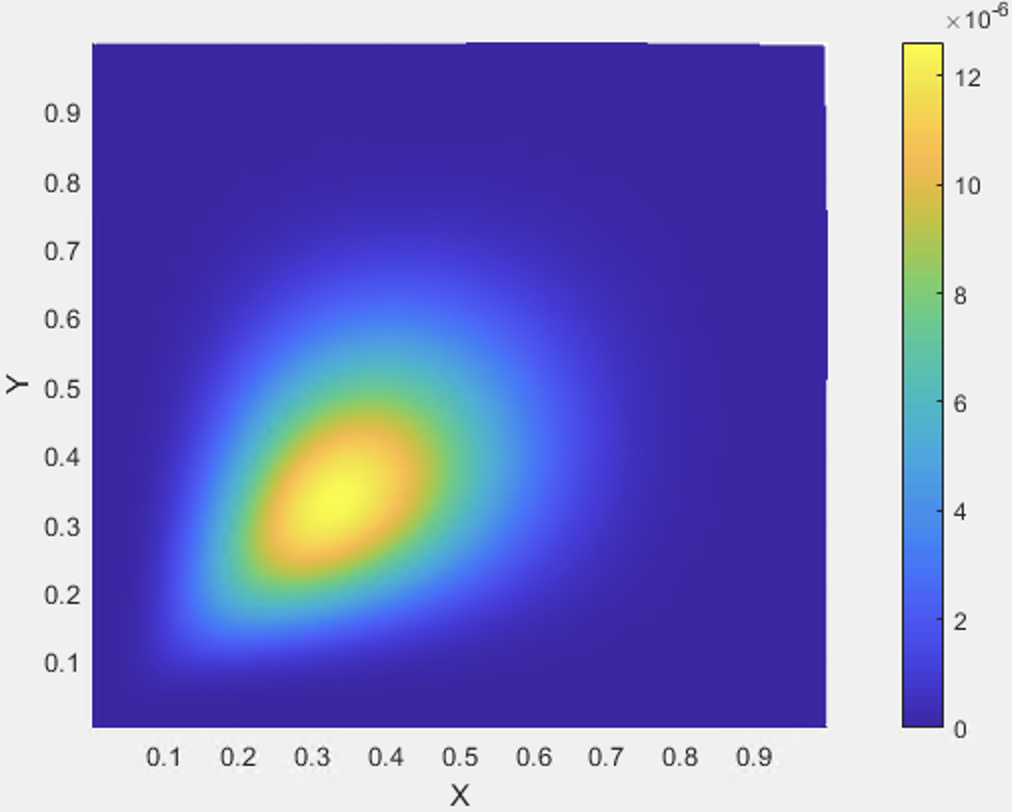}}
	\hfill
	\subfloat[$\z_{\mathcal{T}_h}$ \label{fig:adaptiveunitsqk=1scZh}]{\includegraphics[width=0.49\textwidth]{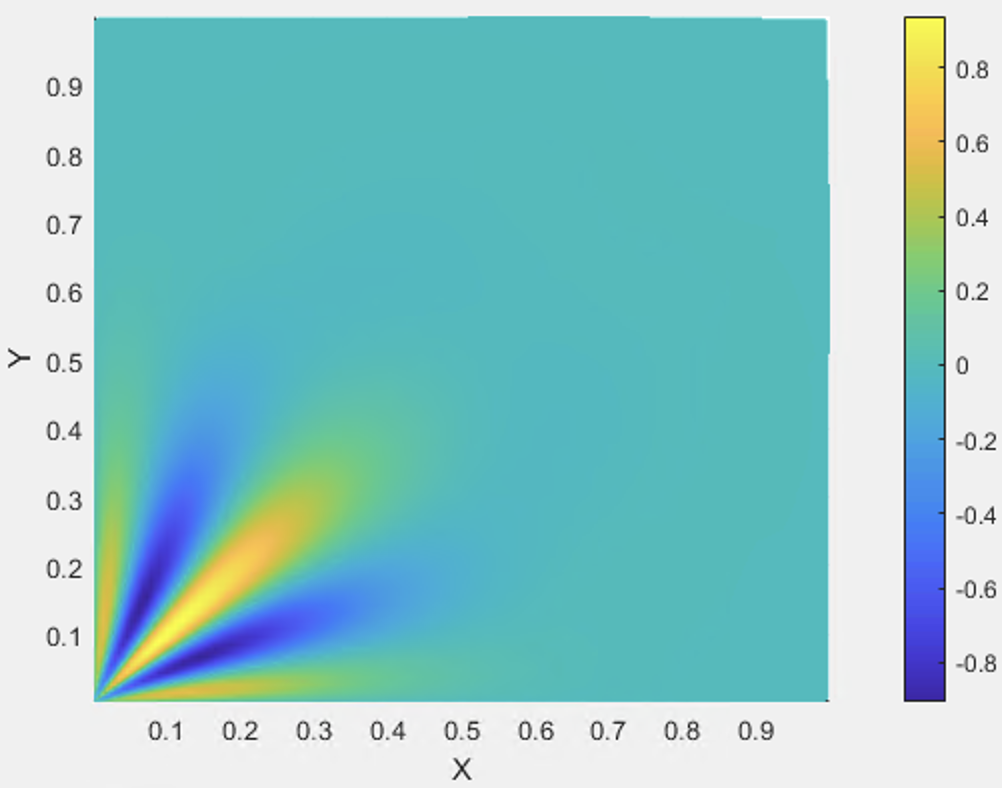}}\\
	\subfloat[ \label{fig:mesh4182unitsqlv32k1}]{\includegraphics[width=0.49\textwidth]{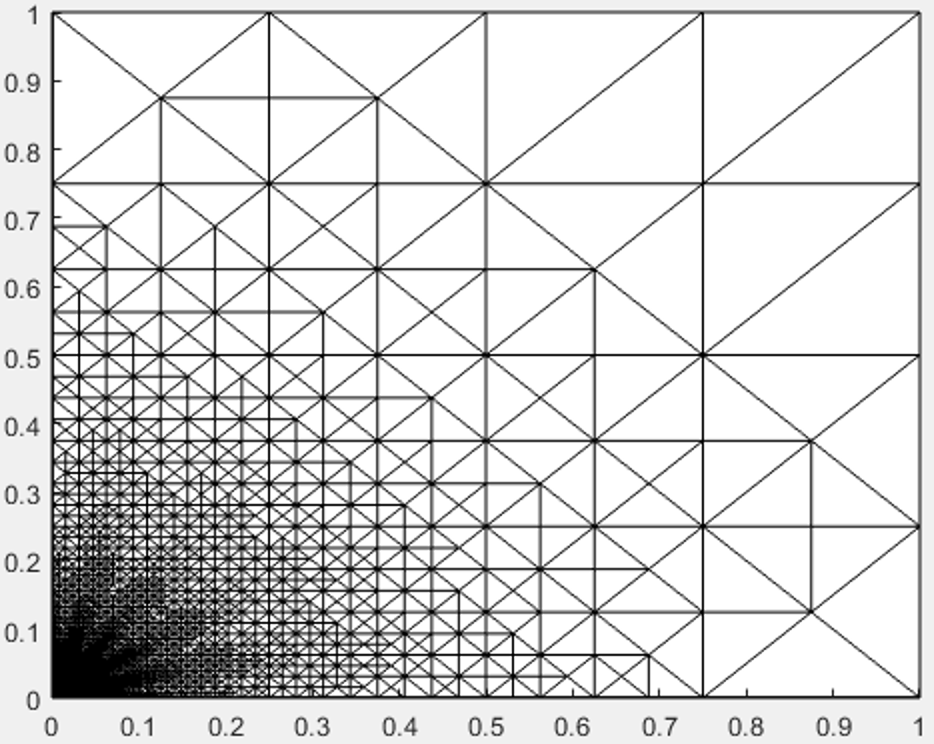}}
	\vspace{-0.5mm}
	\caption{Example~\ref{Example 5.3}: Convergence results and computed solutions \ref{fig:adaptiveunitsqk=1scYh}-\ref{fig:adaptiveunitsqk=1scYh} on  an adaptive mesh \ref{fig:mesh4182unitsqlv32k1} with $4182$ triangles for $k=1$.}
	\label{FIGURE ex3_k1}
\end{figure}

	\begin{figure}
	\centering
	\subfloat[Estimator and energy error\label{fig:cahnadaptivityk1}]{\includegraphics[width=0.495\textwidth]{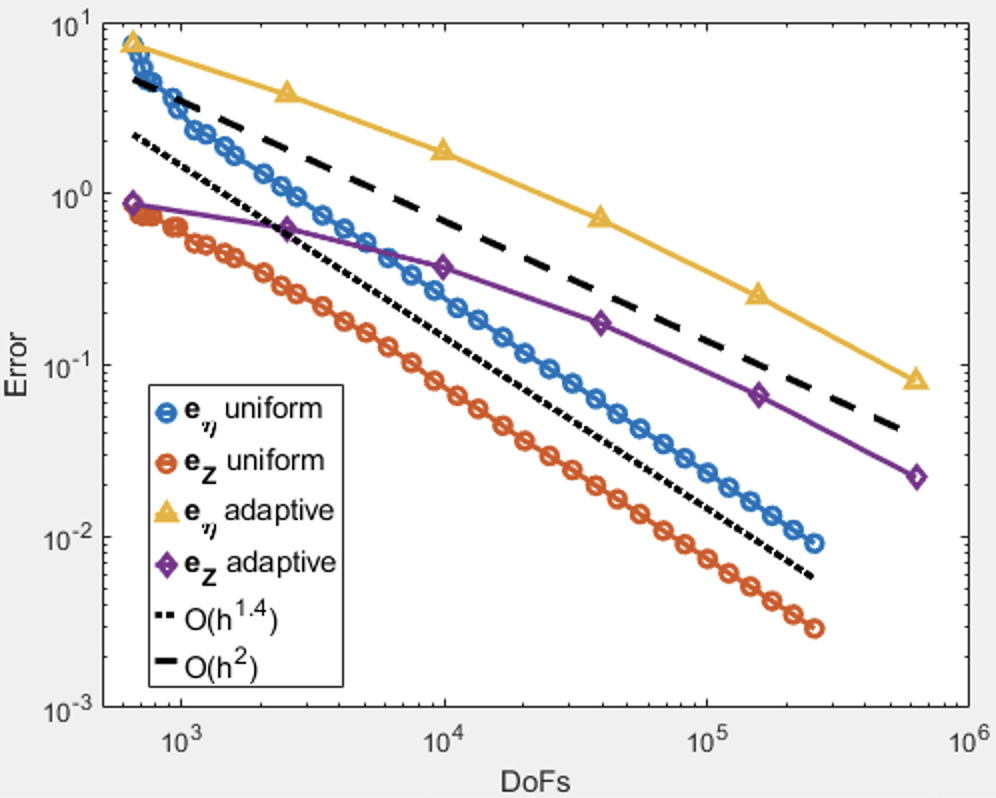}}
	\hfill
	\subfloat[Effectivity index\label{fig:cahneffectivityk1}]{\includegraphics[width=0.495\textwidth]{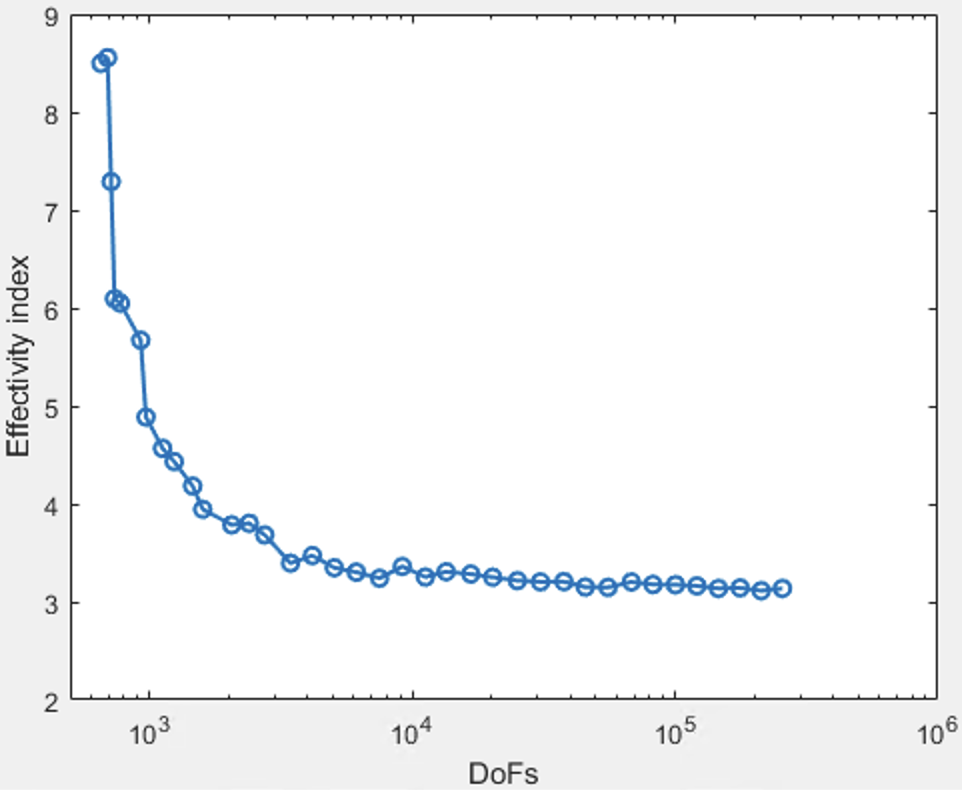}}\\
	\subfloat[ $\y_{\mathcal{T}_h}$ \label{fig:cahnadaptivityk1Yh}]{\includegraphics[width=0.49\textwidth]{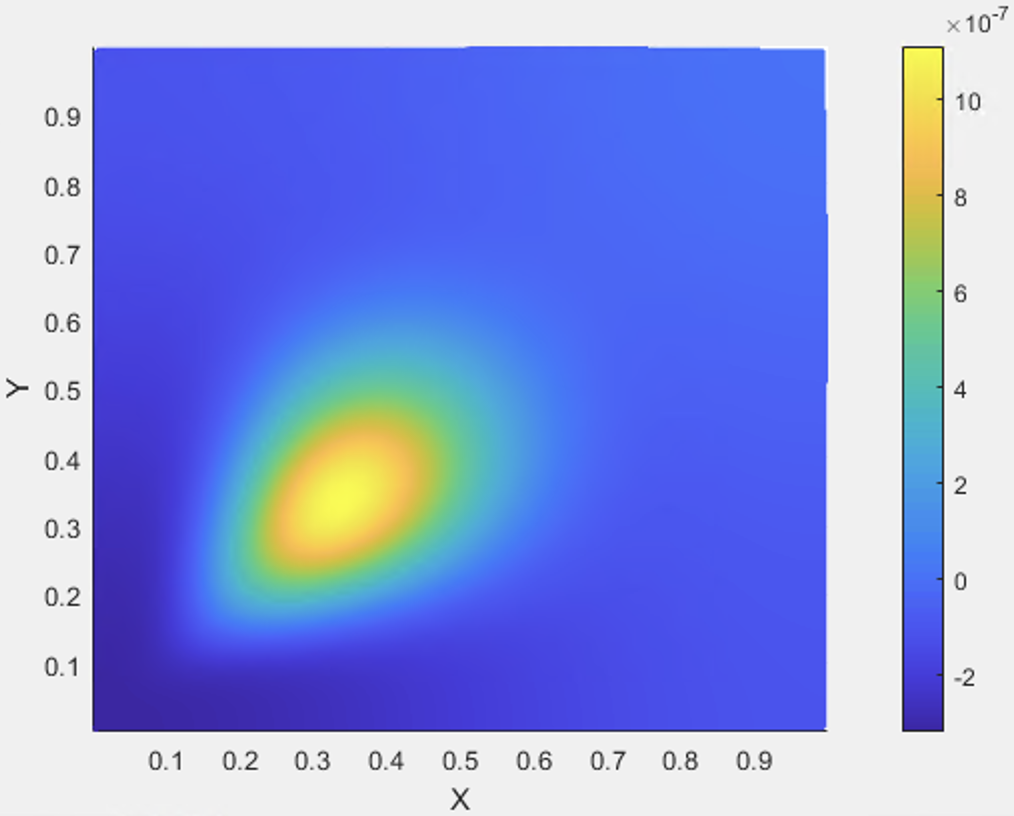}}
	\hfill
	\subfloat[$\z_{\mathcal{T}_h}$ \label{fig:cahnadaptivityk1Zh}]{\includegraphics[width=0.49\textwidth]{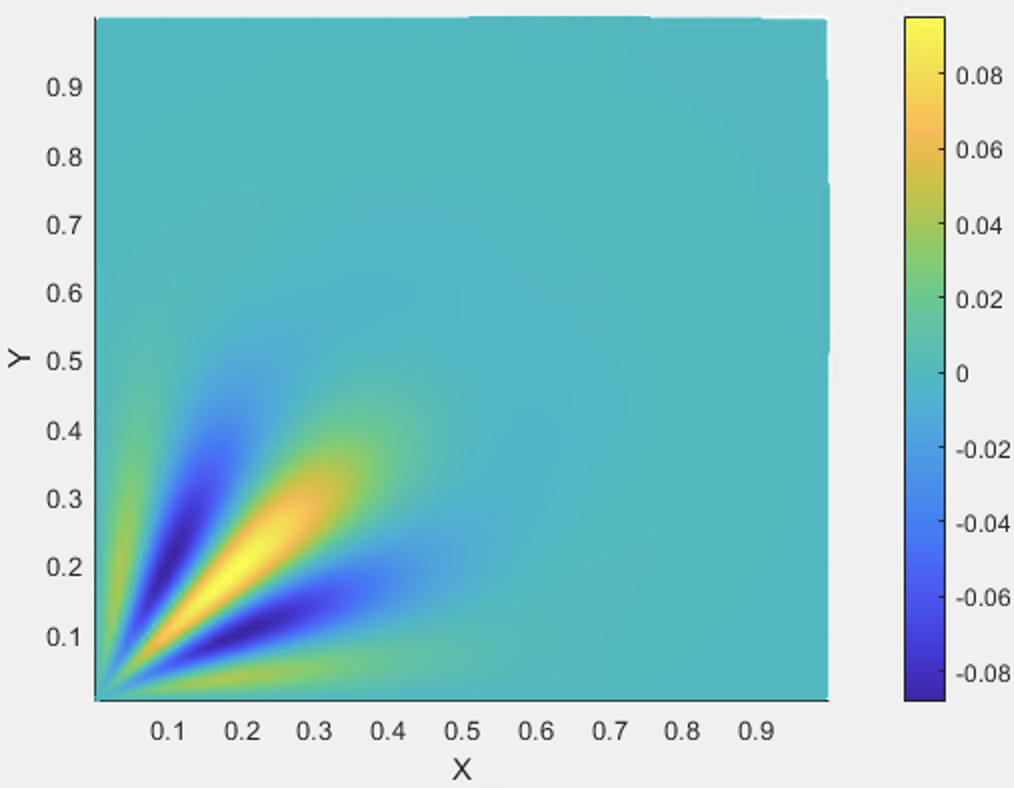}}\\
	\subfloat[ \label{fig:cahnmesh2908lv29k1}]{\includegraphics[width=0.49\textwidth]{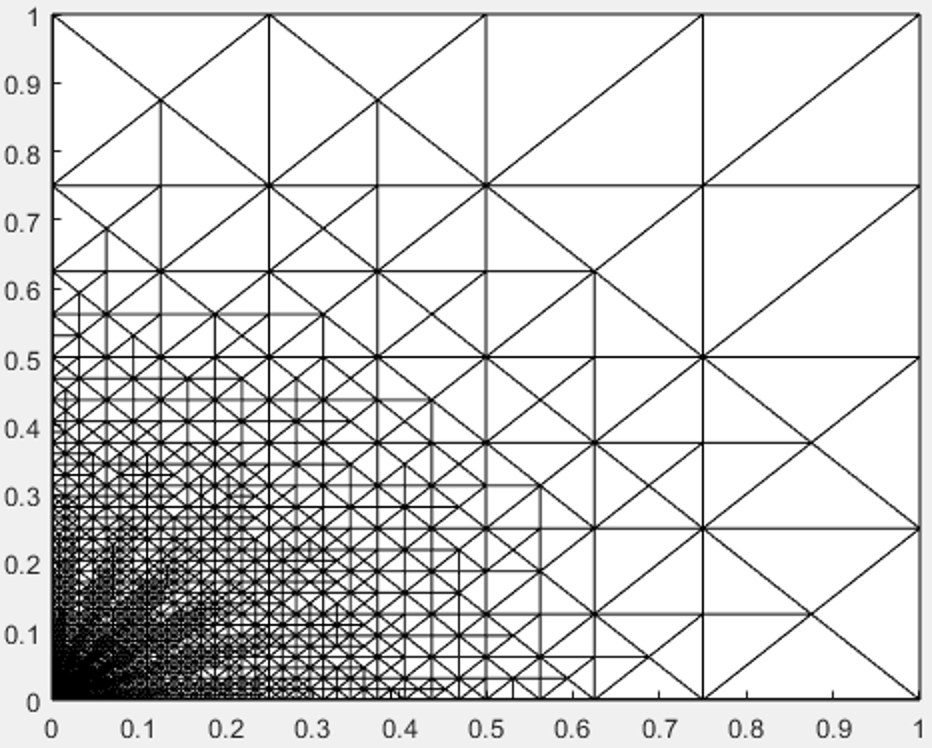}}
	\vspace{-0.5mm}
	\caption{Example~\ref{cahn_adaptive_ex}: Convergence results and computed solutions \ref{fig:cahnadaptivityk1Yh}-\ref{fig:cahnadaptivityk1Zh} on an adaptive mesh \ref{fig:cahnmesh2908lv29k1} with $2908$ triangles for $k=1$.
		.}
	\label{FIGURE ex3cahn}
\end{figure}
\section{Conclusions}
In this paper, we have developed and analysed HHO discretizations for sixth-order elliptic problems with two different types of boundary conditions on convex polygonal and polyhedral domains in two and three dimensions. Based on the Ciarlet--Raviart reformulation, we have considered both a fully non-conforming HHO method on general polytopal meshes and a $C^0$-conforming HHO method on simplicial meshes. For both discretizations, optimal-order a priori error estimates have been established under suitable regularity assumptions. In addition, for the fully non-conforming HHO method, we have derived a reliable residual-based a posteriori error estimator. The numerical experiments confirm the predicted convergence rates and demonstrate the reliability and efficiency of the proposed error estimator through the computed effectivity indices. These results provide a unified HHO framework for the approximation of sixth-order elliptic problems on a broad class of meshes and constitute a basis for the development of adaptive strategies and extensions to more general higher-order elliptic problems.
\section*{Acknowledgements}
The authors acknowledge the computational support provided by the National Supercomputing Mission (NSM), India, through the ``PARAM Ganga'' high-performance computing facility at IIT Roorkee. The authors would also like to express their sincere gratitude to Prof. Zhaonan Dong (Inria, 48 rue Barrault, 75647 Paris, France) for generously providing the HHO code and for many insightful discussions that greatly contributed to this work.
	\bibliographystyle{abbrv}
	\bibliography{reference}

@article {MR4621059,
	AUTHOR = {Zhang, Yongchao and Mei, Liquan and Wang, Gang},
	TITLE = {A posteriori error analysis of the hybrid high-order method
	for the {S}tokes problem},
	JOURNAL = {J. Sci. Comput.},
	FJOURNAL = {Journal of Scientific Computing},
	VOLUME = {96},
	YEAR = {2023},
	NUMBER = {3},
	PAGES = {Paper No. 74, 31},
	ISSN = {0885-7474,1573-7691},
	MRCLASS = {65N30 (65N35)},
	MRNUMBER = {4621059},
	MRREVIEWER = {Nianyu\ Yi},
	DOI = {10.1007/s10915-023-02291-6},
	URL = {https://doi.org/10.1007/s10915-023-02291-6},
}

@article {MR2034620,
	AUTHOR = {Karakashian, Ohannes A. and Pascal, Frederic},
	TITLE = {A posteriori error estimates for a discontinuous {G}alerkin
	approximation of second-order elliptic problems},
	JOURNAL = {SIAM J. Numer. Anal.},
	FJOURNAL = {SIAM Journal on Numerical Analysis},
	VOLUME = {41},
	YEAR = {2003},
	NUMBER = {6},
	PAGES = {2374--2399},
	ISSN = {0036-1429,1095-7170},
	MRCLASS = {65N15 (65N30 65N50 65N55)},
	MRNUMBER = {2034620},
	MRREVIEWER = {M.\ M.\ Sussman},
	DOI = {10.1137/S0036142902405217},
	URL = {https://doi.org/10.1137/S0036142902405217},
}

@article {MR4891729,
	AUTHOR = {Carstensen, Carsten and Tran, Ngoc Tien},
	TITLE = {Locking-free hybrid high-order method for linear elasticity},
	JOURNAL = {SIAM J. Numer. Anal.},
	FJOURNAL = {SIAM Journal on Numerical Analysis},
	VOLUME = {63},
	YEAR = {2025},
	NUMBER = {2},
	PAGES = {827--853},
	ISSN = {0036-1429,1095-7170},
	MRCLASS = {65N30 (65N12 65Y20)},
	MRNUMBER = {4891729},
	MRREVIEWER = {Alexandre\ L.\ Madureira},
	DOI = {10.1137/24M1650363},
	URL = {https://doi.org/10.1137/24M1650363},
}

@article {MR4630545,
	AUTHOR = {Bertrand, Fleurianne and Carstensen, Carsten and Gr\"a\ss le,
	Benedikt and Tran, Ngoc Tien},
	TITLE = {Stabilization-free {HHO} a posteriori error control},
	JOURNAL = {Numer. Math.},
	FJOURNAL = {Numerische Mathematik},
	VOLUME = {154},
	YEAR = {2023},
	NUMBER = {3-4},
	PAGES = {369--408},
	ISSN = {0029-599X,0945-3245},
	MRCLASS = {65N12 (65N30 65Y20)},
	MRNUMBER = {4630545},
	DOI = {10.1007/s00211-023-01366-8},
	URL = {https://doi.org/10.1007/s00211-023-01366-8},
}

@article {MR4570555,
	AUTHOR = {Diegel, Amanda E. and Sharma, Natasha S.},
	TITLE = {A {$\rm C^0$} interior penalty method for the phase field
	crystal equation},
	JOURNAL = {Numer. Methods Partial Differential Equations},
	FJOURNAL = {Numerical Methods for Partial Differential Equations. An
	International Journal},
	VOLUME = {39},
	YEAR = {2023},
	NUMBER = {3},
	PAGES = {2510--2537},
	ISSN = {0749-159X,1098-2426},
	MRCLASS = {65M60 (65M12)},
	MRNUMBER = {4570555},
	DOI = {10.1002/num.22976},
	URL = {https://doi.org/10.1002/num.22976},
}

@article{MR2048235,
	AUTHOR = {Mozolevski, Igor and S\"uli, Endre},
	TITLE = {A priori error analysis for the {$hp$}-version of the
	discontinuous {G}alerkin finite element method for the
	biharmonic equation},
	JOURNAL = {Comput. Methods Appl. Math.},
	FJOURNAL = {Computational Methods in Applied Mathematics},
	VOLUME = {3},
	YEAR = {2003},
	NUMBER = {4},
	PAGES = {596--607},
	ISSN = {1609-4840,1609-9389},
	MRCLASS = {65N30 (65N15 74G15 74K20)},
	MRNUMBER = {2048235},
	MRREVIEWER = {Denis\ Mercier},
	DOI = {10.2478/cmam-2003-0037},
	URL = {https://doi.org/10.2478/cmam-2003-0037},
}

@article{MR4604440,
	AUTHOR = {Diegel, Amanda E. and Sharma, Natasha S.},
	TITLE = {Unconditional energy stability and solvability for a {C}0
	interior penalty method for a sixth-order equation modeling
	microemulsions},
	JOURNAL = {Int. J. Numer. Anal. Model.},
	FJOURNAL = {International Journal of Numerical Analysis and Modeling},
	VOLUME = {20},
	YEAR = {2023},
	NUMBER = {4},
	PAGES = {459--477},
	ISSN = {1705-5105,2617-8710},
	MRCLASS = {65M60 (82C26 82M10)},
	MRNUMBER = {4604440},
	DOI = {10.4208/ijnam2023-1019},
	URL = {https://doi.org/10.4208/ijnam2023-1019},
}

@article{MR5047454,
	AUTHOR = {Saylor, Giselle and Horv\'ath, Tam\'as L. and Sharma, Natasha
	S.},
	TITLE = {An unconditionally stable hybridizable-embedded discontinuous
	{G}alerkin method for the phase field crystal equation},
	JOURNAL = {Appl. Numer. Math.},
	FJOURNAL = {Applied Numerical Mathematics. An IMACS Journal},
	VOLUME = {225},
	YEAR = {2026},
	PAGES = {241--258},
	ISSN = {0168-9274,1873-5460},
	MRCLASS = {65M60 (35Q35 35Q92 35R99 65N30)},
	MRNUMBER = {5047454},
	DOI = {10.1016/j.apnum.2026.03.003},
	URL = {https://doi.org/10.1016/j.apnum.2026.03.003},
}

@article{MR2361532,
	AUTHOR = {Xia, Yinhua and Xu, Yan and Shu, Chi-Wang},
	TITLE = {Local discontinuous {G}alerkin methods for the
	{C}ahn-{H}illiard type equations},
	JOURNAL = {J. Comput. Phys.},
	FJOURNAL = {Journal of Computational Physics},
	VOLUME = {227},
	YEAR = {2007},
	NUMBER = {1},
	PAGES = {472--491},
	ISSN = {0021-9991,1090-2716},
	MRCLASS = {65M60 (35K55 35Q53)},
	MRNUMBER = {2361532},
	DOI = {10.1016/j.jcp.2007.08.001},
	URL = {https://doi.org/10.1016/j.jcp.2007.08.001}
}

@article{dong2026hp,
	title={$hp$-a posteriori error estimates for hybrid high-order methods applied to biharmonic problems},
	author={Dong, Zhaonan and Ern, Alexandre and Wadhawan, Tanvi},
	journal={arXiv preprint arXiv:2602.06872},
	year={2026}
}

@article{dong2025hperroranalysismixedorderhybrid,
	title={$hp$-error analysis of mixed-order hybrid high-order methods for elliptic problems on simplicial meshes},
	author={Zhaonan Dong and Alexandre Ern},
	journal={Numer. Math.},
	year={2026},
	DOI = {10.1007/s00211-026-01555-1},
	URL = {https://doi.org/10.1007/s00211-026-01555-1},
}

@article {MR2519603,
	AUTHOR = {Wise, S. M. and Wang, C. and Lowengrub, J. S.},
	TITLE = {An energy-stable and convergent finite-difference scheme for
	the phase field crystal equation},
	JOURNAL = {SIAM J. Numer. Anal.},
	FJOURNAL = {SIAM Journal on Numerical Analysis},
	VOLUME = {47},
	YEAR = {2009},
	NUMBER = {3},
	PAGES = {2269--2288},
	ISSN = {0036-1429,1095-7170},
	MRCLASS = {65M06 (35B35 35G25 65M12 65M15 82D25)},
	MRNUMBER = {2519603},
	MRREVIEWER = {Laurent\ E.\ Gosse},
	DOI = {10.1137/080738143},
	URL = {https://doi.org/10.1137/080738143},
}

@article {MR3564350,
	AUTHOR = {Shin, Jaemin and Lee, Hyun Geun and Lee, June-Yub},
	TITLE = {First and second order numerical methods based on a new convex
	splitting for phase-field crystal equation},
	JOURNAL = {J. Comput. Phys.},
	FJOURNAL = {Journal of Computational Physics},
	VOLUME = {327},
	YEAR = {2016},
	PAGES = {519--542},
	ISSN = {0021-9991,1090-2716},
	MRCLASS = {65M06 (82D25)},
	MRNUMBER = {3564350},
	DOI = {10.1016/j.jcp.2016.09.053},
	URL = {https://doi.org/10.1016/j.jcp.2016.09.053},
}

@article {MR2028722,
	AUTHOR = {Barrett, John W. and Langdon, Stephen and N\"urnberg, Robert},
	TITLE = {Finite element approximation of a sixth order nonlinear
	degenerate parabolic equation},
	JOURNAL = {Numer. Math.},
	FJOURNAL = {Numerische Mathematik},
	VOLUME = {96},
	YEAR = {2004},
	NUMBER = {3},
	PAGES = {401--434},
	ISSN = {0029-599X,0945-3245},
	MRCLASS = {65M60 (35K35 35K55 35K65)},
	MRNUMBER = {2028722},
	MRREVIEWER = {Erwin\ Schechter},
	DOI = {10.1007/s00211-003-0479-4},
	URL = {https://doi.org/10.1007/s00211-003-0479-4},
}

@book {MR2667016,
	AUTHOR = {Gazzola, Filippo and Grunau, Hans-Christoph and Sweers, Guido},
	TITLE = {Polyharmonic boundary value problems},
	SERIES = {Lecture Notes in Mathematics},
	VOLUME = {1991},
	NOTE = {Positivity preserving and nonlinear higher order elliptic
	equations in bounded domains},
	PUBLISHER = {Springer-Verlag, Berlin},
	YEAR = {2010},
	PAGES = {xviii+423},
	ISBN = {978-3-642-12244-6},
	MRCLASS = {35-02 (31B30 35A08 35J40 35J61 46E35 58E12)},
	MRNUMBER = {2667016},
	MRREVIEWER = {Rodney\ Josu\'e\ Biezuner},
	DOI = {10.1007/978-3-642-12245-3},
	URL = {https://doi.org/10.1007/978-3-642-12245-3},
}

@book {MR1930132,
	AUTHOR = {Ciarlet, Philippe G.},
	TITLE = {The finite element method for elliptic problems},
	SERIES = {Classics in Applied Mathematics},
	VOLUME = {40},
	PUBLISHER = {Society for Industrial and Applied Mathematics (SIAM),
	Philadelphia, PA},
	YEAR = {2002},
	PAGES = {xxviii+530},
	ISBN = {0-89871-514-8},
	MRCLASS = {65N30 (65-02 65N15 74G15 74S05)},
	MRNUMBER = {1930132},
	DOI = {10.1137/1.9780898719208},
	URL = {https://doi.org/10.1137/1.9780898719208},
}

@article {MR1393904,
	AUTHOR = {D\"orfler, Willy},
	TITLE = {A convergent adaptive algorithm for {P}oisson's equation},
	JOURNAL = {SIAM J. Numer. Anal.},
	FJOURNAL = {SIAM Journal on Numerical Analysis},
	VOLUME = {33},
	YEAR = {1996},
	NUMBER = {3},
	PAGES = {1106--1124},
	ISSN = {0036-1429},
	MRCLASS = {65N50 (65N55)},
	MRNUMBER = {1393904},
	MRREVIEWER = {S.\ F.\ McCormick},
	DOI = {10.1137/0733054},
	URL = {https://doi.org/10.1137/0733054},
}

@article {MR3003062,
	AUTHOR = {Gomez, Hector and Nogueira, Xes\'us},
	TITLE = {An unconditionally energy-stable method for the phase field
	crystal equation},
	JOURNAL = {Comput. Methods Appl. Mech. Engrg.},
	FJOURNAL = {Computer Methods in Applied Mechanics and Engineering},
	VOLUME = {249/252},
	YEAR = {2012},
	PAGES = {52--61},
	ISSN = {0045-7825,1879-2138},
	MRCLASS = {65M60 (74N15)},
	MRNUMBER = {3003062},
	DOI = {10.1016/j.cma.2012.03.002},
	URL = {https://doi.org/10.1016/j.cma.2012.03.002},
}

@article {MR4842145,
	AUTHOR = {Diegel, Amanda E. and Bond, Daniel and Sharma, Natasha S.},
	TITLE = {Stability and error analysis for a {$\rm C^0$} interior
	penalty method for the modified phase field crystal equation},
	JOURNAL = {Matematica},
	FJOURNAL = {La Matematica. Official Journal of the Association for Women
	in Mathematics},
	VOLUME = {3},
	YEAR = {2024},
	NUMBER = {4},
	PAGES = {1426--1450},
	ISSN = {2730-9657},
	MRCLASS = {65M12 (35K35 65M60)},
	MRNUMBER = {4842145},
	DOI = {10.1007/s44007-024-00133-x},
	URL = {https://doi.org/10.1007/s44007-024-00133-x},
}

@article {MR3327707,
	AUTHOR = {Liu, Changchun},
	TITLE = {A sixth-order thin film equation in two space dimensions},
	JOURNAL = {Adv. Differential Equations},
	FJOURNAL = {Advances in Differential Equations},
	VOLUME = {20},
	YEAR = {2015},
	NUMBER = {5-6},
	PAGES = {557--580},
	ISSN = {1079-9389},
	MRCLASS = {35K55 (35D30 35K35 35K65 76A20)},
	MRNUMBER = {3327707},
	MRREVIEWER = {Daniel\ Matthes},
	URL = {http://projecteuclid.org/euclid.ade/1427744016},
}

@article {MR2807583,
	AUTHOR = {Verhoosel, Clemens V. and Scott, Michael A. and Hughes, Thomas
	J. R. and de Borst, Ren\'e},
	TITLE = {An isogeometric analysis approach to gradient damage models},
	JOURNAL = {Internat. J. Numer. Methods Engrg.},
	FJOURNAL = {International Journal for Numerical Methods in Engineering},
	VOLUME = {86},
	YEAR = {2011},
	NUMBER = {1},
	PAGES = {115--134},
	ISSN = {0029-5981,1097-0207},
	MRCLASS = {65N30 (74R10 74S05)},
	MRNUMBER = {2807583},
	DOI = {10.1002/nme.3150},
	URL = {https://doi.org/10.1002/nme.3150},
}

@article{liu2007general,
	title={A general sixth order geometric partial differential equation and its application in surface modeling},
	author={Liu, Dan and Xu, Guoliang},
	journal={Journal of Information and Computational Science},
	volume={4},
	number={1},
	pages={129--140},
	year={2007},
	publisher={Citeseer}
}

@article {MR1726480,
	AUTHOR = {Verf\"urth, R\"udiger},
	TITLE = {Error estimates for some quasi-interpolation operators},
	JOURNAL = {M2AN Math. Model. Numer. Anal.},
	FJOURNAL = {M2AN. Mathematical Modelling and Numerical Analysis},
	VOLUME = {33},
	YEAR = {1999},
	NUMBER = {4},
	PAGES = {695--713},
	ISSN = {0764-583X,1290-3841},
	MRCLASS = {65N30 (41A35 65D05 65N15)},
	MRNUMBER = {1726480},
	MRREVIEWER = {Jacques\ Rappaz},
	DOI = {10.1051/m2an:1999158},
	URL = {https://doi.org/10.1051/m2an:1999158},
}

@article {MR4974657,
	AUTHOR = {Li, Hengguang and Yin, Peimeng},
	TITLE = {A {$C^0$} finite element algorithm for the sixth order problem
	with simply supported boundary conditions},
	JOURNAL = {Numer. Methods Partial Differential Equations},
	FJOURNAL = {Numerical Methods for Partial Differential Equations. An
	International Journal},
	VOLUME = {41},
	YEAR = {2025},
	NUMBER = {6},
	PAGES = {Paper No. e70048, 22},
	ISSN = {0749-159X,1098-2426},
	MRCLASS = {65N12 (35J40 65N30)},
	MRNUMBER = {4974657},
	DOI = {10.1002/num.70048},
	URL = {https://doi.org/10.1002/num.70048},
}

@article {MR3283758,
	AUTHOR = {Di Pietro, Daniele A. and Ern, Alexandre},
	TITLE = {A hybrid high-order locking-free method for linear elasticity
	on general meshes},
	JOURNAL = {Comput. Methods Appl. Mech. Engrg.},
	FJOURNAL = {Computer Methods in Applied Mechanics and Engineering},
	VOLUME = {283},
	YEAR = {2015},
	PAGES = {1--21},
	ISSN = {0045-7825,1879-2138},
	MRCLASS = {65N30 (74B05 74G15)},
	MRNUMBER = {3283758},
	MRREVIEWER = {Alexandre\ L.\ Madureira},
	DOI = {10.1016/j.cma.2014.09.009},
	URL = {https://doi.org/10.1016/j.cma.2014.09.009},
}

@article {Walk14,
	AUTHOR = {Walkington, Noel J.},
	TITLE = {{A {$\rm C^1$} tetrahedral finite element without edge degrees of
	freedom}},
	JOURNAL = {SIAM J. Numer. Anal.},
	FJOURNAL = {SIAM Journal on Numerical Analysis},
	VOLUME = {52},
	YEAR = {2014},
	NUMBER = {1},
	PAGES = {330--342},
	ISSN = {0036-1429,1095-7170},
	MRCLASS = {65N30},
	MRNUMBER = {3163246},
	MRREVIEWER = {Gerard\ Awanou},
	optDOI = {10.1137/130912013},
	optURL = {https://doi.org/10.1137/130912013},
}

@article {MR4044442,
	AUTHOR = {Ern, Alexandre and Vohral\'ik, Martin},
	TITLE = {Stable broken {$H^1$} and {$H({\rm div})$} polynomial
	extensions for polynomial-degree-robust potential and flux
	reconstruction in three space dimensions},
	JOURNAL = {Math. Comp.},
	FJOURNAL = {Mathematics of Computation},
	VOLUME = {89},
	YEAR = {2020},
	NUMBER = {322},
	PAGES = {551--594},
	ISSN = {0025-5718,1088-6842},
	MRCLASS = {65N15 (65N30 76M10)},
	MRNUMBER = {4044442},
	MRREVIEWER = {Riccardo\ Sacco},
	DOI = {10.1090/mcom/3482},
	URL = {https://doi.org/10.1090/mcom/3482},
}

@article{dassi2022virtual,
	title={A virtual element method on polyhedral meshes for the sixth-order elliptic problem},
	author={Dassi, Franco and Mora, David and Reales, Carlos and Vel{\`a}squez, Iv{\`a}n},
	journal={arXiv preprint arXiv:2211.07953},
	year={2022}
}

@article {MR4496380,
	AUTHOR = {Causil, Jos\'e{} and Reales, Carlos and Vel\'asquez, Iv\'an},
	TITLE = {A {$C^1$}-{$C^0$} virtual element discretization for a
	sixth-order elliptic equation},
	JOURNAL = {Calcolo},
	FJOURNAL = {Calcolo. A Quarterly on Numerical Analysis and Theory of
	Computation},
	VOLUME = {59},
	YEAR = {2022},
	NUMBER = {4},
	PAGES = {Paper No. 39, 27},
	ISSN = {0008-0624,1126-5434},
	MRCLASS = {65N30 (35J58 65N15)},
	MRNUMBER = {4496380},
	MRREVIEWER = {Yadong\ Zhang},
	DOI = {10.1007/s10092-022-00482-5},
	URL = {https://doi.org/10.1007/s10092-022-00482-5},
}

@article {MR4439881,
	AUTHOR = {Botti, Lorenzo and Massa, Francesco Carlo},
	TITLE = {H{HO} methods for the incompressible {N}avier-{S}tokes and the
	incompressible {E}uler equations},
	JOURNAL = {J. Sci. Comput.},
	FJOURNAL = {Journal of Scientific Computing},
	VOLUME = {92},
	YEAR = {2022},
	NUMBER = {1},
	PAGES = {Paper No. 28, 38},
	ISSN = {0885-7474,1573-7691},
	MRCLASS = {65M60 (65M12 76B47 76D05)},
	MRNUMBER = {4439881},
	DOI = {10.1007/s10915-022-01864-1},
	URL = {https://doi.org/10.1007/s10915-022-01864-1},
}

@article {MR3767823,
	AUTHOR = {Di Pietro, Daniele A. and Krell, Stella},
	TITLE = {A hybrid high-order method for the steady incompressible
	{N}avier-{S}tokes problem},
	JOURNAL = {J. Sci. Comput.},
	FJOURNAL = {Journal of Scientific Computing},
	VOLUME = {74},
	YEAR = {2018},
	NUMBER = {3},
	PAGES = {1677--1705},
	ISSN = {0885-7474,1573-7691},
	MRCLASS = {65N30 (35Q30 65N12 76D05 76M10)},
	MRNUMBER = {3767823},
	MRREVIEWER = {Long\ An\ Ying},
	DOI = {10.1007/s10915-017-0512-x},
	URL = {https://doi.org/10.1007/s10915-017-0512-x},
}

@article {MR3784906,
	AUTHOR = {Chave, Florent and Di Pietro, Daniele A. and Formaggia, Luca},
	TITLE = {A hybrid high-order method for {D}arcy flows in fractured
	porous media},
	JOURNAL = {SIAM J. Sci. Comput.},
	FJOURNAL = {SIAM Journal on Scientific Computing},
	VOLUME = {40},
	YEAR = {2018},
	NUMBER = {2},
	PAGES = {A1063--A1094},
	ISSN = {1064-8275,1095-7197},
	MRCLASS = {65N08 (65N30 76S05)},
	MRNUMBER = {3784906},
	MRREVIEWER = {Roberta\ De Luca},
	DOI = {10.1137/17M1119500},
	URL = {https://doi.org/10.1137/17M1119500},
}

@article {MR4982456,
	AUTHOR = {P\'erez, Luis and Reales, Carlos and Silgado, Alberth and
	Vel\'asquez, Iv\'an},
	TITLE = {{$C^1-C^0$} conforming virtual element approximations for
	sixth-order problems},
	JOURNAL = {J. Comput. Phys.},
	FJOURNAL = {Journal of Computational Physics},
	VOLUME = {545},
	YEAR = {2026},
	PAGES = {Paper No. 114483, 17},
	ISSN = {0021-9991,1090-2716},
	MRCLASS = {65M60 (35J58 65M15)},
	MRNUMBER = {4982456},
	DOI = {10.1016/j.jcp.2025.114483},
	URL = {https://doi.org/10.1016/j.jcp.2025.114483},
}

@article {MR3834430,
	AUTHOR = {Bonaldi, Francesco and Di Pietro, Daniele A. and Geymonat,
	Giuseppe and Krasucki, Fran\c coise},
	TITLE = {A hybrid high-order method for {K}irchhoff-{L}ove plate
	bending problems},
	JOURNAL = {ESAIM Math. Model. Numer. Anal.},
	FJOURNAL = {ESAIM. Mathematical Modelling and Numerical Analysis},
	VOLUME = {52},
	YEAR = {2018},
	NUMBER = {2},
	PAGES = {393--421},
	ISSN = {2822-7840,2804-7214},
	MRCLASS = {65N30 (65N12 74K20)},
	MRNUMBER = {3834430},
	MRREVIEWER = {Gerhard\ Starke},
	DOI = {10.1051/m2an/2017065},
	URL = {https://doi.org/10.1051/m2an/2017065},
}

@article {MR3504993,
	AUTHOR = {Boffi, Daniele and Botti, Michele and Di Pietro, Daniele A.},
	TITLE = {A nonconforming high-order method for the {B}iot problem on
	general meshes},
	JOURNAL = {SIAM J. Sci. Comput.},
	FJOURNAL = {SIAM Journal on Scientific Computing},
	VOLUME = {38},
	YEAR = {2016},
	NUMBER = {3},
	PAGES = {A1508--A1537},
	ISSN = {1064-8275,1095-7197},
	MRCLASS = {65N30 (65N08 76M10 76S05)},
	MRNUMBER = {3504993},
	MRREVIEWER = {Gerhard\ Starke},
	DOI = {10.1137/15M1025505},
	URL = {https://doi.org/10.1137/15M1025505},
}

@article {MR4699573,
	AUTHOR = {Dong, Zhaonan and Ern, Alexandre},
	TITLE = {{$C^0$}-hybrid high-order methods for biharmonic problems},
	JOURNAL = {IMA J. Numer. Anal.},
	FJOURNAL = {IMA Journal of Numerical Analysis},
	VOLUME = {44},
	YEAR = {2024},
	NUMBER = {1},
	PAGES = {24--57},
	ISSN = {0272-4979,1464-3642},
	MRCLASS = {65N30 (65N12)},
	MRNUMBER = {4699573},
	DOI = {10.1093/imanum/drad003},
	URL = {https://doi.org/10.1093/imanum/drad003},
}

@article {MR4779761,
	AUTHOR = {Antonietti, P. F. and Matalon, P. and Verani, M.},
	TITLE = {Iterative solution to the biharmonic equation in mixed form
	discretized by the hybrid high-order method},
	JOURNAL = {Comput. Math. Appl.},
	FJOURNAL = {Computers \& Mathematics with Applications. An International
	Journal},
	VOLUME = {171},
	YEAR = {2024},
	PAGES = {154--163},
	ISSN = {0898-1221,1873-7668},
	MRCLASS = {65N30 (35J58 35J91)},
	MRNUMBER = {4779761},
	DOI = {10.1016/j.camwa.2024.07.018},
	URL = {https://doi.org/10.1016/j.camwa.2024.07.018},
}

@article {MR3259024,
	AUTHOR = {Di Pietro, Daniele A. and Ern, Alexandre and Lemaire, Simon},
	TITLE = {An arbitrary-order and compact-stencil discretization of
	diffusion on general meshes based on local reconstruction
	operators},
	JOURNAL = {Comput. Methods Appl. Math.},
	FJOURNAL = {Computational Methods in Applied Mathematics},
	VOLUME = {14},
	YEAR = {2014},
	NUMBER = {4},
	PAGES = {461--472},
	ISSN = {1609-4840,1609-9389},
	MRCLASS = {65N08 (65N12)},
	MRNUMBER = {3259024},
	MRREVIEWER = {Andrei\ I.\ Tolstykh},
	DOI = {10.1515/cmam-2014-0018},
	URL = {https://doi.org/10.1515/cmam-2014-0018},
}

@article {MR4485999,
	AUTHOR = {Dong, Zhaonan and Ern, Alexandre},
	TITLE = {Hybrid high-order and weak {G}alerkin methods for the
	biharmonic problem},
	JOURNAL = {SIAM J. Numer. Anal.},
	FJOURNAL = {SIAM Journal on Numerical Analysis},
	VOLUME = {60},
	YEAR = {2022},
	NUMBER = {5},
	PAGES = {2626--2656},
	ISSN = {0036-1429,1095-7170},
	MRCLASS = {65N15 (65N30 74K20)},
	MRNUMBER = {4485999},
	DOI = {10.1137/21M1408555},
	URL = {https://doi.org/10.1137/21M1408555},
}

@article {MR4355427,
	AUTHOR = {Dong, Zhaonan and Ern, Alexandre},
	TITLE = {Hybrid high-order method for singularly perturbed fourth-order
	problems on curved domains},
	JOURNAL = {ESAIM Math. Model. Numer. Anal.},
	FJOURNAL = {ESAIM. Mathematical Modelling and Numerical Analysis},
	VOLUME = {55},
	YEAR = {2021},
	NUMBER = {6},
	PAGES = {3091--3114},
	ISSN = {2822-7840,2804-7214},
	MRCLASS = {65N30 (65N15 74K20)},
	MRNUMBER = {4355427},
	DOI = {10.1051/m2an/2021081},
	URL = {https://doi.org/10.1051/m2an/2021081},
}

@book {MR4230986,
    AUTHOR = {Di Pietro, Daniele Antonio and Droniou, J\'er\^ome},
     TITLE = {The hybrid high-order method for polytopal meshes},
    SERIES = {Modeling, Simulation and Applications},
    VOLUME = {19},
      NOTE = {Design, analysis, and applications},
 PUBLISHER = {Springer, Cham},
      YEAR = {2020},
     PAGES = {xxxi+525},
      ISBN = {978-3-030-37202-6; 978-3-030-37203-3},
   MRCLASS = {65-02 (74-01 76-01 78Mxx)},
  MRNUMBER = {4230986},
       DOI = {10.1007/978-3-030-37203-3},
       URL = {https://doi.org/10.1007/978-3-030-37203-3},
}

@book {MR2882148,
	AUTHOR = {Di Pietro, Daniele Antonio and Ern, Alexandre},
	TITLE = {Mathematical aspects of discontinuous {G}alerkin methods},
	SERIES = {Math\'ematiques \& Applications (Berlin) [Mathematics \&
	Applications]},
	VOLUME = {69},
	PUBLISHER = {Springer, Heidelberg},
	YEAR = {2012},
	PAGES = {xviii+384},
	ISBN = {978-3-642-22979-4},
	MRCLASS = {65-02 (35A35 35F15 35J25 35Q35 65M60 65N30)},
	MRNUMBER = {2882148},
	MRREVIEWER = {R\'emi\ Vaillancourt},
	DOI = {10.1007/978-3-642-22980-0},
	URL = {https://doi.org/10.1007/978-3-642-22980-0},
}

@article {MR1974504,
	AUTHOR = {Brenner, Susanne C.},
	TITLE = {Poincar\'e-{F}riedrichs inequalities for piecewise {$H^1$}
	functions},
	JOURNAL = {SIAM J. Numer. Anal.},
	FJOURNAL = {SIAM Journal on Numerical Analysis},
	VOLUME = {41},
	YEAR = {2003},
	NUMBER = {1},
	PAGES = {306--324},
	ISSN = {0036-1429,1095-7170},
	MRCLASS = {65N30 (46E35)},
	MRNUMBER = {1974504},
	DOI = {10.1137/S0036142902401311},
	URL = {https://doi.org/10.1137/S0036142902401311},
}
\end{document}